\documentclass[11pt]{amsart}

\usepackage{hyperref}
\usepackage[latin1]{inputenc}

\usepackage{mfdmetrics}

\usepackage[margin=1in]{geometry}

\usepackage[all]{xy}

\numberwithin{thm}{section}
\numberwithin{equation}{section}

\author{Brian Clarke}
\date{April 1, 2009}
\title{The Completion of the Manifold of Riemannian Metrics}
\thanks{This research was supported by the International Max Planck
  Research School \emph{Mathematics in the Sciences} and the DFG
  Research Training Group \emph{Analysis, Geometry and their
    Interaction with the Natural Sciences}.}

\begin{document}

\begin{abstract}
  We give a description of the completion of the manifold of all
  smooth Riemannian metrics on a fixed smooth, closed,
  finite-dimensional, orientable manifold with respect to a natural
  metric called the $L^2$ metric.  The primary motivation for studying
  this problem comes from Teichmüller theory, where similar
  considerations lead to a completion of the well-known Weil-Petersson
  metric.  We give an application of the main theorem to the
  completions of Teichmüller space with respect to a class of metrics
  that generalize the Weil-Petersson metric.
\end{abstract}

\maketitle

\setcounter{tocdepth}{2}
\tableofcontents{}

\section{Introduction}\label{sec:introduction}

This is the second in a pair of papers studying the metric geometry of
the Fréchet manifold $\M$ of all smooth Riemannian metrics on a
smooth, closed, finite-dimensional, orientable manifold $M$.  The
manifold $\M$ carries a natural weak Riemannian metric called the
$L^2$ metric, defined in the next section.  In the first paper
\cite{clarke:_metric_geomet_of_manif_of_rieman_metric}, we showed that
the $L^2$ metric induces a metric space structure on $\M$ (a
nontrivial statement for weak Riemannian metrics; see Section
\ref{sec:m-cdot-cdot}).  In this paper, we will give the following
description of the metric completion $\overline{\M}$ of $\M$ with
respect to the $L^2$ metric:

\begin{thm*}
  Let $\Mf$ denote the space of all measurable, symmetric,
  finite-volume $(0,2)$-tensor fields on $M$ that induce a positive
  semidefinite scalar product on each tangent space of $M$.  For $g_0,
  g_1 \in \Mf$ and $x \in M$, we say $g_0 \sim g_1$ if the following
  statement holds almost surely:
  \begin{equation*}
    g_0(x) \neq g_1(x) \Longleftrightarrow g_0(x), g_1(x)\ \textnormal{are not positive definite.}
  \end{equation*}
  Then there is a natural identification $\overline{\M} \cong \Mf /
  {\sim}$.
\end{thm*}

Note that while $\M$ is a space of smooth objects, we must add in
points corresponding to extremely degenerate objects in order to
complete it.  This is a reflection of the fact that the $L^2$ metric
is a \emph{weak} rather than a \emph{strong} Riemannian metric.  That
is, the topology it induces on the tangent spaces---the $L^2$
topology---is weaker than the $C^\infty$ topology coming from the
manifold structure.  In essence, the incompleteness of the tangent
spaces then carries over to the manifold itself.

The manifold of Riemannian metrics---along with geometric structures
on it---has been considered in several contexts.  It originally arose
in general relativity \cite{dewitt67:_quant_theor_of_gravit}, and was
subsequently studied by mathematicians
\cite{ebin70:_manif_of_rieman_metric,freed89:_basic_geomet_of_manif_of,gil-medrano91:_rieman_manif_of_all_rieman_metric}.
In particular, the Riemannian geometry of the $L^2$ metric is well
understood---its curvature, geodesics, and Jacobi fields are
explicitly known.  The metric geometry of the $L^2$ metric, though,
was not as clear up to this point, and this paper seeks to illuminate
one aspect of that.

Our motivation for studying the completion of $\M$---besides the
intrinsic interest to Riemannian geometers of studying this important
deformation space---came largely from Teichmüller theory.  If the base
manifold $M$ is a closed Riemann surface of genus larger than one, the
work of Fischer and Tromba \cite{tromba-teichmueller} gives an
identification of the Teichmüller space of $M$ with $\Mhyp / \DO$,
where $\Mhyp \subset \M$ is the submanifold of hyperbolic metrics and
$\DO$ is the group of diffeomorphisms of $M$ that are homotopic to the
identity, acting on $\Mhyp$ by pull-back.  The $L^2$ metric restricted
to $\Mhyp$ descends to the Weil-Petersson metric on Teichmüller space,
and its completion consists of adding in points corresponding to
certain cusped hyperbolic metrics.  The action of the mapping class
group on Teichmüller space extends to this completion, and the
quotient is homeomorphic to the Deligne-Mumford compactification of
the moduli space of $M$.  In Section \ref{cha:appl-teichm-space},
inspired by
\cite{hj-riemannian,habermann98:_metric_rieman_surfac_and_geomet}, we
generalize the Weil-Petersson metric and use the above theorem to
formulate a condition on the completion of these generalized
Weil-Petersson metrics.

The paper is organized as follows:

In Section \ref{sec:preliminaries}, we recall the necessary background
on the manifold of metrics, the $L^2$ metric, and completions of
metric spaces.  We also review some nonstandard geometric notions and
fix notation and conventions for the paper.

In Section \ref{sec:compl-an-amen}, we complete what we call amenable
subsets of $\M$.  They are defined in such a way that we can show that
the completion of these subsets with respect to the $L^2$
\emph{metric} on the subset is the same as with respect to the $L^2$
\emph{norm} on $\M$ (this will be made precise below).  This
completion is the first step in a bootstrapping process of
understanding the full completion.

In Section \ref{cha:almost-everywh-conv}, we introduce a notion called
\emph{$\omega$-convergence} for Cauchy sequences in $\M$ that
describes how a Cauchy sequence converges to an element of $\Mf /
{\sim}$.  It is a kind of pointwise a.e.-convergence---except on a
subset where the sequence degenerates in a certain way, where no
convergence can be demanded of Cauchy sequences.  We then use the
results of Section \ref{sec:compl-an-amen} to show that this
convergence notion gives an injective map from the completion of $\M$
into $\Mf / {\sim}$.  To do this, we need to show two things.  First,
we prove that every Cauchy sequence in $\M$ has an $\omega$-convergent
subsequence.  Second, we show in two theorems that two
Cauchy sequences are equivalent (in the sense of the completion of a
metric space) if and only if they $\omega$-subconverge to the same
limit.  This section comprises the most technically challenging
portion of the paper.  It also contains the following result, which is
in our eyes one of the most unexpected and striking of the paper:

\begin{prop*}
  Suppose that $g_0, g_1 \in \M$, and let $E := \carr (g_1 - g_0) = \{
  x \in M \mid g_0(x) \neq g_1(x) \}$.  Let $d$ be the Riemannian
  distance function of the $L^2$ metric $(\cdot, \cdot)$.  Then there
  exists a constant $C(n)$ depending only on $n := \dim M$ such that
  \begin{equation*}
    d(g_0, g_1) \leq C(n) \left( \sqrt{\Vol(E, g_0)} +
      \sqrt{\Vol(E,g_1)} \right).
  \end{equation*}
  In particular, we have
  \begin{equation*}
    \diam \left( \{ \tilde{g} \in \M \mid \Vol(M, \tilde{g}) \leq
      \delta \} \right) \leq 2 C(n) \sqrt{\delta}.
  \end{equation*}
\end{prop*}

The surprising thing about this proposition is that it says that two
metrics can vary wildly, but as long as they do so on a set that has
small volume with respect to each, they are close together in the
$L^2$ metric.

In Section \ref{chap:sing-metrics}, we complete the proof of the main
result.  This is done by continuing the bootstrapping process begun in
Section \ref{sec:compl-an-amen} to see that the map defined in Section
\ref{cha:almost-everywh-conv} is in fact surjective.  That is, we
prove in stages that there are Cauchy sequences $\omega$-converging to
elements in ever larger subsets of $\Mf / {\sim}$.

In Section \ref{cha:appl-teichm-space}, we give the application to the
geometry of Teichmüller space that was mentioned above.

Several different types of sequences and convergence notions enter
into this work.  For the reader's convenience, we have included an
appendix which summarizes the relationships between these different
concepts.

\subsection*{Acknowledgements}

The results of this paper formed a portion of my Ph.D.~thesis
(\cite{clarked):_compl_of_manif_of_rieman}, where the reader may find
the facts here presented in significantly greater detail), written at
the Max Planck Institute for Mathematics in the Sciences in Leipzig
and submitted to the University of Leipzig.  First and foremost, I
would like to thank my advisor Jürgen Jost for his many years of
patient assistance.  I am also grateful to Guy Buss, Christoph Sachse,
and Nadine Große for many fruitful discussions and their careful
proofreading.  Portions of the research were also carried out while I
was visiting Stanford University, and I would like to thank Rafe
Mazzeo and Larry Guth for helpful conversations.  Special thanks to
Yurii Savchuk for discussions related to Section
\ref{sec:pointwise-estimates}.

\section{Preliminaries}\label{sec:preliminaries}

\subsection{The Manifold of Metrics}\label{sec:manif-riem-metr}

For the entirety of the paper, let $M$ denote a fixed closed,
orientable, $n$-dimensional $C^\infty$ manifold.

The basic facts about the manifold of Riemannian metrics given in this
section can be found in \cite[\S
2.5]{clarked):_compl_of_manif_of_rieman}

We denote by $S^2 T^*M$ the vector bundle of symmetric $(0,2)$ tensors
over $M$, and by $\s$ the Fréchet space of $C^\infty$ sections of $S^2
T^* M$.  The space $\M$ of Riemannian metrics on $M$ is an open subset
of $\s$, and hence it is trivially a Fréchet manifold, with tangent
space at each point canonically identified with $\s$.  (For a detailed
treatment of Fréchet manifolds, see, for example,
\cite{hamilton82:_inver_funct_theor_of_nash_and_moser}.  For a more
thorough treatment of the differential topology and geometry of $\M$,
see \cite{ebin70:_manif_of_rieman_metric}.)

\subsubsection{The $L^2$ Metric}\label{sec:l2-metric}

$\M$ carries a natural Riemannian metric $(\cdot, \cdot)$, called the
\emph{$L^2$ metric}, induced by integration from the natural scalar product
on $S^2 T^* M$.  Given any $g \in \M$ and $h, k \in \s \cong T_g \M$,
we define
\begin{equation*}
  (h, k)_g := \integral{M}{}{\tr_g(hk)}{d \mu_g}.
\end{equation*}
Here, $\tr_g(hk)$ is given in local coordinates by $\tr(g^{-1} h
g^{-1} k) = g^{ij} h_{il} g^{lm} k_{jm}$, and $\mu_g$ denotes the
volume form induced by $g$.

Throughout the paper, we use the notation $d$ for the distance
function induced from $(\cdot, \cdot)$ by taking the infimum of the
lengths of paths between two given points.

The $L^2$ metric is a weak Riemannian metric, which means that its
induced topology on the tangent spaces of $\M$ is weaker than the
manifold topology.  This leads to some phenomena that are unfamiliar
from the world of finite-dimensional Riemannian geometry, or even
strong Riemannian metrics on Hilbert manifolds.  For instance, the
$L^2$ metric does not \emph{a priori} induce a metric space structure
on $\M$.  In \cite{clarke:_metric_geomet_of_manif_of_rieman_metric},
we nevertheless showed directly that $(\M, d)$ is a metric space, but
other strange phenomena occur---for instance, the metric space
topology of $(\M, d)$ is weaker than the manifold topology of $\M$.
Indeed, in Lemma \ref{lem:9} we will see the following: When
considered as a subset of its completion, $\M$ contains no open
$d$-ball around \emph{any} point!  For more information on weak
Riemannian metrics, see \cite[\S
2.4]{clarked):_compl_of_manif_of_rieman}, \cite[\S
3]{clarke:_metric_geomet_of_manif_of_rieman_metric}

The basic Riemannian geometry of $(\M, (\cdot, \cdot))$ is relatively
well understood.  For example, it is known that the sectional
curvature of $\M$ is nonpositive
\cite[Cor.~1.17]{freed89:_basic_geomet_of_manif_of}, and the geodesics
of $\M$ are known explicitly
\cite[Thm.~2.3]{freed89:_basic_geomet_of_manif_of},
\cite[Thm.~3.2]{gil-medrano91:_rieman_manif_of_all_rieman_metric}.

We will also consider related structures restricted to a point $x \in
M$.  Let $\satx := S^2 T^*_x M$ denote the vector space of symmetric
$(0,2)$-tensors at $x$, and let $\Matx \subset \satx$ denote the open
subset of tensors inducing a positive definite scalar product on $T_x
M$.  Then $\Matx$ is an open submanifold of $\satx$, and its tangent
space at each point is canonically identified with $\satx$. For each
$g \in \Matx$, we define a scalar product $\langle \cdot , \cdot
\rangle_g$ on $T_g \Matx \cong \satx$ by setting, for all $h, k \in
\satx$,
\begin{equation*}
  \langle h, k \rangle_g := \tr_g(hk).
\end{equation*}
Then $\langle \cdot , \cdot \rangle$ defines a Riemannian metric on
the finite-dimensional manifold $\Matx$.

For each $g \in \M$, we denote the $L^2$ norm induced by $g$ on $\s$
with $\normdot_g$, that is, $\norm{h}_g := \sqrt{(h,h)_g}$.  For
any $g_0, g_1 \in \M$, the norms $\normdot_{g_0}$ and
$\normdot_{g_1}$ are equivalent \cite[\S
IX.2]{palais65:_semin_atiyah_singer_index_theor}.  (As Ebin pointed
out \cite[\S 4]{ebin70:_manif_of_rieman_metric}, this statement even
holds if $g_0$ and $g_1$ are only required to be continuous.)

\subsubsection{Geodesics}\label{sec:geodesics}

As noted above, the geodesic equation of $\M$ can be solved
explicitly.  We will not need the full expression for an arbitrary
geodesic for our purposes, but rather only for very special geodesics.

We denote by $\pos \subset C^\infty(M)$ the group of strictly positive
smooth functions on $M$.  This is a Fréchet Lie group that acts on
$\M$ by pointwise multiplication.  For any $g_0 \in \M$, the next
proposition gives the geodesics of the orbit $\pos \cdot g_0$.

\begin{prop}[{\cite[Prop.~2.1]{freed89:_basic_geomet_of_manif_of}}]\label{prop:5}
  The geodesic in $\M$ starting at $g_0 \in \M$ with initial tangent
  vector $\rho g_0$, where $\rho \in C^\infty(M)$, is given by
  \begin{equation}\label{eq:3}
    g_t = \left( 1 + n \frac{t}{4} \rho \right)^{4/n} g_0.
  \end{equation}
  In particular, $\pos \cdot g_0$ is a totally geodesic submanifold,
  and the exponential mapping $\exp_{g_0}$ is a diffeomorphism from
  the open set $U \cdot g_0 \subset T_{g_0} (\pos \cdot g_0)$---where
  $U$ is the set of functions satisfying $\rho > -4/n$---onto $\pos
  \cdot g_0$.
\end{prop}

\subsubsection{Metric Space Structures on $\M$}\label{sec:m-cdot-cdot}

In \cite{clarke:_metric_geomet_of_manif_of_rieman_metric}, we proved
the following theorem:

\begin{thm}\label{thm:6}
  $(\M, d)$, where $d$ is the distance function of the $L^2$ metric,
  is a metric space.
\end{thm}

As mentioned above, the fact that the $L^2$ metric is a weak
Riemannian metric means that general theorems imply only that $d$ is a
pseudometric.  In fact, there are examples
\cite{michor05:_vanis_geodes_distan_spaces_of,michor06:_rieman_geomet_spaces_of_plane_curves}
of weak Riemannian metrics where the induced distance between any two
points is always zero!

To prove Theorem \ref{thm:6}, we defined a function on $\M \times \M$
that was manifestly a metric (in the sense of metric spaces) and
showed that this metric bounded $d$ from below in some way.  We also
showed that the function $\M \rightarrow \R$ sending a metric to the
square root of its total volume is Lipschitz with respect to $d$.
These results will play a role in what is to come, so we review them
here, along with the relevant definitions.

First, we have the lemma on Lipschitz continuity of the square root of
the volume function.

\begin{lem}[{\cite[Lemma~17]{clarke:_metric_geomet_of_manif_of_rieman_metric}}]\label{lem:13}
  Let $g_0, g_1 \in \M$.  Then for any measurable subset $Y \subseteq
  M$,
  \begin{equation*}
    \left| \sqrt{\Vol(Y,g_1)} - \sqrt{\Vol(Y,g_0)} \right| \leq \frac{\sqrt{n}}{4} d(g_0,g_1).
  \end{equation*}
\end{lem}

Next, we define the metric on $\M$ that was mentioned above and state
the lower bound it provides on $d$.

\begin{dfn}\label{dfn:15}
  For each $x \in M$, consider $\M_x = \{ \tilde{g} \in \satx \mid
  \tilde{g} > 0 \}$ (cf.~Section \ref{sec:l2-metric}).  For any fixed
  $g \in \M$, define a Riemannian metric $\langle \cdot , \cdot
  \rangle^0$ on $\M_x$ by
  \begin{equation*}
    \langle h , k \rangle^0_{\tilde{g}} = \tr_{\tilde{g}} (h k) \det
    g(x)^{-1} \tilde{g} \qquad \forall h, k \in T_{\tilde{g}} \M_x \cong
    \satx.
  \end{equation*}
  We denote by $\theta^g_x$ the Riemannian distance function of
  $\langle \cdot , \cdot \rangle^0$.
\end{dfn}

Note that $\theta^g_x$ is automatically positive definite, since it is
the distance function of a Riemannian metric on a finite-dimensional
manifold.  By integrating it in $x$, we can pass from a metric on
$\M_x$ to a function on $\M \times \M$ as follows:

\begin{lem}[{\cite[Lemma~20, 21]{clarke:_metric_geomet_of_manif_of_rieman_metric}}]\label{dfn:16}
  For any measurable $Y \subseteq M$, define a function $\Theta_Y : \M
  \times \M \rightarrow \R$ by
  \begin{equation*}
    \Theta_Y(g_0, g_1) = \integral{Y}{}{\theta^g_x(g_0(x), g_1(x))}{d \mu_g(x)}.
  \end{equation*}

  Then $\Theta_Y$ does not depend upon the choice of metric $g$ used
  to define $\theta^g_x$.  Furthermore, $\Theta_Y$ is a pseudometric
  on $\M$, and $\Theta_M$ is a metric.  Finally, if $Y_0 \subseteq
  Y_1$, then $\Theta_{Y_0}(g_0, g_1) \leq \Theta_{Y_1}(g_0, g_1)$ for
  all $g_0, g_1 \in \M$.
\end{lem}

The lower bound on $d$ is the following:

\begin{prop}[{\cite[Prop.~22]{clarke:_metric_geomet_of_manif_of_rieman_metric}}]\label{prop:20}
  For any $Y \subseteq M$ and $g_0, g_1 \in \M$, we have the following
  inequality:
  \begin{equation*}
    \Theta_Y(g_0, g_1) \leq d(g_0, g_1) \left( \sqrt{n}\, d(g_0, g_1) +
      2 \sqrt{\Vol(M, g_0)} \right).
  \end{equation*}
  In particular, $\Theta_Y$ is a continuous pseudometric (w.r.t.~$d$).
\end{prop}

\subsection{Completions of Metric Spaces}\label{sec:compl-metr-spac}

To fix notation and recall a few elementary points, we briefly review
the completion of a metric space.  We will simply state the definition
and explore a couple of consequences of it, then give an alternative,
equivalent viewpoint for path metric spaces.

The \emph{precompletion} \label{p:precompl} of $(X, \delta)$ is the
set $\overline{(X, \delta)}^{\textnormal{pre}}$, usually just denoted
by $\overline{X}^{\textnormal{pre}}$, consisting of all Cauchy
sequences of $X$, together with the distance function
\begin{equation*}
  \delta(\{x_k\}, \{y_k\}) := \lim_{k \rightarrow \infty}
  \delta(x_k, y_k).
\end{equation*}
(We denote the distance function of the precompletion of a space using
the same symbol as for the space itself; which distance function is
meant will always be clear from the context.)

The \emph{completion} of $(X, \delta)$ is a quotient space of
$\overline{X}^{\textnormal{pre}}$, $\overline{X} :=
\overline{X}^{\textnormal{pre}} / {\sim}$,
where $\sim$ is the equivalence relation defined by
\begin{equation}\label{eq:97}
  \{x_k\} \sim \{y_k\} \Longleftrightarrow \delta(\{x_k\}, \{y_k\}) =
  0.
\end{equation}
Note that if $\{x_k\}$ is a Cauchy sequence in $X$ and $\{x_{k_l}\}$
is a subsequence, then clearly $\{x_{k_l}\} \sim \{x_k\}$.  Thus,
given an element of the precompletion of $X$, we can always pass to a
subsequence and still be talking about the same element of the
completion.

Recall that a \emph{path metric space} is a metric space for which the
distance between any two points coincides with the infimum of the
lengths of rectifiable curves joining the two points.  (We will also
call a rectifiable curve a \emph{finite-length path} or simply a
\emph{finite path}.)

The following theorem describes the completion of a path metric space
in terms of finite paths.  Its proof is straightforward.

\begin{thm}[{\cite[Thm.~2.2]{clarked):_compl_of_manif_of_rieman}}]\label{thm:30}
  Let $(X, \delta)$ be a path metric space.  Then the following
  description of the completion of $(X, \delta)$ is equivalent to the
  definition given above, in the sense that there exists a
  homeomorphism between the two completions.

  Define the precompletion $\overline{X}^{\textnormal{pre}}$ of $X$ to
  be the set of rectifiable curves $\alpha : (0,1] \rightarrow X$.
  It carries the pseudometric
  \begin{equation}\label{eq:101}
    \delta(\alpha_0, \alpha_1) := \lim_{t \to 0} \delta(\alpha_0(t), \alpha_1(t)).
  \end{equation}

  Then the completion of $(X, \delta)$ is $\overline{X} :=
  \overline{X}^{\textnormal{pre}} / {\sim}$,
  where $\alpha_0 \sim \alpha_1 \Longleftrightarrow \delta(\alpha_0,
  \alpha_1) = 0$.
\end{thm}

\subsection{Geometric Preliminaries}\label{sec:geom-prel}

We now give some of the nonstandard geometric facts that we will
need, in order to fix notation and recall the relevant notions.

\subsubsection{Sections of the Endomorphism Bundle of
  $M$}\label{sec:endom-bundle-m}

Given a section $H$ of the endomorphism bundle of $M$, the
determinant, trace, and eigenvalues of $H$ are well-defined functions
over $M$.  Furthermore, if $H$ is measurable/continuous/smooth, then
the determinant and trace will be so as well, since they are smooth
functions from the space of $n \times n$ matrices into $\R$.

The following proposition allows us to characterize positive definite
and positive semidefinite $(0, 2)$-tensors.

\begin{prop}[{\cite[Thm.~7.2.1]{horn90:_matrix_analy}}]\label{prop:7}
  A symmetric $n \times n$ matrix $T$ is positive definite
  (resp.~positive semidefinite) if and only if all eigenvalues of $T$
  are positive (resp.~nonnegative).

  In particular, if $T$ is positive definite (resp.~positive
  semidefinite), then $\det T > 0$ (resp.~$\det T \geq 0$).  If $T$ is
  positive semidefinite but not positive definite, then $\det T = 0$.
\end{prop}

We also need a result on the eigenvalues of a section of the
endomorphism bundle.

\begin{lem}[{\cite[Lemma~2.11]{clarked):_compl_of_manif_of_rieman}}]\label{lem:45}
  Let $h$ be any continuous, symmetric $(0,2)$-tensor field.  Suppose
  $g$ is a Riemannian metric on $M$, and let $H$ be the $(1,1)$-tensor
  field obtained from $h$ by raising an index using $g$.  (That is,
  locally $H^i_j = g^{ik} h_{kj}$.)  Then $H$ is a continuous section
  of the endomorphism bundle $\textnormal{End}(M)$.  Denote by
  $\lambda^H_{\textnormal{min}}(x)$ the smallest eigenvalue of $H(x)$.
  We have that
  \begin{enumerate}
  \item $\lambda^H_{\textnormal{min}}$ is a continuous function and
  \item if $h$ is positive definite, then $\min_{x \in M}
    \lambda^H_{\textnormal{min}}(x) > 0$.
  \end{enumerate}

  Furthermore, if $\lambda^H_{\textnormal{max}}(x)$ denotes the
  largest eigenvalue of $H(x)$, then $\lambda^H_{\textnormal{max}}$ is
  a continuous and hence bounded function.
 \end{lem}
\begin{proof}
  
  From the min-max theorem
  \cite[Thm.~XIII.1]{reed78:_method_of_moder_mathem_physic_iv}, one
  can show that the map $A \mapsto \lambda^A_{\min{}}$ (resp.~$A
  \mapsto \lambda^A_{\max{}}$) is a concave (resp.~convex) function
  from the space of $n \times n$ symmetric matrices to $\R$.  Thus
  both mappings are continuous
  \cite[Thm.~10.1]{rockafellar70:_convex_analy}.  The proof of the
  continuity of $\lambda^H_{\min{}}$ and $\lambda^H_{\max{}}$ then
  follows via a standard argument using compactness of the sphere
  bundle $SM \subset TM$.

  The bound on the minimal eigenvalue follows from continuity and
  Proposition \ref{prop:7}.
\end{proof}

\subsubsection{Lebesgue Measure on
  Manifolds}\label{sec:lebesg-meas-manif}

The concept of Lebesgue measurability carries over from $\R^n$ to
smooth (or even topological) manifolds 
by simply declaring a subset to be Lebesgue measurable if its
restriction to coordinates is Lebesgue measurable.  Since the
transition functions are smooth, this is independent of the chosen
coordinates.  Transition functions will also map nullsets to nullsets,
so this notion is well-defined.

\begin{cvt}\label{cvt:7}
  Whenever we refer to a measure-theoretic concept on $M$, we
  implicitly mean that we work with Lebesgue measure or Lebesgue sets,
  unless we explicitly state otherwise.
\end{cvt}

With Lebesgue measurable sets well-defined, the concept of a
measurable function or a measurable map between manifolds is also
well-defined.
We can also speak about Lebesgue measures---for example, any
nonnegative $n$-form $\mu$ on $M$ with measurable coefficients induces
a Lebesgue measure on $M$.

It is not hard to see that the same relation between Lebesgue
measurable sets and Borel measurable sets that holds on $\R^n$
\cite[\S 3.11]{rana02:_introd_to_measur_and_integ} also holds on $M$.
Namely, any Lebesgue measurable set $E$ can be decomposed as $E = F
\cup G$, where $F$ is Borel measurable and $G$ is a Lebesgue-nullset.

\subsection{Notation and Conventions}\label{sec:conventions}

Before we begin with the main body of the work, we will describe all
nonstandard notation and conventions that will be used throughout the
text.

The first thing we do is fix a reference metric, with respect to which
all standard concepts will be defined.

\begin{cvt}\label{cvt:3}
  For the remainder of the paper, we fix an element $g \in \M$.
  Whenever we refer to the $L^p$ norm, $L^p$ topology, $L^p$
  convergence etc., we mean that induced by $g$ unless we explicitly
  state otherwise.  The designation nullset refers to Lebesgue
  measurable subsets of $M$ that have zero measure with respect to
  $\mu_g$.  If we say that something holds almost everywhere, we mean
  that it holds outside of a $\mu_g$-nullset.

  If we have a tensor $h \in \s$, we denote by the capital letter $H$
  the tensor obtained by raising an index with $g$, i.e., locally
  $H^i_j := g^{ik} h_{kj}$.  Given a point $x \in M$ and an element $a
  \in \Matx$, the capital letter $A$ means the same---i.e., we assume
  some coordinates and write $A = g(x)^{-1} a$, though for readability
  we will generally omit $x$ from the notation.
\end{cvt}

Next, we'll fix an atlas of coordinates on $M$ that is convenient to
work with.

\begin{dfn}\label{dfn:1}
  We call a finite atlas of coordinates $\{(U_\alpha, \phi_\alpha)\}$
  for $M$ \emph{amenable} if for each $U_\alpha$, there exist a
  compact set $K_\alpha$ and a different coordinate chart $(V_\alpha,
  \psi_\alpha)$ (which does not necessarily belong to $\{ (U_\alpha,
  \phi_\alpha) \}$) such that
  \begin{equation*}
    U_\alpha \subset K_\alpha \subset V_\alpha \quad \textnormal{and}
    \quad \phi_\alpha = \psi_\alpha | U_\alpha.
  \end{equation*}
\end{dfn}

\begin{cvt}\label{cvt:5}
  For the remainder of this paper, we work over a fixed amenable
  coordinate atlas $\{(U_\alpha, \phi_\alpha)\}$ for all computations
  and concepts that require local coordinates.
\end{cvt}

The next lemma heuristically says the following: in amenable
coordinates, the coordinate representations of a smooth metric are
somehow ``uniformly positive definite''.  Additionally, the
coefficients satisfy a uniform upper bound.

\begin{lem}\label{lem:47}
  For any metric $\tilde{g} \in \M$, there exist constants
  $\delta(\tilde{g}) > 0$ and $C(\tilde{g}) < \infty$, depending only
  on $\tilde{g}$, with the property that for any $\alpha$, any $x \in
  U_\alpha$, and $1 \leq i,j \leq n$,
  \begin{equation}\label{eq:33}
    |\tilde{g}_{ij}(x)| \leq C(\tilde{g})\ \textnormal{and}\ \lambda^{\tilde{G}}_{\textnormal{min}}(x) \geq \delta(\tilde{g}),
  \end{equation}
  where we of course mean the value of $\tilde{g}_{ij}(x)$ in the
  chart $(U_\alpha, \phi_\alpha)$.
\end{lem}
\begin{proof}
  Note that Definition \ref{dfn:1} implies that $\phi_\alpha(U_\alpha)
  \subseteq \R^n$ is a relatively compact subset of
  $\psi_\alpha(V_\alpha)$.  Thus, the proof of the first inequality is
  immediate and the second is clear from Lemma \ref{lem:45}.
\end{proof}

\begin{rmk}\label{rmk:2}
  The estimate $|\tilde{g}_{ij}(x)| \leq C(\tilde{g})$ also implies an
  upper bound in terms of $C(\tilde{g})$ on $\det \tilde{g}(x)$.  This
  is clear from the fact that the determinant is a homogeneous
  polynomial in $\tilde{g}_{ij}(x)$ with $n!$ terms and coefficients
  $\pm 1$.
\end{rmk}

The main point of using a fixed amenable coordinate atlas is the
following: it gives us an easily understood and uniform---but
nevertheless coordinate-dependent---notion of how ``large'' or
``small'' a metric is.  The dependence of this notion on coordinates
is perhaps somewhat dissatisfying at first glance, but it should be
seen as merely an aid in our quest to prove statements that are,
indeed, invariant in nature.

It is necessary to introduce somewhat more general objects than
Riemannian metrics in this paper:

\begin{dfn}\label{dfn:24}
  Let $\tilde{g}$ be a section of $S^2 T^* M$.  Then $\tilde{g}$ is
  called a \emph{(Riemannian) semimetric} if it induces a positive
  semidefinite scalar product on $T_x M$ for each $x \in M$.
\end{dfn}

We now make a couple of definitions on semimetrics and sequences of
metrics.

\begin{dfn}\label{dfn:23}
  Let $\tilde{g}$ be a semimetric on $M$ (which we do not assume to be
  even measurable).
  We define the set
  \begin{equation*}
    X_{\tilde{g}} := \{ x \in M \mid \tilde{g}(x)\ \textnormal{is not
      positive definite} \} \subset M,
  \end{equation*}
  which we call the \emph{deflated set} of $\tilde{g}$.

  We call $\tilde{g}$ \emph{bounded} if there exists a constant
  $C$ such that
  \begin{equation*}
    | \tilde{g}_{ij}(x) | \leq C
  \end{equation*}
  for a.e.~$x \in M$ and all $1 \leq i, j \leq n$.  Otherwise
  $\tilde{g}$ is called \emph{unbounded}.
\end{dfn}

\begin{dfn}\label{dfn:25}
  Let $\{ g_k \} \subset \M$ be any sequence.  We define the set  \begin{equation*}
    D_{\{ g_k \}} := \{ x \in M \mid \forall \delta > 0,\ \exists k
    \in \N\ \textnormal{s.t.}\ \det G_k(x) <
    \delta \},
  \end{equation*}
  which we call the \emph{deflated set} of $\{g_k\}$.
\end{dfn}

The last definition we need in this vein distinguishes smooth metrics
from (possibly nonsmooth) semimetrics.

\begin{dfn}\label{dfn:8}
  A semimetric $\tilde{g}$ is called \emph{degenerate} if $\tilde{g}
  \not\in \M$ and \emph{nondegenerate} if $\tilde{g} \in \M$.
\end{dfn}

Note that any measurable semimetric $\tilde{g}$ on $M$ induces a
nonnegative measure on $M$ that is absolutely continuous with respect
to the fixed volume form $\mu_g$.

A measurable Riemannian semimetric $\tilde{g}$ on $M$ gives rise to an
``$L^2$ scalar product'' on measurable functions in the following way.
For any two functions $\rho$ and $\sigma$ on $M$, we define, as usual,
$\mu_{\tilde{g}} := \sqrt{\det \tilde{g}}\, dx^1 \cdots dx^n$ and
\begin{equation}\label{eq:121}
  (\rho, \sigma)_{\tilde{g}} := \integral{M}{}{\rho \sigma}{\mu_{\tilde{g}}}.
\end{equation}
(We denote this by the same symbol as the $L^2$ scalar product on
$\s$; which is meant will always be clear from the context.)  We put
``$L^2$ scalar product'' in quotation marks because unless we put
specific conditions on $\rho$, $\sigma$, and $\tilde{g}$,
\eqref{eq:121} is not guaranteed to be finite.  It suffices, for
example, to demand that $\rho$ and $\sigma$ are bounded and that
the total volume $\Vol(M, \tilde{g}) =
\integral{M}{}{}{\mu_{\tilde{g}}}$ of $\tilde{g}$ is finite.  As in
the case of the $L^2$ scalar product on $\s$, if $g_0$ and $g_1$ are
both continuous metrics, then $(\cdot, \cdot)_{g_0}$ and $(\cdot ,
\cdot)_{g_1}$ are equivalent scalar products on $C^\infty(M)$.
Therefore they induce the same topology, which we call the \emph{$L^2$
  topology}.

\section{Amenable Subsets}\label{sec:compl-an-amen}

We begin the study of the completion of $\M$ in this section, by first
completing very special subsets of $\M$ called \emph{amenable subsets}
(defined below).  The main result of the section is that the
completion of such a subset with respect to $d$ coincides with the
completion with respect to the $L^2$ norm on $\s$, the vector space in
which $\M$ resides.

\subsection{Amenable Subsets and their
Basic  Properties}\label{sec:amen-subs-their}

For the following definition, recall that we work over an amenable
atlas (cf.~Definition \ref{dfn:1}).

\begin{dfn}\label{dfn:2}
  We call a subset $\U \subset \M$ \emph{amenable} if $\U$ is convex
  and we can find constants $C, \delta > 0$ such that for all
  $\tilde{g} \in \U$, $x \in M$ and $1 \leq i,j \leq n$,
  \begin{equation*}
    \lambda^{\tilde{G}}_{\textnormal{min}}(x) \geq \delta
  \end{equation*}
  (where we recall that $\tilde{G} = g^{-1} \tilde{g}$, with $g$ our
  fixed metric) and
  \begin{equation*}
    |\tilde{g}_{ij}(x)| \leq C.
  \end{equation*}
\end{dfn}

\begin{rmk}\label{rmk:1}
  We make a couple of remarks about the definition:
  \begin{enumerate}
  \item The requirement that $\U$ is convex is technical, and is there
    to insure that we can consider simple, straight-line paths between
    points of $\U$ to estimate the distance between them.
  \item \label{item:1} Recall that the function sending a matrix to
    its minimal eigenvalue is concave (cf.~the proof of Lemma
    \ref{lem:45}).  Also, the absolute value function on $\R$ is
    convex by the triangle inequality.  Therefore, the two bounds
    given in Definition \ref{dfn:2} are compatible with the
    requirement of convexity.
  \end{enumerate}
\end{rmk}

One useful property the metrics $\tilde{g}$ of an amenable subset have
is that the Radon-Nikodym derivatives $( \mu_{\tilde{g}} / \mu_g )$,
with respect to the reference volume form $\mu_g$, are bounded away
from zero and infinity independently of $\tilde{g}$.

\begin{lem}\label{lem:49}
  Let $\U$ be an amenable subset.  Then there exists a constant $K >
  0$ such that for all $\tilde{g} \in \U$,
  \begin{equation}\label{eq:130}
    \frac{1}{K} \leq \left( \frac{\mu_{\tilde{g}}}{\mu_g} \right) \leq K
  \end{equation}
\end{lem}
\begin{proof}
  First, we note that
  \begin{equation*}
    \left( \frac{\mu_{\tilde{g}}}{\mu_g} \right) = \det \tilde{G} \quad
    \textnormal{and} \quad \left( \frac{\mu_{\tilde{g}}}{\mu_g}
    \right)^{-1} = \left( \frac{\mu_g}{\mu_{\tilde{g}}}
    \right) = (\det \tilde{G})^{-1}.
  \end{equation*}
  So the bounds \eqref{eq:130} are equivalent to upper bounds on both
  $\det \tilde{G}$ and $(\det \tilde{G})^{-1}$.
  
  Now, if the eigenvalues of $\tilde{G}$ are $\lambda^{\tilde{G}}_1,
  \dots, \lambda^{\tilde{G}}_n$, then
  \begin{equation*}
    \det \tilde{G} = \lambda^{\tilde{G}}_1
    \cdots \lambda^{\tilde{G}}_n \geq \left(
      \lambda^{\tilde{G}}_{\textnormal{min}} \right)^n \geq \delta^n,
  \end{equation*}
  where $\delta$ is the constant guaranteed by the fact that
  $\tilde{g} \in \U$.  This allows us to bound $(\det \tilde{G})^{-1}$
  from above.

  To bound $\det \tilde{G}$ from above, it is sufficient to bound the
  absolute value of the coefficients of $\tilde{G} = g^{-1} \tilde{g}$
  from above.  But bounds on the coefficients of $\tilde{g}$ are
  already assured by the fact that $\tilde{g} \in \U$, and bounds on
  the coefficients of $g^{-1}$ are guaranteed by the fact that
  $g^{-1}$ is a fixed, smooth cometric on $M$.  So we are finished.
\end{proof}

Amenable subsets guarantee good behavior of the norms on $\s$ that are
defined by their members---namely, the norms are in some sense
``uniformly equivalent''.  More precisely, we have:

\begin{lem}\label{lem:18}
  Let $\U \subset \M$ be an amenable subset.  Then there exists a
  constant $K$ such that for all pairs $g_0, g_1 \in \U$ and all $h
  \in \s$,
  \begin{equation*}
    \frac{1}{K} \norm{h}_{g_1} \leq \norm{h}_{g_0} \leq
    K \norm{h}_{g_1}.
  \end{equation*}
\end{lem}
\begin{proof}

  We will show that the norm of each $\tilde{g} \in \U$ is equivalent
  to that of the fixed reference metric $g$ with a fixed constant $K$.
  This is equivalent to the following statement.  Let
  \begin{equation*}
    T_{\tilde{g}} : (S^2 T^* M, \langle \cdot , \cdot
    \rangle_{\tilde{g}}) \rightarrow (S^2 T^* M, \langle \cdot, \cdot \rangle_g)
  \end{equation*}
  be the identity mapping on the level of sets, sending the bundle
  $S^2 T^* M$ with the Riemannian structure $\langle \cdot, \cdot
  \rangle_{\tilde{g}}$ to itself with the Riemannian structure
  $\langle \cdot , \cdot \rangle_g$.  Let $N(T_{\tilde{g}})(x)$ be the
  operator norm of $T_{\tilde{g}}(x) : \satx \rightarrow \satx$, and
  let $N(T_{\tilde{g}}^{-1})(x)$ be defined similarly.
  Then the statement on norms holds if and only if there are constants
  $K_0$ and $K_1$ such that
  \begin{equation*}
    N(T_{\tilde{g}})(x)^2, N(T_{\tilde{g}}^{-1})(x)^2
    \leq K_0 \quad \textnormal{and} \quad  (\mu_g / \mu_{\tilde{g}}), (\mu_{\tilde{g}} / \mu_g)
    \leq K_1.
  \end{equation*}

  But $N(T_{\tilde{g}})$ and $N(T_{\tilde{g}}^{-1})$ are continuous
  functions on the compact manifold $M$ for fixed $\tilde{g}$.
  Secondly, we notice that $N(T_{\tilde{g}})(x)$ and
  $N(T_{\tilde{g}}^{-1})(x)$ depend only on the coordinate
  representations of $\tilde{g}(x)$ and $g(x)$.  But $\tilde{g}(x)$
  and $g(x)$ can only range over a compact subset of the space of
  positive definite $n \times n$ symmetric matrices, because
  $\tilde{g}$ ranges over an amenable subset.  This implies the
  existence of $K_0$.

  The existence of $K_1$ is immediately implied by Lemma \ref{lem:49}.
\end{proof}

Lemma \ref{lem:49} immediately implies that the function $\tilde{g}
\mapsto \Vol(M,\tilde{g})$ is bounded when restricted to any amenable
subset.  Recalling the form of the estimate in Proposition
\ref{prop:20} then shows the following lemma.

\begin{lem}\label{lem:26}
  Let $\U$ be an amenable subset and $g \in \M$.  Then there exists a
  constant $V$ such that for any $g_0, g_1 \in \U$ and $Y \subset M$,
  \begin{equation*}
    \Theta_Y(g_0, g_1) \leq 2 d(g_0, g_1) \left( \frac{2 \sqrt{n}}{4}d(g_0, g_1) +
      \sqrt{V} \right).
  \end{equation*}
  More precisely, $V = \sup_{\tilde{g} \in \U}\Vol(M,\tilde{g})$,
  which is finite by the discussion preceding the lemma.
\end{lem}

\subsection{The Completion of $\U$ with Respect to $d$ and
  $\normdot_g$}\label{sec:completion-u-with}

We are now ready to prove a result that, in particular, implies
equivalence of the topologies defined by $d$ and $\normdot_g$ on an
amenable subset $\U$.

\begin{thm}\label{thm:5}
  Consider the $L^2$ topology on $\M$ induced from the scalar product
  $(\cdot,\cdot)_g$ (where $g$ is fixed).  Let $\U \subset \M$ be any
  amenable subset.

  Then the $L^2$ topology on $\U$ coincides with the topology induced
  from the restriction of the Riemannian distance function $d$ of $\M$
  to $\U$.

  Additionally, the following holds:
  \begin{enumerate}
  \item There exists a constant $K$ such that
    $d(g_0,g_1) \leq K \norm{g_1 - g_0}_g$
    for all $g_0,g_1 \in \U$.
  \item \label{item:9} For any $\epsilon > 0$, there exists $\delta >
    0$ such that if $d(g_0,g_1) < \delta$, then $\norm{g_0 - g_1}_g <
    \epsilon$.
  \end{enumerate}
\end{thm}
\begin{proof}
  To prove (1), consider the linear path $g_t$ from $g_0$ to $g_1$.
  Note that we can clearly find an amenable subset $\U'$ containing
  $\U$ and $g$.  We then have
  \begin{equation}\label{eq:38}
    L(g_t) = \int_0^1 \norm{(g_t)'}_{g_t} \, dt = \int_0^1 \norm{g_1 -
      g_0}_{g_t} \, dt \leq \int_0^1 K \norm{g_1 - g_0}_g \, dt = K
    \norm{g_1 - g_0}_g,
  \end{equation}
  where $K$ is the constant associated to $\U'$ guaranteed by Lemma
  \ref{lem:18}.  Since $d(g_0,g_1) \leq L(g_t)$ and the constant $K$
  depends only on the set $\U$, this inequality is shown.

  Since $\Matx$ is a finite-dimensional Riemannian manifold, the
  topology induced from $\theta^g_x$ is the same as the manifold
  topology, which in turn is given by any norm on $\satx$.  For
  instance this norm is given by the scalar product $\langle \cdot ,
  \cdot \rangle_{g(x)}$ on $\satx$, which we recall is given by
  \begin{equation}\label{eq:41}
    \langle h , k \rangle_{g(x)} = \tr_{g(x)}(h k)
  \end{equation}
  for $h, k \in \satx$.  That these two topologies are the same
  implies, in particular, that for all $\zeta > 0$ and $\tilde{g} \in
  \Matx$, we can find $\kappa > 0$ such that
  \begin{equation*}
    B^{\theta^g_x}_{\tilde{g}}(\zeta) \subset B^{\langle \cdot, \cdot \rangle_{g(x)}}_{\tilde{g}}(\kappa),
  \end{equation*}
  where
  the above denotes the $\theta^g_x$-ball of radius $\zeta$ around
  $\tilde{g}$ and the $\langle \cdot, \cdot \rangle_{g(x)}$-ball of
  radius $\kappa$ around $\tilde{g}$, respectively.
  
  Now, for $x \in M$ and $\tilde{g} \in \M$, we define a function
  $\eta_{x, \tilde{g}}(\zeta)$ by
  \begin{equation*}
    \eta_{x,\tilde{g}}(\zeta) := \inf \left\{ \kappa \in \R \mid
      B^{\theta^g_x}_{\tilde{g}(x)}(\zeta) \subset B^{\langle \cdot , \cdot
        \rangle_{g(x)}}_{\tilde{g}(x)}(\kappa)
    \right\}
  \end{equation*}
  Then, because of the smooth dependence of $\langle \cdot, \cdot
  \rangle^0$ and $\langle \cdot , \cdot \rangle_{g(x)}$ on $x$,
  $\eta_{x,\tilde{g}}(\zeta)$ is continuous separately in $x$ and
  $\tilde{g}$.  If we define
  \begin{equation*}
    \U_x :=
    \left\{
      \hat{g}(x) \mid \hat{g} \in \U
    \right\},
  \end{equation*}
  then $\U_x$ is a relatively compact subset of $\Matx$, since $\U$ is
  amenable.  Since $M$ is also compact, for any fixed $\zeta > 0$, we
  can define a function
  \begin{equation*}
    \eta(\zeta) := \sup_{x \in M,\ \tilde{g} \in \U}
    \eta_{x,\tilde{g}}(\zeta)
    < \infty.
  \end{equation*}
  It follows from the definition that $\eta(\zeta) \rightarrow 0$
  for $\zeta \rightarrow 0$.

  Because of the relative compactness of $\U_x$ for each $x \in M$,
  together with compactness of $M$, there exists a constant $C_0$ such
  that $\theta^g_x(g_0(x), g_1(x)) \leq C_0$ for all $g_0,g_1 \in \U$
  and $x \in M$.  This implies immediately that
  \begin{equation*}
    \Theta_M(g_0, g_1) = \integral{M}{}{\theta^g_x (g_0(x), g_1(x))}{\mu_g(x)} \leq C_0 \Vol(M,g).
  \end{equation*}

  Now, choose $\zeta > 0$ small enough that $\eta(\zeta) < \epsilon / \sqrt{2 \Vol(M,g)}$.

  By Lemma \ref{lem:26}, there exists a constant $V$ such that
  \begin{equation}\label{eq:46}
    \Theta_M(g_0, g_1) \leq 2 d(g_0, g_1) \left( \frac{2 \sqrt{n}}{4}d(g_0, g_1) +
      \sqrt{V} \right)
  \end{equation}
  for all $g_0, g_1 \in \U$.
  
  Choose $\delta$ small enough that
  \begin{equation*}
    2 \delta \left( \frac{2 \sqrt{n}}{4} \delta + \sqrt{V} \right) <
    \frac{\epsilon^2 \zeta}{2 \eta(C_0)^2}.
  \end{equation*}
  We claim that $d(g_0, g_1) < \delta$ implies that $\norm{g_1 - g_0}_g
  < \epsilon$.  Note that the choices of $\zeta$ and $C_0$ were
  made independently of $g_0$ and $g_1$, hence $\delta$ is independent
  of $g_0$ and $g_1$, as required.
  
  We define two closed subsets of $M$ by
  \begin{align*}
    M_+ &:= \left\{ x \in M \mid \theta^g_x (g_0(x), g_1(x)) \geq
      \zeta
    \right\}, \\
    M_- &:= \left\{ x \in M \mid \theta^g_x (g_0(x), g_1(x)) \leq
      \zeta \right\}.
  \end{align*}
  From \eqref{eq:46} and our choice of $\delta$, we have that
  \begin{equation}\label{eq:37}
    \zeta \Vol(M_+, g) = \zeta \integral{M_+}{}{}{\mu_g} \leq
    \integral{M_+}{}{\theta^g_x (g_0(x), g_1(x))}{\mu_g(x)} \leq \Theta_M(g_0, g_1) <
    \frac{\epsilon^2 \zeta}{2 \eta(C_0)^2}.
  \end{equation}
  implying $\Vol(M_+, g_0) < \epsilon^2 / 2 \eta(C_0)^2$.

  From the definitions of $M_-$ and $\eta$, we have that
  \begin{equation*}
    \sqrt{\langle g_1(x) - g_0(x), g_1(x) - g_0(x) \rangle_{g(x)}} \leq
    \eta(\zeta)
  \end{equation*}
  on $M_-$.  From $\theta^g_x(g_0(x), g_1(x)) \leq C_0$,
  we have that
  \begin{equation*}
    \sqrt{\langle g_1(x) - g_0(x), g_1(x) - g_0(x) \rangle_{g(x)}}
    \leq \eta(C_0)
  \end{equation*}
  on all of $M$, and in particular on $M_+$.  Using this, we compute
  \begin{align*}
    \norm{g_1 - g_0}_g^2 &= \integral{M}{}{\langle g_1(x) - g_0(x),
      g_1(x) - g_0(x) \rangle_{g(x)}}{\mu_g(x)}     \leq \eta(\zeta)^2 \integral{M_-}{}{}{\mu_g} + \eta(C_0)^2
    \integral{M_+}{}{}{\mu_g} \\
    &< \eta(\zeta)^2 \Vol(M,g) + \eta(C_0)^2 \frac{\epsilon^2}{2
      \eta(C_0)^2}     < \frac{\epsilon^2}{2} + \frac{\epsilon^2}{2} = \epsilon^2.
  \end{align*}
  This proves the second statement.
\end{proof}

Theorem \ref{thm:5} will give us our first result regarding the
completion of $\M$.  First, though, we need to make some definitions
and prove a statement about metric spaces.

\begin{dfn}\label{dfn:4}

  If $\U \subset \M$ is any subset, we
  denote by $\U^0$ the $L^2$-completion of $\U$ (that is, the
  completion of $\U$ with respect to $\normdot_g$).
\end{dfn}

Let's look back at Theorem \ref{thm:5} again.  The first statement
says that for any amenable subset $\U$ and any $g \in \M$, $d$ is
Lipschitz continuous with respect to $\normdot_g$ when viewed as a
function on $\U \times \U$.  The second statement says that
$\normdot_g$ is uniformly continuous on $\U \times \U$ with respect to
$d$.  To put this knowledge to good use, we will need the following
lemma:

\begin{lem}\label{lem:20}
  Let $X$ be a set, and let two metrics, $d_1$ and $d_2$, be defined
  on $X$.  Denote by $\phi : (X,d_1) \rightarrow (X, d_2)$ the map
  which is the identity on the level of sets, i.e., $\phi$ simply maps
  $x \mapsto x$.  Finally, denote by $\overline{X}^1$ and
  $\overline{X}^2$ the completions of $X$ with respect to $d_1$ and
  $d_2$, respectively.

  If both $\phi$ and $\phi^{-1}$ are uniformly continuous, then there
  is a natural homeomorphism between $\overline{X}^1$ and
  $\overline{X}^2$.
\end{lem}
\begin{proof}

  The proof follows in a straightforward manner from the definition of
  the completion of a metric space from Section
  \ref{sec:compl-metr-spac}, and the fact that a uniformly continuous
  function maps Cauchy sequences to Cauchy sequences.  The natural
  homeomorphism is given by the unique uniformly continuous extension
  of $\phi$ to $\overline{X^1}$.
\end{proof}

Now, Theorem \ref{thm:5} and Lemma \ref{lem:20} immediately imply

\begin{thm}\label{thm:7}
  Let $\U$ be an amenable subset.  Then we can identify
  $\overline{\U}$, the completion of $\U$ with respect to $d$, with
  $\U^0$, in the sense of Lemma \ref{lem:20}.  We can make the natural
  homeomorphism $\overline{\U} \rightarrow \U^0$ into an isometry by
  placing a metric on $\U^0$ defined by
  \begin{equation*}
    d(g_0, g_1) = \lim_{k \rightarrow \infty} d(g^0_k, g^1_k),
  \end{equation*}
  where $\{g^0_k\}$ and $\{g^1_k\}$ are any sequences in $\U$ that
  $L^2$-converge to $g_0$ and $g_1$, respectively.
\end{thm}

We have thus found a nice description of the completion of very
special subsets of $\M$.  As already discussed, our plan now is to
start removing the nice properties that allowed us to understand
amenable subsets so clearly, advancing through the completions of ever
larger and more generally defined subsets of $\M$.  Before that,
though, we need to study general Cauchy sequences in $\M$ more
closely in the next section.

\section{Cauchy sequences and $\omega$-convergence}\label{cha:almost-everywh-conv}

In this chapter, we introduce and study a fundamental notion of
convergence of our own invention for $d$-Cauchy sequences in $\M$.  We
call this \emph{$\omega$-convergence}, and its importance is made
clear through two theorems we will prove, an existence and a
uniqueness result.  The existence result, proved in Section
\ref{sec:existence-ad-limit}, says that every $d$-Cauchy sequence has
a subsequence that $\omega$-converges to a measurable semimetric,
which we will then show has finite total volume.  The uniqueness
result, proved in Section \ref{sec:uniqueness-ad-limit}, is that two
$\omega$-convergent Cauchy sequences in $\M$ are equivalent (in the
sense of \eqref{eq:97}) if and only if they have the same
$\omega$-limit.  These results allow us to identify an equivalence
class of $d$-Cauchy sequences with the unique $\omega$-limit that its
representatives subconverge to, and thus give a geometric meaning to
points of $\overline{\M}$.

\subsection{Existence of the $\omega$-Limit}\label{sec:existence-ad-limit}

We begin this section with an important estimate and some examples,
followed by the definition of $\omega$-convergence and some of its
basic properties.  After that, we start on the existence proof by
showing a pointwise version, i.e., an analogous result on $\Matx$.
Finally, we globalize this pointwise result to show the existence of
an $\omega$-convergent subsequence for any Cauchy sequence in $\M$.

\subsubsection{Volume-Based Estimates on $d$ and Examples}\label{sec:volume-based-estim}

The following surprising proposition shows us that two metrics that
differ only on a subset with small (intrinsic) volume are close with
respect to $d$. 

\begin{prop}\label{prop:18}
  Suppose that $g_0, g_1 \in \M$, and let $E := \carr (g_1 - g_0) = \{
  x \in M \mid g_0(x) \neq g_1(x) \}$.  Then there exists a constant
  $C(n)$ depending only on $n = \dim M$ such that
  \begin{equation}\label{eq:1}
    d(g_0, g_1) \leq C(n) \left( \sqrt{\Vol(E, g_0)} +
      \sqrt{\Vol(E,g_1)} \right).
  \end{equation}
  In particular, we have
  \begin{equation*}
    \diam \left( \{ \tilde{g} \in \M \mid \Vol(M, \tilde{g}) \leq
      \delta \} \right) \leq 2 C(n) \sqrt{\delta}.
  \end{equation*}
\end{prop}
\begin{proof}
  The second statement follows immediately from the first, so we only
  prove the first.
  
  The heuristic idea is the following.  We want to construct a family
  of paths with three pieces, depending on a real parameter $s$, such
  that the metrics do not change on $M \setminus E$ as we travel along
  the paths.  Therefore, we pretend that we can restrict all
  calculations to $E$.  On $E$, the first piece of the path is the
  straight line from $g_0$ to $s g_0$ for some small positive number
  $s$.  It is easy to compute a bound for the length of this path
  based on $\Vol(E, g_0)$.  The second piece is the straight line from
  $s g_0$ to $s g_1$, which, as we will see, has length approaching
  zero for $s \rightarrow 0$.  The last piece is the straight line
  from $s g_1$ to $g_1$, which again has length bounded from above by
  an expression involving $\Vol(E, g_1)$.
  
  Our job is to now take this heuristic picture, which uses paths of
  $L^2$ metrics, and construct a family of paths of \emph{smooth}
  metrics that captures the essential properties.
  
  For each $k \in \N$ and $s \in (0,1]$, we define three families of
  metrics as follows.  Choose closed sets $F_k \subseteq E$ and open
  sets $U_k$ containing $E$ such that $\Vol(U_k, g_i) - \Vol(F_k, g_i)
  \leq 1/k$ for $i = 0,1$.  (This is possible because the Lebesgue
  measure is regular.)  Let $f_{k,s} \in C^\infty(M)$ be functions
  with the following properties:
  \begin{enumerate}
  \item $f_{k,s}(x) = s$ if $x \in F_k$,
  \item $f_{k,s}(x) = 1$ if $x \not\in U_k$ and
  \item $s \leq f_{k,s}(x) \leq 1$ for all $x \in M$.
  \end{enumerate}
  Now, for $t \in [0,1]$, define
  \begin{equation*}
    \hat{g}^{k,s}_t := ((1-t) + t f_{k,s}) g_0, \qquad
    \bar{g}^{k,s}_t := f_{k,s} ((1-t) g_0 + t g_1), \qquad
    \tilde{g}^{k,s}_t := ((1-t) + t f_{k,s}) g_1.
  \end{equation*}
  We view these as paths in $t$ depending on the family parameter
  $s$.  Furthermore, we define a concatenated path
  \begin{equation*}
    g^{k,s}_t := \hat{g}^{k,s}_t * \bar{g}^{k,s}_t * (\tilde{g}^{k,s}_t)^{-1},
  \end{equation*}
  where of course the inverse means we run through the path backwards.
  It is easy to see that $g^{k,s}_0 = g_0$ and $g^{k,s}_1 = g_1$ for
  all $s$.

  We now investigate the lengths of each piece of $g^{k,s}_t$
  separately, starting with that of $\hat{g}^{k,s}_t$.  Recalling that
  by Convention \ref{cvt:3}, $G_0 = g^{-1} g_0$, we compute
  \begin{align*}
    L(\hat{g}^{k,s}_t)     &= \integral{0}{1}{\left( \integral{M}{}{\tr_{((1 - t) + t
            f_{k,s}) g_0} \left( ((f_{k,s} - 1)g_0)^2 \right)
          \sqrt{\det \left( ((1 - t) + t f_{k,s}) G_0 \right)}}{\mu_{g}} \right)^{1/2}}{d t} \\
    &= \integral{0}{1}{\left( \integral{U_k}{}{((1 - t) + t
          f_{k,s})^{\frac{n}{2} - 2} \tr_{g_0} \left( ((1 -
            f_{k,s})g_0)^2 \right) \sqrt{\det G_0}}{\mu_{g}} \right)^{1/2}}{d t}.
  \end{align*}
  since $\det (\lambda A) = \lambda^{n/2} \det A$ for any $n \times
  n$-matrix $A$ and $\lambda \in \R$.  Note that in the last line, we
  only integrate over $U_k$, which is justified by the fact that $1 -
  f_{k,s} = 0$ on $M \setminus U_k$.  Since $s > 0$, it is easy to see
  that
  \begin{equation*}
    (1 - f_{k,s})^2 \leq (1 - s)^2 < 1,
  \end{equation*}
  from which we can compute
  the estimate
  \begin{equation*}
    L(\hat{g}^{k,s}_t) < \integral{0}{1}{\left( n \integral{U_k}{}{((1
          - t) + t f_{k,s})^{\frac{n}{2} - 2}}{\mu_{g_0}}
      \right)^{1/2}}{d t}.
  \end{equation*}
  Now, to estimate this, we note that for $n \geq 4$, $\frac{n}{2} - 2
  \geq 0$ and therefore $f_{k,s} \leq 1$ implies that
  \begin{equation}\label{eq:72}
    L(\hat{g}_t^{k,s}) < \sqrt{n \Vol(U_k, g_0)}.
  \end{equation}
  For $1 \leq n \leq 3$, $\frac{n}{2} - 2 < 0$ and therefore one can
  compute that $f_{k,s} \geq s > 0$ implies
  \begin{equation*}
    ((1 - t) + t f_{k,s})^{\frac{n}{2} - 2} \leq (1 - t)^{\frac{n}{2}
      - 2}.
  \end{equation*}
  In this case, then,
  \begin{equation}\label{eq:71}
    L(\hat{g}_t^{k,s}) < \sqrt{n \Vol(U_k, g_0)} \integral{0}{1}{(1 -
      t)^{\frac{n}{4} - 1}}{dt},
  \end{equation}
  and the integral term is finite since $\frac{n}{4} - 1 > -1$.
  Furthermore, the value of this integral depends only on $n$.
  Putting together \eqref{eq:72} and \eqref{eq:71} therefore gives
  \begin{equation}\label{eq:75}
    L(\hat{g}^{k,s}_t) \leq C(n) \sqrt{\Vol(U_k, g_0)},
  \end{equation}
  where $C(n)$ is a constant depending only on $n$.

  In exact analogy, we can show that
  the same estimate holds with $g_1$ in place of $g_0$.
  
  Next, we look at the second piece of $g^{k,s}_t$.  Here we have,
  using that $g_1 - g_0 = 0$ on $M \setminus E$,
  \begin{align*}
    \norm{(\bar{g}^{k,s}_t)'}_{\bar{g}^{k,s}_t}^2 &= \integral{M}{}{
      \tr_{f_{k,s} ((1 - t) g_0 + t g_1)} \left( (f_{k,s} (g_1 -
        g_0))^2\right) 
      \sqrt{\det \left( f_{k,s} ((1 - t) G_0 + t G_1)
        \right)}}{\mu_g} \\
    &= \integral{E}{}{f_{k,s}^{n/2} \tr_{(1 -
        t) g_0 + t g_1} \left( (g_1 - g_0)^2 \right)
      \sqrt{\det ((1 - t) G_0 + t G_1)}}{\mu_g}.
  \end{align*}
  Since $f_{k,s}(x) = s$ if $x \in F_k$, and $f_{k,s}(x) \leq 1$ for
  all $x \in M$, it follows from the above that
  \begin{align*}
    \norm{(\bar{g}^{k,s}_t)'}_{\bar{g}^{k,s}_t}^2 &\leq s^{n/2}
    \integral{F_k}{}{\tr_{(1 - t) g_0 + t g_1} \left( (g_1 - g_0)^2
      \right) \sqrt{\det ((1 - t) G_0 + t G_1)}}{\mu_g} \\
    &\quad + \integral{E \setminus F_k}{}{\tr_{(1 - t) g_0 + t g_1}
      \left( (g_1 - g_0)^2 \right) \sqrt{\det ((1 - t) G_0 + t
        G_1)}}{\mu_g}
  \end{align*}
  For each fixed $t$ and $k$, the first term in the above clearly goes
  to zero as $s \rightarrow 0$.  By our assumption on the sets $F_k$,
  the second term goes to zero as $k \rightarrow \infty$ for each
  fixed $t$ (it does not depend on $s$ at all).  But since $t$ only
  ranges over the compact interval $[0,1]$ and all terms in the
  integrals depend smoothly on $t$, both of these convergences are
  uniform in $t$.  From this, it is easy to see that
  \begin{equation}\label{eq:135}
    \lim_{k \rightarrow \infty} \lim_{s \rightarrow 0}
    L(\bar{g}^{k,s}_t) = 0.
  \end{equation}
  Combining these considerations gives the desired estimate.
\end{proof}

As the following examples show, Proposition \ref{prop:18} implies that
we cannot expect a Cauchy sequence in $\M$ to converge pointwise over
subsets of $M$ with volume that vanishes in the limit.  Indeed, we
cannot control its behavior at all.

\begin{eg}
  Consider the case where $M$ is a two-dimensional torus, $M = T^2$.
  On the standard chart for $T^2$ ($[0,1] \times [0,1]$ with opposite
  edges identified), define the following sequences of flat metrics:
  \begin{equation*}
    g^1_k :=
    \begin{pmatrix}
      1 & 0 \\
      0 & k^{-1}
    \end{pmatrix}, \qquad
    g^2_k :=
    \begin{pmatrix}
      k^{-1} & 0 \\
      0 & k^{-1}
    \end{pmatrix}, \qquad
    g^3_k :=
    \begin{pmatrix}
      e^{k t} & 0 \\
      0 & e^{-2k t}
    \end{pmatrix}, \qquad
    g^4_k :=
    \begin{pmatrix}
      \abs{\cos k} & 0 \\
      0 & k^{-1}
    \end{pmatrix}.
  \end{equation*}
  Since $\Vol(T^2, g^i_k) \rightarrow 0$ for all $i = 1,2,3,4$,
  Proposition \ref{prop:18} implies that each of these sequences is
  $d$-Cauchy, and all are equivalent.  Yet in terms of pointwise
  convergence, $\{g^1_k\}$ converges to a circle (one dimension of
  $T^2$ collapses), $\{g^2_k\}$ converges to a point (both dimensions
  collapse), $\{g^3_k\}$ ``converges'' to a circle of infinite radius
  (in terms of Gromov-Hausdorff convergence, for example, it converges
  to the real line), and $\{g^4_k\}$ does not converge at all.
\end{eg}

\subsubsection{$\omega$-Convergence and its Basic Properties}\label{sec:ad-convergence-its}

In this subsection, we give a convergence notion suited to the
completion of $\M$, in that it allows sequences to behave badly on
sets that collapse in the limit.

First, though, recall that we define general measure-theoretic notions
(e.g., the notion of something holding almost everywhere, or a.e.)
using the fixed reference metric $g$ (cf.~Convention \ref{cvt:3}).
Furthermore, we need one definition before that of
$\omega$-convergence.

\begin{dfn}\label{dfn:7}
  We denote by $\Mm$ the set of all measurable semimetrics on $M$.
  That is, $\Mm$ is the set of all sections of $S^2 T^* M$ that have
  measurable coefficients and that induce a positive semidefinite
  scalar product on $T_x M$ for each $x \in M$.
  
  Define an equivalence relation ``$\sim$'' on $\Mm$ by $g_0 \sim
  g_1$ if and only if
  \begin{enumerate}
  \item their deflated sets $X_{g_0}$ and $X_{g_1}$ differ at most
    by a nullset, and
  \item $g_0(x) = g_1(x)$ for a.e.~$x \in M \setminus (X_{g_0} \cup X_{g_1})$.
  \end{enumerate}
  We denote the quotient space of $\Mm$ by
  \begin{equation*}
    \Mmhat := \Mm / {\sim}.
  \end{equation*}
\end{dfn}

\begin{dfn}\label{dfn:13}
  Let $\{g_k\}$ be a sequence in $\M$, and let $[g_\infty] \in
  \Mmhat$.  Recall that we denote the deflated set of the sequence
  $\{g_k\}$ by $D_{\{g_k\}}$ and the deflated set of an individual
  semimetric $\tilde{g}$ by $X_{\tilde{g}}$ (cf.~Definitions
  \ref{dfn:23} and \ref{dfn:25}).  We say that $\{g_k\}$
  \emph{$\omega$-converges} to $[g_\infty]$ if for every
  representative $\tilde{g}_\infty \in [g_\infty]$, the following
  holds:
  \begin{enumerate}
  \item \label{item:4} $\{g_k\}$ is $d$-Cauchy,
  \item \label{item:5} $X_{g_\infty}$ and $D_{\{g_k\}}$ differ at most
    by a nullset,
  \item \label{item:6} $g_k(x) \rightarrow g_\infty(x)$ for a.e.~$x
    \in M \setminus D_{\{g_k\}}$, and
  \item \label{item:7} $\sum_{k=1}^\infty d(g_k, g_{k+1}) < \infty$.
  \end{enumerate}
  We call $[g_\infty]$ the \emph{$\omega$-limit} of the sequence
  $\{g_k\}$ and write $g_k \overarrow{\omega} [g_\infty]$.

  More generally, if $\{g_k\}$ is a $d$-Cauchy sequence containing a
  subsequence that $\omega$-converges to $[g_\infty]$, then we say
  that $\{g_k\}$ \emph{$\omega$-subconverges} to $[g_\infty]$.
\end{dfn}

Condition (\ref{item:4}) in the definition is simply there for
convenience, so we don't have to repeatedly assume that a sequence is
$\omega$-convergent \emph{and} Cauchy.  Condition (\ref{item:7}) is
technical and will aid us in proofs.  Conceptually, it means that we
can find paths $\alpha_k$ connecting $g_k$ to $g_{k+1}$ such that the
concatenated path $\alpha_1 * \alpha_2 * \cdots$ has finite length.
If $\{g_k\}$ is $d$-Cauchy, then this can always be achieved by
passing to a subsequence.  (We remark here, however, that these two
conditions are not independent.  In fact, (\ref{item:7}) implies
(\ref{item:4}).)

Note that condition (\ref{item:6}) implies that if $g_k
\overarrow{\omega} [g_\infty]$, then for all $x \in M \setminus
D_{\{g_k\}}$ and all $\tilde{g}_\infty \in [g_\infty]$, there exists
some $\delta(x) > 0$ such that
\begin{equation*}
  \det g_k(x) \geq \delta(x)
\end{equation*}
for all $k \in \N$ and in every chart of an amenable atlas that
contains $x$.

We now move on to proving some properties of $\omega$-convergence.  We
first state an entirely trivial consequence of Definitions \ref{dfn:7}
and \ref{dfn:13}.

\begin{lem}\label{lem:29}
  Let $[g_\infty] \in \M$, and let $\{g_k\}$ be a sequence in $\M$.
  Suppose that for one given representative $g_\infty \in [g_\infty]$,
  $\{g_k\}$ together with $g_\infty$ satisfies conditions
  \eqref{item:4}--\eqref{item:7} of Definition \ref{dfn:13}.  Then
  these conditions are also satisfied for $\{g_k\}$ together with
  every other representative of $[g_\infty]$.

  Therefore, if can we verify these conditions for one representative
  of an equivalence class, this already implies $\{g_k\}
  \overarrow{\omega} [g_\infty]$.
\end{lem}

We can thus consistently say that $\{g_k\}$ $\omega$-converges to an
individual semimetric $g_\infty \in \Mm$ if the two together satisfy
conditions \eqref{item:4}--\eqref{item:7} of Definition \ref{dfn:13}.

The next property of $\omega$-convergence is obvious from property
(\ref{item:5}) of Definition \ref{dfn:13}.

\begin{lem}\label{lem:41}
  If $\{g^0_k\}$ and $\{g^1_k\}$ both $\omega$-converge to the same
  element $[g_\infty] \in \Mmhat$, then $\{g^0_k\}$ and $\{g^1_k\}$
  have the same deflated set, up to a nullset.
\end{lem}

Recall that the main goal of this section is to show that each Cauchy
sequence in $\M$ has an $\omega$-convergent subsequence.  To do this,
we will first prove a pointwise result in the following subsection.

\subsubsection{(Riemannian) Metrics on $\Matx$ Revisited}\label{sec:metr-matx-revis}

In order to more closely study the metric $\Theta_M$ on $\M$, we now take
a closer look at the Riemannian metrics $\langle \cdot , \cdot
\rangle$ and $\langle \cdot , \cdot \rangle^0$ (see Section \ref{sec:l2-metric}
and Definition \ref{dfn:15}, respectively) that we have defined on the
finite-dimensional manifold $\Matx$.  The relationship between the two
is quite simple:
\begin{equation}\label{eq:50}
  \langle h, k \rangle_{\tilde{g}}^0 = \langle h, k
  \rangle_{\tilde{g}} \det \tilde{G} \quad \textnormal{for all}\
  \tilde{g} \in \Matx\ \textnormal{and}\ h,k \in T_{\tilde{g}} \Matx \cong \satx.
\end{equation}
Thus, we will first study the simpler Riemannian metric $\langle \cdot
, \cdot \rangle$ and find out what properties of $\langle \cdot ,
\cdot \rangle^0$ we can deduce in this way.  Despite their close
relationship, their qualitative properties are very different---in
particular, $\Matx$ is complete with respect to $\langle \cdot, \cdot
\rangle$.  We will show this using a simplified version of the
analogous computations for $( \cdot , \cdot )$ on $\M$ carried out in
\cite[Thm.~2.3]{freed89:_basic_geomet_of_manif_of}.

Before we start, let's clear up some notation.

\begin{dfn}\label{dfn:11}
  By $d_x$, we denote the distance function induced on $\Matx$ by
  $\langle \cdot , \cdot \rangle$.  We denote
  the $\langle \cdot , \cdot \rangle$-length of a path $a_t$ in
  $\Matx$ by $L^{\langle \cdot , \cdot \rangle}(a_t)$ and the $\langle
  \cdot , \cdot \rangle^0$-length by $L^{\langle \cdot, \cdot
    \rangle^0}(a_t)$.
\end{dfn}

Now we compute the Christoffel symbols.

\begin{prop}\label{prop:13}
  Let $h$ and $k$ be constant vector fields on $\Matx$, and denote the
  Levi-Civita connection of $\langle \cdot , \cdot \rangle$ by
  $\nabla$.  Then the Christoffel symbols of $\langle \cdot , \cdot
  \rangle$ are given by
  \begin{equation*}
    \Gamma(h,k) = \nabla_h k |_{\tilde{g}} = - \frac{1}{2} \left( h \tilde{g}^{-1} k +
      k \tilde{g}^{-1} h \right).
  \end{equation*}
\end{prop}
\begin{proof}
  All computations are done at the base point $\tilde{g}$, which we
  will omit from the notation for convenience.  Let $\ell$ be any
  other constant vector field on $\Matx$.
  Using the Koszul formula, we can compute that
  \begin{equation}\label{eq:66}
    2 \langle \nabla_h k, \ell \rangle = h \langle k , \ell \rangle +
    k \langle \ell , h \rangle - \ell \langle h , k \rangle.
  \end{equation}

  Using the fact that the derivative of the map $\hat{g} \mapsto
  \hat{g}^{-1}$ at the point $\tilde{g}$ is given by $a \mapsto
  -\tilde{g}^{-1} a \tilde{g}^{-1}$, one can compute that
  \begin{equation*}
    h \langle k , \ell \rangle = -\tr\left( (\tilde{g}^{-1} h \tilde{g}^{-1} k) (\tilde{g}^{-1}
      \ell) \right) - \tr\left( (\tilde{g}^{-1} k) (\tilde{g}^{-1} h
      \tilde{g}^{-1} \ell) \right).
  \end{equation*}
  Repeating the same computation for the other permutations and
  substituting the results into \eqref{eq:66} then gives the result.
\end{proof}

Using this, it is a relatively simple matter to solve the geodesic
equation of $\langle \cdot , \cdot \rangle$.

\begin{prop}\label{prop:12}
  The geodesic $g_t$ in $(\Matx, \langle \cdot , \cdot \rangle)$ with
  initial data $g_0$, $g'_0$ is given by
  \begin{equation*}
    g_t = g_0 e^{t g_0^{-1} g'_0}.
  \end{equation*}
  In particular, $(\Matx, d_x)$ is a complete metric space.
\end{prop}
\begin{proof}
  Let $a_t := g'_t$.  Since $g_t$ is a geodesic, we have
  \begin{equation}\label{eq:67}
    0 = \nabla_{a_t} a_t = a'_t + \Gamma(a_t, a_t) = a'_t - a_t
    g_t^{-1} a_t
  \end{equation}
  by Proposition \ref{prop:13}.  Now, multiplying \eqref{eq:67} on the
  left by $g_t^{-1}$ gives
  \begin{equation*}
    (g_t^{-1} a_t)' = 0.
  \end{equation*}
  Thus $g_t^{-1} g'_t$ is constant, or $\log(g_t)' = g_t^{-1} g'_t
  \equiv g_0^{-1} g'_0$.  The geodesic equation now follows, and since
  $g_0 e^{t g_0^{-1} g'_0} \in \Matx$ for all $t \in \R$, the Hopf-Rinow
  theorem implies that $(\M_x, d_x)$ is complete.
\end{proof}

We now want to use Proposition \ref{prop:12} and \eqref{eq:50} to
characterize $\theta^g_x$-Cauchy sequences in $\Matx$.  First, though,
we need a lemma that is the pointwise version of Lemma \ref{lem:13}.
The proof is completely analogous to that of Lemma \ref{lem:13}, and
so we omit it.

\begin{lem}\label{lem:34}
  Let $a_0, a_1 \in \Matx$.  Then
  \begin{equation*}
    \left|
      \sqrt{\det A_1} - \sqrt{\det A_0}
    \right| \leq \frac{\sqrt{n}}{2} \theta^g_x (a_0, a_1).
  \end{equation*}
  (Recall Convention \ref{cvt:3} for the definitions of $A_i$.)
\end{lem}

\begin{prop}\label{prop:14}
  Let $a_k$ be a $\theta^g_x$-Cauchy sequence.  Then either
  \begin{enumerate}
  \item $\det A_k \rightarrow 0$ for $k \rightarrow \infty$, or
  \item there exist constants $C,\eta > 0$ such that $|(a_k)_{ij}|
    \leq C$ and $\det A_k \geq \eta$ for all $1 \leq i,j \leq n$ and
    $k \in \N$.
  \end{enumerate}
\end{prop}
\begin{proof}
  Keeping Lemma \ref{lem:34} in mind, it is more convenient to work
  with the square root of the determinant.  This is, of course,
  completely equivalent for our purposes.

  Now, by Lemma \ref{lem:34}, the map $a \mapsto \sqrt{\det A}$ is
  $\theta^g_x$-Lipschitz.  Since $a_k$ is $\theta^g_x$-Cauchy, $L :=
  \lim_{k \to \infty} \sqrt{\det A_k}$ is well-defined.

  If for every $\eta > 0$, there exists $k$ such that $\sqrt{\det A_k}
  \leq \eta$, then clearly $L = 0$.

  It remains to show that if there exist $i$ and $j$ such that for all
  $C > 0$, there is a $k$ such that $|(a_k)_{ij}| > C$, then
  $\sqrt{\det A_k} \rightarrow 0$.  We will assume that $L > 0$ and
  show a contradiction.

  Let's say that we are given $b_0, b_1 \in \Matx$ with $\det B_0,
  \det B_1 \geq \delta$.  Let
  \begin{align*}
    L_{-\delta} &:=
    \inf \left\{
      L^{\langle \cdot, \cdot \rangle^0}(b_t) \mid b_t\ \textnormal{is a path from $b_0$ to $b_1$ with}\ \det
      B_t \leq \delta/2\ \textnormal{for some}\ t \in (0,1)
    \right\}, \\
    L_{+\delta} &:=
    \inf \left\{
      L^{\langle \cdot, \cdot \rangle^0}(b_t) \mid b_t\ \textnormal{is a path from $b_0$ to $b_1$ with}\ \det
      B_t \geq \delta/2\ \textnormal{for all}\ t \in [0,1]
    \right\}.
  \end{align*}
  It is easy to see that $\theta^g_x(b_0, b_1) = \min (L_{-\delta},
  L_{+\delta})$.  Now let $b_t$ be a path as in the definition of
  $L_{-\delta}$, and assume $\tau \in (0,1)$ is such that $\det B_\tau
  \leq \delta/2$.  Then using Lemma \ref{lem:34}, we have
  \begin{align*}
    L^{\langle \cdot, \cdot \rangle^0}(b_t)     &\geq \frac{\sqrt{n}}{2} \left| \sqrt{\det B_0} - \sqrt{\det
        B_\tau} \right| + \frac{\sqrt{n}}{2} \left| \sqrt{\det B_1} -
      \sqrt{\det B_\tau}
    \right| \\
    &\geq \sqrt{n} \left( \sqrt{\delta} - \sqrt{\frac{\delta}{2}}
    \right) = \sqrt{n} \left( 1 - \frac{1}{\sqrt{2}} \right)
    \sqrt{\delta}.
  \end{align*}
  Therefore $L_{-\delta} \geq \sqrt{n} (1 - 1/\sqrt{2})
  \sqrt{\delta}$.  Then, if $b_t$ is a path as in the definition of
  $L_{+\delta}$, we have
  \begin{equation*}
    L(b_t) = \integral{0}{1}{\sqrt{\langle b'_t , b'_t \rangle^0}}{d
      t} = \integral{0}{1}{\sqrt{\langle b'_t , b'_t \rangle} \det B_t}{d t}
    \geq \sqrt{\frac{\delta}{2}} \integral{0}{1}{\sqrt{\langle b'_t ,
        b'_t \rangle}}{d t} \geq \sqrt{\frac{\delta}{2}} \, d_x(b_0, b_1).
  \end{equation*}
  This gives $L_{+\delta} \geq \sqrt{\delta / 2} \, d_x(b_0, b_1)$.
  Putting this together, we get that
  \begin{equation}\label{eq:65}
    \theta^g_x(b_0, b_1) \geq \min \{ \sqrt{n} (1 - 1/\sqrt{2}) \sqrt{\delta} ,
    \sqrt{\delta / 2} \, d_x(b_0, b_1) \}
  \end{equation}
  whenever $\det B_0, \det B_1 \geq \delta$.

  Now, let's apply the considerations of the last paragraph to the
  problem at hand.  Let $i$ and $j$ be, as above, the indices for which
  $|(a_k)_{ij}|$ is unbounded, and choose a subsequence, which we
  again denote by $a_k$, such that $|(a_k)_{ij}| \geq k$ for all $k
  \in \N$.  Passing to this subsequence does not change the limit
  $\lim_{k \to \infty} \sqrt{\det A_k}$.

  Next, choose $K \in \N$ such that $k \geq K$ implies $\sqrt{\det
    A_k} \geq L/2$ and $k,l \geq K$ implies $\theta^g_x(a_k, a_l) \leq
  \frac{1}{2} \sqrt{n} (1 - 1/\sqrt{2}) \sqrt{L/2}$.  The latter
  assumption is possible since $a_k$ is Cauchy.  By \eqref{eq:65}, if
  $k \geq K$, we also have
  \begin{equation*}
    \theta^g_x(a_K, a_k) \geq \min \{ \sqrt{n} (1 - 1 / \sqrt{2}) \sqrt{L/2} ,
    \sqrt{L / 4} \, d_x(a_K, a_k) \}.
  \end{equation*}
  But $\theta^g_x(a_K, a_k) \geq \sqrt{n} (1 - 1 / \sqrt{2})
  \sqrt{L/2}$ violates our assumptions on $K$. Furthermore, $d_x(a_K,
  a_k) \rightarrow \infty$ since $|(a_k)_{ij}| \rightarrow \infty$ and
  $(\Matx, \langle \cdot , \cdot \rangle)$ is complete.  Therefore, if
  $\theta^g_x(a_K, a_k) \geq \sqrt{L / 4} \, d_x(a_K, a_k)$ for all
  $k$, then our assumptions on $K$ are violated as well.  Thus we have
  achieved the desired contradiction.
\end{proof}

Since for every pair of constants $C, \eta > 0$, the set of elements
$\tilde{g}$ of $\Matx$ with $|\tilde{g}_{ij}| \leq C$ and $\det
\tilde{G} \geq \eta$ for all $1 \leq i,j \leq n$ is compact, we
immediately get the following corollary of Proposition \ref{prop:14}:

\begin{cor}\label{cor:10}
  Let $\{g_k\}$ be a $\theta^g_x$-Cauchy sequence.  Then either
  \begin{enumerate}
  \item $\det G_k \rightarrow 0$ for $k \rightarrow \infty$, or
  \item there exists an element $g_\infty \in \Matx$ such that $g_k
    \rightarrow g_\infty$, with convergence in the manifold topology
    of $\M_x$.
  \end{enumerate}
\end{cor}

This is essentially the pointwise equivalent of $\omega$-convergence.  In
the next subsection, we will globalize this result.  Before we do
that, though, we use this opportune moment to prove two last pointwise
results, which will be useful in Section
\ref{sec:uniqueness-ad-limit}.  The first is the pointwise analog of
Proposition \ref{prop:18}.  Again, the proof is analogous to that of
Proposition \ref{prop:18}, so we omit it.

\begin{prop}\label{prop:25}
  Let $\tilde{g}, \hat{g} \in \M_x$.  Then there exists a constant
  $C'(n)$, depending only on $n$, such that
  \begin{equation*}
    \theta^g_x(\tilde{g}, \hat{g}) \leq C'(n) \left( \sqrt{\det \tilde{G}} +
      \sqrt{\det \hat{G}} \right).
  \end{equation*}
\end{prop}

The last pointwise result we need combines Corollary \ref{cor:10} and
Proposition \ref{prop:25} to give a description of the completion of
the metric space $(\Matx, \theta^g_x)$.

\begin{thm}\label{thm:35}
  For any given $x \in M$, let $\textnormal{cl}(\Matx)$ denote the
  closure of $\Matx \subset \satx$ with regard to the natural
  topology.  Then $\cl(\Matx)$ consists of all positive semidefinite
  $(0,2)$-tensors at $x$.  Let us denote the boundary of $\Matx$, as a
  subspace of $\satx$, by $\partial \Matx$.

  Then the completion of $(\Matx, \theta^g_x)$ can be identified with
  the quotient of the space $\textnormal{cl}(\Matx)$ where $\partial
  \Matx$ has been identified to a single point.  The distance function
  is given by
  \begin{equation*}
    \theta^g_x(g_0, g_1) = \lim_{k \rightarrow
      \infty}\theta^g_x(g^0_k, g^1_k),
  \end{equation*}
  where $\{g^0_k\}$ and $\{g^1_k\}$ are any sequences in $\Matx$
  converging (in the topology of $\satx$) to $g_0$ and $g_1$,
  respectively.
\end{thm}
\begin{proof}
  Let $\{g_k\}$ be any sequence in $\Matx$.  By Corollary
  \ref{cor:10}, if $\{g_k\}$ is Cauchy then either $g_k \rightarrow
  g_\infty \in \Matx$ (with convergence in the topology of $\satx$),
  or $\det G_k \rightarrow 0$.  In fact, by the equivalence of the
  topologies on $\Matx$ inherited from $\satx$ and $\theta^g_x$, it is
  easy to see that the converse holds as well.  Furthermore, two
  Cauchy sequences in $\Matx$ that converge to distinct elements of
  $\Matx$ are inequivalent Cauchy sequences.

  By Proposition \ref{prop:25}, all
  sequences with $\det G_k \rightarrow 0$ are equivalent Cauchy
  sequences, and so they are identified in $\overline{(\Matx,
    \theta^g_x)}$.
\end{proof}

\subsubsection{The Existence Proof}\label{sec:limit-volume-sing}

We now wish to globalize Corollary \ref{cor:10} to characterize
$d$-Cauchy sequences, using Proposition \ref{prop:20} to reduce
questions about $d$ to questions about the simpler metric $\Theta_M$.

\begin{lem}\label{lem:35}
  Let $\{g_k\}$ be a Cauchy sequence in $\M$.  By passing to a
  subsequence if necessary, we can assume that
  \begin{equation*}
    \sum_{k=1}^\infty d(g_k, g_{k+1}) < \infty.
  \end{equation*}

  Then the following holds:
  \begin{equation*}
    \sum_{k=1}^\infty \Theta_M (g_k, g_{k+1}) < \infty.
  \end{equation*}
  Furthermore, define functions $\Omega$ and $\Omega_N$ for each $N
  \in \N$ by
  \begin{equation*}
    \Omega_N := \sum_{k=1}^N \theta^g_x(g_k(x), g_{k+1}(x)),
    \quad \Omega := \sum_{k=1}^\infty \theta^g_x(g_k(x), g_{k+1}(x)).
  \end{equation*}
  Then $\Omega$ is a.e.~finite, $\Omega \in L^1(M, g)$ and $\Omega_N
  \overarrow{L^1} \Omega$.  Furthermore, by definition, $\Omega_N$
  converges to $\Omega$ pointwise.
\end{lem}
\begin{proof}
  The first statement is clear, as is the statement that $\Omega_N
  \rightarrow \Omega$ pointwise.  So we move on to the other
  statements.

  Lemma \ref{lem:13} implies that $\sqrt{\Vol(M, g_k)}$ is a Cauchy
  sequence in $\R$.  Therefore it is bounded, and we can find a
  constant $V$ such that $\sqrt{\Vol(M, g_k)} \leq V$ for all $k$.
  Thus, by Proposition \ref{prop:20},
  \begin{equation*}
    \Theta_M(g_k, g_{k+1}) \leq 2 d(g_k, g_{k+1})
    \left( \frac{\sqrt{n}}{2} d(g_k, g_{k+1}) + V \right).
  \end{equation*}
  But for large $k$, since $\{g_k\}$ is Cauchy, we must have $d(g_k,
  g_{k+1}) \leq 1$, so
  \begin{equation*}
    \Theta_M(g_k, g_{k+1}) \leq \sqrt{n} \, d(g_k,
    g_{k+1})^2 + V d(g_k, g_{k+1}) \leq (\sqrt{n} + V) d(g_k, g_{k+1}).
  \end{equation*}
  The first statement is now immediate.

  We can then compute
  \begin{equation*}
    \integral{M}{}{\Omega}{\mu_g} = \integral{M}{}{\left( \sum_{k =
          1}^\infty \theta^g_x(g_k, g_{k+1}) \right)}{\mu_g} = 
    \sum_{k=1}^\infty \integral{M}{}{\theta^g_x(g_k, g_{k+1})}{\mu_g} = \sum_{k=1}^\infty \Theta_M(g_k, g_{k+1}) < \infty,
  \end{equation*}
  where we have used the monotone convergence theorem of Lebesgue and
  Levi \cite[Thm.~2.8.2]{bogachev07:_measur_theor} to exchange the
  infinite sum and the integral.  Finiteness follows from the first
  part of the lemma.  This proves that $\Omega$ is a.e.~finite and
  $\Omega \in L^1(M,g)$.  It remains to show that $\Omega_N
  \overarrow{L^1} \Omega$.  But this is now immediate from
  \cite[Thm.~8.5.1]{rana02:_introd_to_measur_and_integ}, which states
  that if $1 \leq p < \infty$, $f_i \rightarrow f$ a.e.~and
  $\norm{f_i}_p \rightarrow \norm{f}_p$, then $f_i \overarrow{L^p} f$.
\end{proof}

Using this lemma, we can
globalize Corollary \ref{cor:10}.

\begin{prop}\label{prop:15}
  Let $\{g_k\}$ be a Cauchy sequence in $\M$ such that
  \begin{equation*}
    \sum_{k=1}^\infty d(g_k, g_{k+1}) < \infty.
  \end{equation*}
  Then for a.e.~$x \in M$, $\{g_k(x)\}$ is $\theta^g_x$-Cauchy and
  either:
  \begin{enumerate}
  \item \label{item:12} $\det G_{t_k}(x) \rightarrow 0$ for $k \rightarrow \infty$, or
  \item \label{item:13} $g_k(x)$ is a convergent sequence in $\Matx$.
  \end{enumerate}
  Furthermore, \eqref{item:12} holds for a.e.~$x \in D_{\{g_k\}}$, and
  \eqref{item:13} holds for a.e.~$x \in M \setminus D_{\{g_k\}}$.
\end{prop}
\begin{proof}
  By our assumption, all the conclusions of Lemma \ref{lem:35} hold.
  In particular, $\Omega_N \rightarrow \Omega$ pointwise and $\Omega$
  is a.e.~finite.  Therefore, for a.e.~$x \in M$,
  \begin{equation}\label{eq:69}
    \sum_{k=1}^\infty \theta^g_x(g_k(x), g_{k+1}(x)) = \Omega(x) < \infty.
  \end{equation}
  From this, it is immediate that $\{g_k(x)\}$ is $\theta^g_x$-Cauchy.
  The remaining results follow from Corollary \ref{cor:10}.
\end{proof}

This proposition essentially delivers us the proof of the existence
result.

\begin{thm}\label{thm:39}
  For every Cauchy sequence $\{g_k\}$, there exists an element
  $[g_\infty] \in \Mmhat$ and a subsequence $\{g_{k_l}\}$ such that
  $\{g_{k_l}\}$ $\omega$-converges to $[g_\infty]$.

  Explicitly, $[g_\infty]$ is the unique equivalence class containing
  the element $g_\infty \in \Mm$ defined as follows.  At points $x \in
  M$ where $\{g_{k_l}(x)\}$ is $\theta^g_x$-Cauchy,
  \begin{enumerate}
  \item $g_\infty(x) := 0$ for $x \in D_{\{g_{k_l}\}}$ and
  \item $g_\infty(x) := \lim g_{k_l}(x)$ for $x \in M \setminus
    D_{\{g_{k_l}\}}$.
  \end{enumerate}
  At points $x \in M$ where $\{g_{k_l}(x)\}$ is not
  $\theta^g_x$-Cauchy, we set $g_\infty(x) := 0$.
\end{thm}
\begin{proof}
  Let $\{g_{k_l}\}$ be a subsequence of $\{g_k\}$ such that
  \begin{equation*}
    \sum_{l=1}^\infty d(g_{k_l}, g_{k_l+1}) < \infty.
  \end{equation*}
  Then $\{g_{k_l}\}$ satisfies properties (\ref{item:4}) and
  (\ref{item:7}) of Definition \ref{dfn:13}, as well as the hypotheses
  of Corollary \ref{prop:15}.  Thus $\{g_{k_l}\}$ is
  a.e.~$\theta^g_x$-Cauchy, and so $g_\infty$ is defined a.e.~by the
  two conditions given above.  From this, it is immediate that
  $\{g_{k_l}\}$ together with $g_\infty$ also satisfies properties
  (\ref{item:5}) and (\ref{item:6}) of Definition \ref{dfn:13}.  Thus,
  $\{g_{k_l}\}$ $\omega$-converges to $g_\infty$, and by Lemma
  \ref{lem:29} it therefore $\omega$-converges to
  $[g_\infty]$---provided we can show that $g_\infty \in \Mm$.  But
  $g_\infty$ is clearly a semimetric, and it is the pointwise limit of
  the measurable semimetrics $\chi(M \setminus D_{\{g_{k_l}\}})
  g_{k_l}$ (one can easily construct the set $D_{\{g_{k_l}\}}$ from
  countable unions and intersections of open sets, so it is
  measurable).  Therefore $g_\infty$ itself is measurable.
\end{proof}

Knowing now that the $\omega$-limit of a Cauchy sequence of $\M$
exists (after passing to a subsequence), we go further into the
properties of $\omega$-convergence.

\subsection{$\omega$-Convergence and the Concept of
  Volume}\label{sec:ad-conv-conc}

In this brief subsection, we wish to prove that the volumes of
measurable subsets behave well under $\omega$-convergence.
Specifically, we want to show that if $\{g_k\}$ $\omega$-converges to
$[g_\infty]$ and $Y \subseteq M$ is measurable, then for any
representative $g_\infty \in [g_\infty]$,
\begin{equation}\label{eq:139}
  \Vol(Y, g_k) \rightarrow \Vol(Y, g_\infty).
\end{equation}
To see that the above expression is well-defined,
note that given any two
representatives $g^0_\infty, g^1_\infty \in [g_\infty]$, we have that
$\mu_{g^0_\infty} = \mu_{g^1_\infty}$ as measures---it is clear from
Definition \ref{dfn:7} that $\mu_{g^0_\infty}$ and $\mu_{g^1_\infty}$
can differ at most on a nullset.  Thus $\Vol(Y, g^0_\infty) = \Vol(Y,
g^1_\infty)$.

The proof of \eqref{eq:139} is achieved via the Lebesgue dominated
convergence theorem with the help of the next two lemmas.

\begin{lem}\label{lem:52}
  Let $\{g_k\}$ $\omega$-converge to $g_\infty \in \Mm$.  Then
  \begin{equation*}
    \left( \frac{\mu_{g_k}}{\mu_g} \right) \overarrow{\textnormal{a.e.}} \left(
      \frac{\mu_{g_\infty}}{\mu_g} \right).
  \end{equation*}
\end{lem}
\begin{proof}
  We first prove that for a.e.~$x \in D_{\{g_k\}}$,
  \begin{equation*}
    \left( \frac{\mu_{g_k}}{\mu_g} \right) = \det G_k(x) \rightarrow 0 = \det G_\infty = \left( \frac{\mu_{g_\infty}}{\mu_g} \right)
  \end{equation*}
  as $k \rightarrow \infty$.  By the definition of the
  deflated set, for every $x \in D_{\{g_k\}}$ and $\epsilon > 0$,
  there exists $k \in \N$ such that $\det G_k(x) < \epsilon$.
  But we also know from Proposition \ref{prop:15} and property
  (\ref{item:7}) of Definition \ref{dfn:13} that $\{g_k(x)\}$ is
  $\theta^g_x$-Cauchy for a.e.~$x \in M$.  Hence, by Lemma
  \ref{lem:34}, $\left\{\sqrt{\det G_k(x)}\right\}$ is a Cauchy
  sequence in $\R$ at such points.  Therefore it has a limit, and this
  limit must be $0$.

  Now, for a.e.~$x \in M \setminus D_{\{g_k\}}$, $g_k(x) \rightarrow
  g_\infty(x)$.  Since the determinant is a continuous map from the
  space of $n \times n$ matrices into $\R$, this immediately implies
  that $\det G_k(x) \rightarrow \det G_\infty(x)$ for a.e.~$x \in M
  \setminus D_{\{g_k\}}$.
\end{proof}

Our next task is to find an $L^1$ function that dominates
$(\mu_{g_k} / \mu_g)$.

\begin{lem}\label{lem:36}
  Let $\{g_k\}$ be a Cauchy sequence such that
  \begin{equation*}
    \sum_{k=1}^\infty d(g_k, g_{k+1}) < \infty,
  \end{equation*}
  and let $\Omega$ be the function of Lemma \ref{lem:35}.  Then
  \begin{equation*}
    \left( \frac{\mu_{g_k}}{\mu_g} \right)(x) \leq \frac{\sqrt{n}}{2}
    \Omega(x) + \left( \frac{\mu_{g_1}}{\mu_g} \right)(x)
  \end{equation*}
  for a.e.~$x \in M$ and all $k \in \N$.
\end{lem}
\begin{proof}
  Fix some $k$ for the moment.  By Proposition \ref{prop:15},
  $\{g_k(x)\}$ is $\theta_x^g$-Cauchy for a.e.~$x \in M$.  Let $x \in
  M$ be a point where this holds.  Then by Lemma \ref{lem:34}, the
  triangle inequality, and the definitions of $\Omega_N$ and $\Omega$,
  we have
  \begin{align*}
    \left| \sqrt{\det G_k} - \sqrt{\det G_1} \right| &\leq
    \frac{\sqrt{n}}{2} \theta^g_x(g_k, g_1) \leq
    \frac{\sqrt{n}}{2} \sum_{m=1}^{k-1} \theta^g_x(g_m,
    g_{m+1}) \\
    &= \frac{\sqrt{n}}{2} \Omega_{k-1}(x) \leq \frac{\sqrt{n}}{2}
    \Omega(x).
  \end{align*}
  The result is now immediate.
\end{proof}

Now, since $\mu_{g_1}$ is smooth, it has finite volume, implying that
$(\mu_{g_1} / \mu_g) \in L^1(M, g)$.  We have already seen in Lemma
\ref{lem:35} that $\Omega \in L^1(M, g)$.  Therefore Lemmas
\ref{lem:52} and \ref{lem:36} allow us to apply the Lebesgue dominated
convergence theorem to obtain:

\begin{thm}\label{thm:19}
  Let $\{g_k\}$ $\omega$-converge to $g_\infty \in \Mm$, and let $Y
  \subseteq M$ be any measurable subset.  Then $\Vol(Y, g_k)
  \rightarrow \Vol(Y, g_\infty)$.
\end{thm}

An immediate corollary of this theorem and Lemma \ref{lem:13} is that
the total volume of the $\omega$-limit is finite:

\begin{cor}\label{cor:15}
  If $g_\infty$ is the $\omega$-limit of a sequence $\{g_k\}$ in $\M$, then
  $\Vol(M, g_\infty) < \infty$.  That is, $g_\infty \in \M_f$.
\end{cor}

Furthermore, as we might have suspected from the beginning, the volume
of the deflated set $D_{\{g_k\}}$ of an $\omega$-convergent sequence
vanishes in the limit.  This is because $\Vol(D_{\{g_k\}}, g_\infty) =
0$.

\begin{cor}\label{cor:17}
  Let $\{g_k\}$ $\omega$-converge to $g_\infty \in \M_f$.  Then
  $\Vol(D_{\{g_k\}}, g_k) \rightarrow 0$.
\end{cor}

\fussy Since we now know the volume of an $\omega$-limit is finite, we can
refine Theorem \ref{thm:39}:

\begin{thm}\label{thm:40}
  For every Cauchy sequence $\{g_k\}$, there exists an element
  $[g_\infty] \in \Mfhat$ such that $\{g_k\}$ $\omega$-subconverges to
  $[g_\infty]$.
\end{thm}

\subsection{Uniqueness of the $\omega$-Limit}\label{sec:uniqueness-ad-limit}

The goal of this section is to prove the uniqueness of the
$\omega$-limit in the sense mentioned in the introduction to the
chapter: we will show that two $\omega$-convergent Cauchy sequences in
$\M$ are equivalent if and only if they have the same $\omega$-limit.
We prove each direction in a separate subsection.

\subsubsection{First Uniqueness Result}\label{sec:first-uniq-result}

We first prove the statement that if two $\omega$-con\-ver\-gent
Cauchy sequences are equivalent, then their $\omega$-limits agree.  To
do so, we will extend the pseudometric $\Theta_Y$ (cf.~Definition
\ref{dfn:16}) to the precompletion of $\M$.  For this, we need an easy
lemma.

\begin{lem}\label{lem:10}
  Let $Y \subseteq M$ be measurable.  If $\{g_k\}$ is a $d$-Cauchy
  sequence, then it is also $\Theta_Y$-Cauchy.
\end{lem}
\begin{proof}
  As noted in the proof of Lemma \ref{lem:35}, since $\{g_k\}$ is
  $d$-Cauchy, the sequence $\sqrt{\Vol(M, g_k)}$ in $\R$ is bounded,
  so Proposition \ref{prop:20} gives the result easily.
\end{proof}

Now we give the extension of $\Theta_Y$ mentioned above.

\begin{prop}\label{prop:21}
  Let $Y \subseteq M$ be measurable.  Then the pseudometric $\Theta_Y$
  on $\M$ can be extended to a pseudometric on
  $\overline{\M}^{\textnormal{pre}}$, the precompletion of $\M$, via
  \begin{equation}\label{eq:14}
    \Theta_Y(\{g^0_k\}, \{g^1_k\}) := \lim_{k \rightarrow \infty}
    \Theta_Y(g^0_k, g^1_k)
  \end{equation}
  This pseudometric is weaker than $d$ in the sense that $d(\{g^0_k\},
  \{g^1_k\}) = 0$ implies $\Theta_Y(\{g^0_k\}, \{g^1_k\}) = 0$ for any
  Cauchy sequences $\{g^0_k\}$ and $\{g^1_k\}$.  More precisely, we
  have
  \begin{equation}\label{eq:91}
    \Theta_Y (\{g^0_k\}, \{g^1_k\}) \leq d(\{g^0_k\}, \{g^1_k\}) \left( \sqrt{n}\, d(\{g^0_k\},
      \{g^1_k\}) + 2 \sqrt{\Vol(M, g_0)} \right),
  \end{equation}
  where $g_0$ is any element of $\Mf$ with $g^0_k \overarrow{\omega} [g_0]$.

  Furthermore, if $\{g^0_k\}$ and $\{g^1_k\}$ are sequences in $\MV$
  that $\omega$-converge to $g_0$ and $g_1$, respectively, then the formula
  \begin{equation}\label{eq:92}
    \Theta_Y(\{g^0_k\}, \{g^1_k\}) = \integral{Y}{}{\theta_x^g(g_0(x), g_1(x))}{\mu_g(x)}
  \end{equation}
  holds for all $g_0, g_1 \in \MV$.
\end{prop}

\begin{rmk}\label{rmk:26}
  In \eqref{eq:91}, we choose any $\omega$-limit of $\{g^0_k\}$.  The
  existence of such a limit has already been proved, but not its
  uniqueness.  On the other hand, if $\tilde{g}_0$ is a different
  $\omega$-limit of $\{g^0_k\}$, Theorem \ref{thm:19} guarantees that
  $\Vol(M, \tilde{g}_0) = \Vol(M, g_0)$.  Therefore, the estimate
  \eqref{eq:91} is independent of the choice of $\omega$-limit.
\end{rmk}

\begin{proof}[Proof of Proposition \ref{prop:21}]
  The construction of a pseudometric on the precompletion of a metric
  space can be carried over to the case where we begin with a
  pseudometric space.  Therefore, the limit in \eqref{eq:14} is
  well-defined due to the fact that $\{g^0_k\}$ and $\{g^1_k\}$ are
  Cauchy sequences with respect to $\Theta_Y$, and \eqref{eq:14}
  indeed defines a pseudometric.
  
  The inequality \eqref{eq:91} is proved via the following simple
  computation, which uses \eqref{eq:14}, Proposition \ref{prop:20},
  and Theorem \ref{thm:19}:
  \begin{align*}
    \Theta_Y(\{g^0_k\}, \{g^1_k\}) &= \lim_{k \rightarrow \infty}
    \Theta_Y(g^0_k, g^1_k) \leq \lim_{k \rightarrow \infty} d(g^0_k, g^1_k) \left(
      \sqrt{n}\, d(g^0_k,
      g^1_k) + 2 \sqrt{\Vol(M, g^0_k)} \right) \\
    &= d(\{g^0_k\}, \{g^1_k\}) \left( \sqrt{n}\, d(\{g^0_k\},
      \{g^1_k\}) + 2 \sqrt{\Vol(M, g_0)} \right).
  \end{align*}

  As for the last statement, note first that $\theta^g_x(g_0(x),
  g_1(x))$ is well-defined by Theorem \ref{thm:35}, since $g_0$ and
  $g_1$ are positive semidefinite tensors at each point $x \in M$.  To
  prove \eqref{eq:92}, we will first use Fatou's Lemma to show that
  $\theta^g_x(g_0(x), g_1(x))$ is integrable.  We will then use this
  to apply the Lebesgue dominated convergence theorem.

  By Proposition \ref{prop:15}, for a.e.~$x \in M$, $\{g^0_k(x)\}$ and
  $\{g^1_k(x)\}$ are $\theta^g_x$-Cauchy.  At such points, by
  definition,
  \begin{equation}\label{eq:89}
    \theta^g_x(g_0(x), g_1(x)) = \lim_{k \rightarrow \infty}
    \theta^g_x(g^0_k(x), g^1_k(x)).
  \end{equation}
  So defining 
  \begin{equation*}
    f_k(x) := \theta^g_x(g^0_k(x), g^1_k(x)), \quad f(x) := \theta^g_x(g_0(x),
    g_1(x)),
  \end{equation*}
  we have $f_k \rightarrow f$ a.e.

  Now, note that
  \begin{equation*}
    \Theta_Y(g^0_k, g^1_k) = \integral{Y}{}{f_k(x)}{\mu_g(x)}.
  \end{equation*}
  We have already seen that $\lim_{k \rightarrow \infty}
  \Theta_Y(g^0_k, g^1_k)$ exists, so $\{\Theta_Y(g^0_k, g^1_k)\}$ is
  in particular a bounded sequence of real numbers.  Thus
  \begin{equation*}
    \sup_k \integral{Y}{}{f_k(x)}{\mu_g(x)} = \sup_k \Theta_Y(g^0_k,
    g^1_k) < \infty,
  \end{equation*}
  where we have used Fatou's lemma.

  Now we wish to verify the assumptions of the Lebesgue dominated
  convergence theorem for $f_k$ and $f$.  We note that for each $l >
  k$, the triangle inequality gives
  \begin{equation*}
    f_k(x)     \leq \sum_{m=k}^{l-1}
    \theta^g_x(g^0_m(x), g^0_{m+1}(x)) + \theta^g_x(g^0_l(x),
    g^1_l(x)) + \sum_{m=k}^{l-1} \theta^g_x(g^1_m(x), g^1_{m+1}(x)) 
  \end{equation*}
  Starting the sum above at $m = 1$ instead of $m = k$ and taking the
  limit $l \rightarrow \infty$ gives, for a.e.~$x \in M$,
  \begin{equation*}
    f_k(x) \leq \sum_{m=1}^\infty
    \theta^g_x(g^0_m(x), g^0_{m+1}(x)) + f(x) + \sum_{m=1}^\infty \theta^g_x(g^1_m(x), g^1_{m+1}(x)),
  \end{equation*}
  where we have used \eqref{eq:89}.  Now we claim that the right-hand
  side of the above inequality is $L^1$-integrable.  We already showed
  $f$ is integrable using Fatou's Lemma.  As for the two infinite
  sums, they are each also integrable by Lemma \ref{lem:35} and
  $\omega$-convergence of $g^i_k$, $i=0,1$ (specifically, property
  (\ref{item:7}) of Definition \ref{dfn:13} and Lemma \ref{lem:35}).
  Thus each $f_k$ is bounded a.e.~by an $L^1$ function not depending
  on $k$.
  
  Knowing all of this, we can apply the Lebesgue dominated convergence
  theorem to show
  \begin{equation*}
    \Theta_Y(\{g^0_k\}, \{g^1_k\}) = \lim_{k \rightarrow \infty} \Theta_Y(g^0_k, g^1_k) = \lim_{k
      \rightarrow \infty} \integral{Y}{}{f_k}{\mu_g} =
    \integral{Y}{}{f}{\mu_g} = \integral{Y}{}{\theta^g_x(g_0(x),
      g_1(x))}{\mu_g(x)},
  \end{equation*}
  which completes the proof.
\end{proof}

With this proposition, proving the first uniqueness result becomes a
relatively simple matter.

\begin{thm}\label{thm:20}
  Let two $\omega$-convergent sequences $\{g^0_k\}$ and $\{g^1_k\}$,
  with $\omega$-limits $[g_0]$ and $[g_1]$, respectively, be given.
  If $g^0_k$ and $g^1_k$ are equivalent, i.e., if
  \begin{equation*}
    \lim_{k \rightarrow \infty} d(g^0_k, g^1_k) = 0,
  \end{equation*}
  then $[g_0] = [g_1]$.
\end{thm}
\begin{proof}
  Suppose the contrary; then for any representatives $g_0 \in [g_0]$
  and $g_1 \in [g_1]$, one of two possibilities holds:
  \begin{enumerate}
  \item \label{item:14} $X_{g_0}$ and $X_{g_1}$ differ by a set of positive measure,
    or
  \item \label{item:15} $X_{g_0} = X_{g_1}$, up to a nullset, but $g_0$ and $g_1$
    differ on a set $E$ with $E \cap (X_{g_0} \cup X_{g_1}) =
    \emptyset$ and $\Vol(E, g) > 0$, where $g$ is our fixed metric.
  \end{enumerate}
  We will show that neither of these possibilities can actually occur.

  To rule out \eqref{item:14}, let $X_i := D_{\{g^i_k\}}$ denote the
  deflated set of the sequence $\{g_k^i\}$ for $i=0,1$.  Then we claim
  $X_0 = X_1$, up to a nullset.  If this is not true, then by swapping
  the two sequences if necessary, we see that $Y := (X_0 \setminus
  X_1)$ has positive volume with respect to $g_1$ and zero volume with
  respect to $g_0$.  ($Y$ is simply the set on which $\{g^0_k\}$
  deflates and $\{g^1_k\}$ doesn't.)  But then by Lemma \ref{lem:13},
  \begin{equation*}
    \lim_{k \rightarrow \infty} d(g^0_k, g^1_k) \geq \lim_{k
      \rightarrow \infty} \sqrt{\Vol(Y, g^1_k)} = \sqrt{\Vol(Y, g_1)}
    > 0,
  \end{equation*}
  where we have used Theorem \ref{thm:19}.  This contradicts the
  assumptions of the theorem, so in fact $X_0 = X_1$ up to a nullset.
  Since by property (\ref{item:5}) of Definition \ref{dfn:13} $X_{g_i}
  = D_{\{g^i_k\}}$ up to a nullset as well, \eqref{item:14} cannot
  hold.

  So suppose that \eqref{item:15} holds.  Note that on $E$, $g_0$ and
  $g_1$ are both positive definite.  Since $E$ has positive
  $g$-volume, we can conclude from Proposition \ref{prop:21}
  (specifically \eqref{eq:92}) that $\Theta_E(\{g^0_k\}, \{g^1_k\}) >
  0$.  But then this and \eqref{eq:91} also imply that
  \begin{equation*}
    \lim_{k \rightarrow \infty} d(g^0_k, g^1_k) = d(\{g^0_k\},
    \{g^1_k\}) > 0.
  \end{equation*}
  This contradicts the assumptions of the theorem, and so
  \eqref{item:15} cannot hold either.
\end{proof}

\subsubsection{Second Uniqueness Result}\label{sec:second-uniq-result}

Our goal in this subsection is to prove the following statement: up to
equivalence, there is only one $d$-Cauchy sequence $\omega$-converging
to a given element of $\Mfhat$.  That is, if we have two sequences
$\{g^0_k\}, \{g^1_k\}$ that both $\omega$-converge to the same
$[g_\infty] \in \Mfhat$, then
\begin{equation*}
  d(\{g^0_k\}, \{g^1_k\}) = \lim_{k \to \infty} d(g^0_k, g^1_k) = 0.
\end{equation*}

We will first prove the above statement for sequences that remain
within a given amenable subset $\U$, and will then use this to extend
the proof to arbitrary sequences.

\begin{prop}\label{prop:1}
  Let $\U$ be an amenable subset, and let $\U^0$ be the
  $L^2$-completion of $\U$.  If two sequences $\{g^0_k\}$ and
  $\{g^1_k\}$ in $\U$ both $\omega$-converge to $[g_\infty] \in \Mfhat$,
  then $\{g^0_k\}$ and $\{g^1_k\}$ are equivalent.  That is,
  \begin{equation*}
    \lim_{k \to \infty} d(g^0_k, g^1_k) = 0.
  \end{equation*}
  Furthermore, up to differences on a nullset, $[g_\infty]$ only
  contains one representative, $g_\infty$, and $\{g^0_k\}$ and
  $\{g^1_k\}$ both $L^2$-converge to $g_\infty$.  In particular,
  $g_\infty \in \U^0$.
\end{prop}
\begin{proof}
  Note that Definition \ref{dfn:2} of an amenable subset implies that
  the deflated sets of $\{g^0_k\}$ and $\{g^1_k\}$ are empty.
  Therefore, all representatives of $[g_\infty]$ differ at most by a
  nullset, and property (\ref{item:6}) of Definition \ref{dfn:13}
  implies that $g^0_k, g^1_k \overarrow{\textnormal{a.e.}} g_\infty$.
  
  Since all $g^0_k$ and $g^1_k$ satisfy the same bounds a.e.~in each
  coordinate chart, it is easy to see that the set
  \begin{equation*}
    \{ | (g_l^k)_{ij} |^2 \mid 1 \leq i,j \leq n,\ k \in \N \}
  \end{equation*}
  is equicontinuous at $\emptyset$
  \cite[Dfn.~8.5.2]{rana02:_introd_to_measur_and_integ} in each
  coordinate chart for both $l = 0$ and $l = 1$.  Therefore,
  we can conclude from \cite[Thm.~8.5.14,
  Thm.~8.3.3]{rana02:_introd_to_measur_and_integ} that $\{g^0_k\}$ and
  $\{g^1_k\}$ converge in $L^2$ to $g_\infty$, proving the second
  statement.  This also implies that $\lim_{k \to \infty} \norm{g^1_k
    - g^0_k}_{g} = 0$.
  But now, invoking Theorem \ref{thm:5} gives $\lim_{k \to \infty}
  d(g^0_k, g^1_k) = 0$.
\end{proof}

The next lemma establishes the strong correspondence between $L^2$-
and $\omega$-convergence within amenable subsets.

\begin{lem}\label{lem:43}
  Let $\U \subset \M$ be amenable, and let $\tilde{g} \in \U^0$.  Then
  for any sequence $\{ g_k \}$ in $\U$ that $L^2$-converges to
  $\tilde{g}$, there exists a subsequence $\{ g_{k_l} \}$ that
  $\omega$-converges to $\tilde{g}$.

  In particular, for any element $\tilde{g} \in \U^0$, we can always
  find a sequence in $\U$ that both $L^2$- and $\omega$-converges to
  $\tilde{g}$.
\end{lem}
\begin{proof}
  Let $\{ g_k \}$ be any sequence $L^2$-converging to $\tilde{g} \in
  \U^0$.  Then $\tilde{g}$ together with any subsequence of $\{ g_k
  \}$ already satisfies properties (\ref{item:4}) and (\ref{item:5})
  of Definition \ref{dfn:13}.  This is clear from Theorem \ref{thm:5}
  and Definition \ref{dfn:2} of an amenable subset.  (Property
  (\ref{item:5}) is empty here, as $\{g_k\}$ has empty deflated set
  by the definition of an amenable subset.)  Since $\{ g_k \}$ is
  $d$-Cauchy by Theorem \ref{thm:5}, it is also easy to see that there
  is a subsequence $\{ g_{k_m} \}$ of $\{ g_k \}$ satisfying property
  (\ref{item:7}) of $\omega$-convergence.

  To verify property (\ref{item:6}), note that $L^2$-convergence of
  $\{ g_{k_m} \}$ implies that there exists a subsequence
  $\{g_{k_l}\}$ of $\{g_{k_m}\}$ that converges to $\tilde{g}$
  a.e. \cite[Thm.~8.5.14,
  Thm.~8.3.6]{rana02:_introd_to_measur_and_integ}.
\end{proof}

Given the results that we have so far, we can give an alternative
description of the completion of an amenable set using $\omega$-convergence
instead of $L^2$-convergence.

\begin{prop}\label{prop:26}
  Let $\U \subset \M$ be an amenable subset.  Then the completion
  $\overline{\U}$ of $\U$ as a metric subspace of $\M$ can be
  identified with $\U^0$, the $L^2$ completion of $\U$, using
  $\omega$-convergence.  That is, there is a natural bijection between
  $\overline{\U}$ and $\U^0$ given by identifying each equivalence
  class of Cauchy sequences $[\{g_k\}]$ with the unique element of
  $\U^0$ that they $\omega$-subconverge to.
\end{prop}
\begin{proof}
  The existence result (Theorem \ref{thm:39}) the first uniqueness
  result (Theorem \ref{thm:20}) and Proposition \ref{prop:1} together
  imply that for every equivalence class $[\{g_k\}]$ of $d$-Cauchy
  sequences in $\U$, there is a unique $L^2$ metric $g_\infty \in
  \U^0$ such that every representative of $[\{g_k\}]$
  $\omega$-subconverges to $g_\infty$, and that the representatives of
  a different equivalence class cannot also $\omega$-subconverge to
  $g_\infty$.  This gives us the map from $\overline{\U}$ to $\U^0$
  and shows that it is injective.  Furthermore, by Lemma \ref{lem:43},
  there is a sequence in $\U$ $\omega$-subconverging to every element
  of $\U^0$.  Thus, this map is also surjective.
\end{proof}

With this identification, we can define a metric on $\U^0$ by
declaring the bijection of the previous proposition to be an
isometry.  The result is the following:

\begin{dfn}\label{dfn:14}
  Let $\U$ be an amenable subset.  By $d_\U$, we denote the metric on
  the completion of $\U$, which we identify with the $L^2$-completion
  $\U^0$ via Proposition \ref{prop:26}.  Thus, for $g_0, g_1 \in \U^0$
  and any sequences $g^0_k \overarrow{\omega} g_0$, $g^1_k \overarrow{\omega}
  g_1$, we have
  \begin{equation*}
    d_\U(g_0, g_1) = \lim_{k \rightarrow \infty} d(g^0_k, g^1_k).
  \end{equation*}
  Note that by the preceding results, we can equivalently define
  $d_\U$ by assuming that $\{g^0_k\}$ and $\{g^1_k\}$ $L^2$-converge
  to $g_0$ and $g_1$, respectively.
\end{dfn}

The next lemma, the proof of which is immediate, shows that the metric
$d_\U$ is nicely compatible with the metric $d$.

\begin{lem}\label{lem:42}
  Let $\U \subset \M$ be amenable, and suppose $g_0, g_1 \in \U$ and
  $g_2 \in \U^0$.  Then
  \begin{enumerate}
  \item \label{item:16} $d(g_0, g_1) = d_\U(g_0, g_1)$, and 
  \item \label{item:3} $d(g_0, g_1) \leq d_\U(g_0, g_2) + d_\U(g_2, g_1)$.
  \end{enumerate}
\end{lem}

With a little bit of effort, we can use previous results to extend
Proposition \ref{prop:18}, a statement about $\M$, to the completion
of an amenable subset.  We first prove a very special case in a lemma,
followed by the full result.

\begin{lem}\label{lem:51}
  Let $\U$ be any amenable subset and $g^0, g^1 \in \U$.
  Let $C(n)$ be the constant of Proposition \ref{prop:18}, and let $E
  \subseteq M$ be measurable.  Then
  \begin{equation*}
    d_\U(g^0, \chi(M \setminus E) g^0 + \chi(E) g^1)
    \leq C(n) \left( \sqrt{\Vol(E, g^0)} + \sqrt{\Vol(E,
        g^1)} \right)
  \end{equation*}
\end{lem}
\begin{proof}
  For each $k \in \N$, choose closed subsets $F_k$ and open subsets
  $U_k$ such that $F_k \subseteq E \subseteq U_k$ and $\Vol(U_k, g) -
  \Vol(F_k, g) \leq 1/k$.  Furthermore, choose functions $f_k \in
  C^\infty(M)$ satisfying
  \begin{enumerate}
  \item $0 \leq f_k(x) \leq 1$ for all $x \in M$,
  \item $f_k(x) = 1$ for $x \in F_k$ and
  \item $f_k(x) = 0$ for $x \not\in U_k$.
  \end{enumerate}
  Then it is not hard to see that the sequence defined by
  \begin{equation*}
    g_k := (1 - f_k) g^0 + f_k g^1
  \end{equation*}
  $L^2$-converges to $\chi(M \setminus E) g^0 + \chi(E)
  g^1$, so in particular
  \begin{equation}\label{eq:87}
    d_\U(g^0, \chi(M \setminus E) g^0 + \chi(E) g^1) =
    \lim_{k \rightarrow \infty} d(g^0, g_k).
  \end{equation}
  Furthermore, since $g^0$ and all $g_k$ are smooth, Proposition
  \ref{prop:18} gives
  \begin{equation}\label{eq:74}
    d(g^0, g_k) \leq C(n) \left( \sqrt{\Vol(U_k, g^0)} +
      \sqrt{\Vol(U_k, g_k)} \right).
  \end{equation}
  By our assumptions on the sets $U_k$, it is clear that $\Vol(U_k,
  g^0) \rightarrow \Vol(E, g^0)$.  So if we can show that
  $\Vol(U_k, g_k) \rightarrow \Vol(E, g^1)$, then \eqref{eq:87}
  and \eqref{eq:74} combine to give the desired result.

  Now, because $g_k = g^1$ on $F_k$, we have
  \begin{equation*}
    \Vol(U_k, g_k) = \integral{F_k}{}{}{\mu_{g^1}} +
    \integral{U_k \setminus F_k}{}{}{\mu_{g_k}}.
  \end{equation*}
  The first term converges to $\Vol(E, g^1)$ for $k \rightarrow
  \infty$ by the definition of $F_k$.  We claim that the second term
  converges to zero.  Note that since the bounds of Definition
  \ref{dfn:2} are pointwise convex, we can enlarge $\U$ to an amenable
  subset containing $g_k$ for each $k \in \N$.  (By the definition,
  each $g_k$ is, at each point $x \in M$, a sum $(1-s) g^0(x) + s
  g^1(x)$ with $0 \leq s \leq 1$.)  Therefore, by Lemma \ref{lem:49},
  there exists a constant $K$ such that
  \begin{equation*}
    \left( \frac{\mu_{g_k}}{\mu_g} \right) \leq K.
  \end{equation*}
  But using this, our claim is clear from the assumptions on $U_k$ and
  $F_k$.
\end{proof}

\begin{thm}\label{thm:15}
  Let $\U$ be any amenable subset with $L^2$-completion $\U^0$.
  Suppose that $g_0, g_1 \in \U^0$, and let $E := \carr (g_1 - g_0) =
  \{ x \in M \mid g_0(x) \neq g_1(x) \}$.  Then there exists a
  constant $C(n)$ depending only on $n = \dim M$ such that
  \begin{equation*}
    d_\U (g_0, g_1) \leq C(n) \left( \sqrt{\Vol(E, g_0)} +
      \sqrt{\Vol(E,g_1)} \right).
  \end{equation*}
  In particular, we have
  \begin{equation*}
    \diam_\U \left( \{ \tilde{g} \in \U^0 \mid \Vol(M, \tilde{g}) \leq
      \delta \} \right) \leq 2 C(n) \sqrt{\delta}.
  \end{equation*}
\end{thm}
\begin{proof}
  Using Lemma \ref{lem:43}, choose any two sequences $\{g^0_k\}$ and
  $\{g^1_k\}$ in $\U$ that both $L^2$- and $\omega$-converge to $g_0$
  and $g_1$, respectively.  Then by the triangle inequality and Lemma
  \ref{lem:42}\eqref{item:16}, for each $k \in \N$,
  \begin{equation}\label{eq:68}
    d_\U(g_0, g_1) \leq d_\U(g_0, g^0_k) + d(g^0_k, g^1_k) + d_\U(g^1_k, g_1).
  \end{equation}
  By Theorem \ref{thm:5}, the first and last terms above approach zero
  as $k \rightarrow \infty$.  Furthermore, we claim that the middle
  term satisfies
  \begin{equation*}
    \lim_{k \rightarrow \infty} d(g^0_k, g^1_k) \leq C(n) \left(
      \sqrt{\Vol(E, g_0)} + \sqrt{\Vol(E, g_1)} \right),
  \end{equation*}
  which would complete the proof.

  By the triangle inequality (\ref{item:3}) of Lemma \ref{lem:42}, we
  have
  \begin{equation}\label{eq:136}
    d(g^0_k, g^1_k) \leq d_\U(g^0_k, \chi(M \setminus E) g^0_k +
    \chi(E) g^1_k) + d_\U(\chi(M \setminus E) g^0_k +
    \chi(E) g^1_k, g^1_k).
  \end{equation}
  By Lemma \ref{lem:51}
  and Theorem \ref{thm:19}, we can conclude
  \begin{equation*}
    \lim_{k \rightarrow \infty} d_\U(g^0_k, \chi(M \setminus E) g^0_k
    + \chi(E) g^1_k) \leq C(n) \left( \sqrt{\Vol(E, g_0)} +
      \sqrt{\Vol(E, g_1)} \right).
  \end{equation*}
  Therefore, if we can show that the second term of \eqref{eq:136}
  converges to zero as $k \rightarrow \infty$, then we will have the
  desired result.  But $\{g^0_k\}$ $L^2$-converges to $g_0$ and
  $\{g^1_k\}$ $L^2$-converges to $g_1$.  Additionally, $\chi(M
  \setminus E) g_0 = \chi(M \setminus E) g_1$.  Therefore, Definition
  \ref{dfn:14} implies that
  \begin{equation*}
    \lim_{k\ \rightarrow \infty} d_\U(\chi(M \setminus E) g^0_k +
    \chi(E) g^1_k, g^1_k) = 0, 
  \end{equation*}
  which is what was to be shown.
\end{proof}

Next, we need another technical result that will help us in extending
the second uniqueness result from amenable subsets to all of $\M$.

\begin{prop}\label{prop:19}
  Say $g_0 \in \M$ and $h \in \s$, and let $E \subseteq M$ be any open
  set.  Define an $L^2$ tensor $g_1 \in \s^0$ by $g_1 := g_0 + h^0$,
  where $h^0 := \chi(E) h$.  Assume that we can find an amenable
  subset $\U$ such that $g_1 \in \U^0$.  Finally, define a path $g_t$
  of $L^2$ metrics by $g_t := g_0 + t h^0$, $t \in [0,1]$.

  Then without loss of generality (by enlarging $\U$ if necessary),
  $g_t \in \U^0$ for all $t$, so in particular $d_\U(g_0, g_1)$ is
  well-defined.  Furthermore,
  \begin{equation}\label{eq:47}
    d_\U(g_0, g_1) \leq L(g_t) := \integral{0}{1}{\norm{h^0}_{g_t}}{d t},
  \end{equation}
  i.e., the length of $g_t$, when measured in the naive way, bounds
  $d_\U(g_0, g_1)$ from above.

  Lastly, suppose that on $E$, the metrics $g_t$, $t \in [0, 1]$, all
  satisfy the bounds
  \begin{equation*}
    | (g_t)_{ij}(x) | \leq C \quad \textnormal{and} \quad
    \lambda^{G_t}_{\textnormal{min}}(x) \geq \delta
  \end{equation*}
  for some $C, \delta > 0$, all $1 \leq i,j \leq n$ and a.e.~$x \in
  E$.  (That this is satisfied for some $C$ and $\delta$ is guaranteed
  by $g_t \in \U^0$.)  Then there is a constant $K = K(C, \delta)$
  such that
  \begin{equation*}
    d_\U(g_0, g_1) \leq K \norm{h^0}_g.
  \end{equation*}
\end{prop}
\begin{proof}
  The existence of the enlarged amenable subset $\U$ is clear from the
  construction of $g_t$.  So we turn to the proof of \eqref{eq:47}.
  
  Let any $\epsilon > 0$ be given.  By Theorem \ref{thm:5}, we can
  choose $\delta > 0$ such that for any $\tilde{g}_0, \tilde{g}_1 \in
  \U$, $\norm{\tilde{g}_1 - \tilde{g}_0}_g < \delta$ implies
  $d(\tilde{g}_0, \tilde{g}_1) < \epsilon$.

  Next, for each $k \in \N$, we choose closed sets $F_k \subseteq E$
  and open sets $U_k \supseteq E$ with the property that $\Vol(U_k, g)
  - \Vol(F_k, g) < 1/k$.  Given this, let's even restrict ourselves to
  $k$ large enough that
  \begin{equation}\label{eq:138}
    \norm{\chi(U_k \setminus F_k) h}_g < \min \{\delta, \epsilon\}.
  \end{equation}

  We then choose $f_k \in C^\infty(M)$ satisfying
  \begin{enumerate}
  \item $f_k(x) = 1$ if $x \in F_k$,
  \item $f_k(x) = 0$ if $x \not\in U_k$ and
  \item $0 \leq f_k(x) \leq 1$ for all $x \in M$,
  \end{enumerate}

  The first consequence of our assumptions above is
  \begin{equation}\label{eq:77}
    \norm{g_1 - (g_0 + f_k h)}_g \leq \norm{\chi(U_k \setminus F_k) h)}_g < \delta.
  \end{equation}
  The second inequality is \eqref{eq:138}, and the first inequality
  holds for two reasons.  First, on both $F_k$ and $M \setminus U_k$,
  $g_0 + f_k h = g_0 + \chi(F_k) h = g_1$.  Second, on $U_k \setminus
  F_k$, $g_1 - (g_0 + f_k h) = (1 - f_k) h$, and by our third
  assumption on $f_k$, $0 \leq 1 - f_k \leq 1$.  Now, inequality
  \eqref{eq:77} allows us to conclude, by our assumption on $\delta$,
  that
  \begin{equation}\label{eq:78}
    d_\U(g_0 + f_k h, g_1) < \epsilon.
  \end{equation}
  Since by the triangle inequality
  \begin{equation*}
    d_\U(g_0, g_1) \leq d_\U(g_0, g_0 + f_k h) + d_\U(g_0 + f_k h, g_1)
    < d_\U(g_0, g_0 + f_k h) + \epsilon,
  \end{equation*}
  we must now get some estimates on $d_\U(g_0, g_0 + f_k h)$ to prove
  \eqref{eq:47}.
  
  To do this, define a path $g_t^k$ in $\M$, for $t \in [0,1]$, by
  $g_t^k := g_0 + t f_k h$.  Then we have, as is easy to see,
  \begin{equation}\label{eq:79}
    d(g_0, g_0 + f_k h) \leq L(g_t^k) = \integral{0}{1}{\norm{f_k
        h}_{g_t^k}}{d t}
  \end{equation}
  This is almost what we want, but we first have to replace $f_k h$
  with $h^0 = \chi(E) h$.  Also note that the $L^2$ norm in
  \eqref{eq:79} is that of $g_t^k$.  To put this in a form useful for
  proving \eqref{eq:47}, we therefore also have to to replace $g^k_t$
  with $g_t$.

  Using the facts that on $F_k$, $f_k h = h^0$ and
  $g^k_t = g_t$, as well as that $f_k = 0$ on $M \setminus
  U_k$, we can write
  \begin{equation}\label{eq:80}
      \norm{f_k h}_{g_t^k}^2       = \integral{F_k}{}{\tr_{g_t}
        \left( ( h^0 )^2 \right)}{\mu_{g_t}} + \integral{U_k
        \setminus F_k}{}{\tr_{g_t^k} \left( (f_k h)^2 \right)}{\mu_{g_t^k}}.
  \end{equation}

  For the first term above, we clearly have
  \begin{equation}\label{eq:81}
    \integral{F_k}{}{\tr_{g_t} \left( ( h^0 )^2 \right)}{\mu_{g_t}} \leq \norm{h^0}_{g_t}^2.
  \end{equation}

  As for the second term, it can be rewritten and estimated by
  \begin{equation*}
    \integral{U_k \setminus F_k}{}{\tr_{g_t^k} \left(
        (f_k h)^2 \right)}{\mu_{g_t^k}} = \norm{\chi(U_k
    \setminus F_k) f_k h}_{g_t^k} \leq \norm{\chi(U_k
    \setminus F_k) h}_{g_t^k},
  \end{equation*}
  where the inequality follows from our third assumption on $f_k$
  above.  Now, recall that $g_t$ is contained within an amenable
  subset $\U$.  It is possible to enlarge $\U$, without changing the
  property of being amenable, so that $\U$ contains $g_t^k$ for all $t
  \in [0,1]$ and all $k \in \N$.
  Therefore, by Lemma \ref{lem:18}, there exists a
  constant $K' = K'(\U, g_0, g_1)$---i.e., $K'$ does not depend on
  $k$---such that
  \begin{equation*}
    \norm{\chi(U_k \setminus F_k) h}_{g_t^k} \leq K' \norm{\chi(U_k
    \setminus F_k) h}_g.
  \end{equation*}
  But by \eqref{eq:138}, we have that $\norm{\chi(U_k \setminus F_k)
    h}_g < \epsilon$.  Combining this with \eqref{eq:80} and
  \eqref{eq:81}, we therefore get
  \begin{equation*}
    \norm{f_k h}_{g_t^k}^2 \leq \norm{h^0}_{g_t} + K' \epsilon.
  \end{equation*}

  The above inequality, substituted into \eqref{eq:79}, gives
  \begin{equation*}
    d(g_0, g_1^k) \leq \integral{0}{1}{\left( \norm{h^0}_{g_t} +
        K' \epsilon \right)}{d t} =
    L(g_t) + K' \epsilon.
  \end{equation*}
  The final step in the proof is then to estimate, using the above
  inequality and \eqref{eq:78}, that
  \begin{equation*}
    d_\U(g_0, g_1) \leq d(g_0, g_1^k) + d_\U(g_1^k, g_1) < L(g_t)
    + (1 + K') \epsilon.
  \end{equation*}
  Since $\epsilon$ was arbitrary and $K'$ is independent of $k$, we
  are finished with the proof of \eqref{eq:47}.

  Finally, the third statement follows from the following estimate,
  which is proved in exactly the same way as Lemma \ref{lem:18}:
  \begin{equation*}
    \norm{h^0}_{g_t} = \left( \integral{E}{}{\tr_{g_t}\left( h^2
        \right)}{\mu_{g_t}} \right)^{1/2} \leq K(C, \delta) \norm{h^0}_g.
  \end{equation*}
\end{proof}

With Theorem \ref{thm:15} and Proposition \ref{prop:19} as part of our
toolbox, we are now ready to take on the proof of the second
uniqueness result in its full generality.

So let two $d$-Cauchy sequences $\{g^0_k\}$ and $\{g^1_k\}$, as well
as some $g_\infty \in \M_f$, be given.  Suppose further that
$\{g^0_k\}$ and $\{g^1_k\}$ both $\omega$-converge to $g_\infty$ for $k
\rightarrow \infty$.  We will prove that
\begin{equation}\label{eq:70}
  \lim_{k \rightarrow \infty} d(g^0_k, g^1_k) = 0.
\end{equation}
The heuristic idea of our proof is very simple, which is belied by the
rather technical nature of the rigorous proof.  The point, though, is
essentially that for all $l \in \N$, we break $M$ up into two sets,
$E_l$ and $M \setminus E_l$.  The set $E_l$ has positive volume with
respect to $g_\infty$, but $\{g^0_k\}$ and $\{g^1_k\}$ $L^2$-converge
to $g_\infty$ on $E_l$, so the contribution of $E_l$ to $d(g^0_k,
g^1_k)$ vanishes in the limit $k \rightarrow \infty$.  The set $M
\setminus E_l$ contains the deflated sets of $\{g^0_k\}$ and
$\{g^1_k\}$, so the sequences need not converge on $M \setminus E_l$.
However, we choose things such that $\Vol(M \setminus E_l, g_\infty)$
vanishes in the limit $l \rightarrow \infty$, so that Proposition
\ref{prop:18} implies that the contribution of $M \setminus E_l$ to
$d(g^0_k, g^1_k)$ vanishes after taking the limits $k \rightarrow
\infty$ and $l \rightarrow \infty$ in succession.

The rigorous proof is achieved in three basic steps, which we will
describe after some brief preparation.

For each $l \in \N$, let
\begin{equation}\label{eq:85}
  E_l :=
  \left\{
    x \in M \midmid \det g^i_k(x) > \frac{1}{l},\ \abs{(g^i_k)_{rs}(x)}
    < l\ \forall i=0,1;\ k \in \N;\ 1 \leq r,s \leq n
  \right\},
\end{equation}
where these local notions are of course defined with respect to our
fixed amenable atlas (cf.~Convention \ref{cvt:5}), and the
inequalities in the definition should hold in each chart containing
the point $x$ in question.  Thus, $E_l$ is a set over which the
sequences $g_i^k$ neither deflate nor become unbounded.  We first note
that for each $k \in \N$, there exists an amenable subset $\U_k$ such
that the metrics
\begin{equation*}
  g^0_k,\ g^1_k\ \textnormal{and}\  g^0_k + \chi(E_l)(g^1_k - g^0_k)
\end{equation*}
are contained in $\U_k^0$.  This is possible due to smoothness of
$g^0_k$ and $g^1_k$, as well as pointwise convexity of the bounds of
Definition \ref{dfn:2}.

The steps in our proof are
the following.  We will show first that
\begin{equation}\label{eq:82}
 \lim_{k \rightarrow \infty} d_{\U_k}(g^0_k, g^0_k + \chi(E_l) (g^1_k -
  g^0_k)) = 0
\end{equation}
for all fixed $l \in \N$.  Second,
\begin{equation}\label{eq:83}
\lim_{k \rightarrow \infty} d_{\U_k}(g^0_k + \chi(E_l) (g^1_k - g^0_k),
g^1_k) \leq 2 C(n)
  \sqrt{\Vol(M \setminus E_l, g_\infty)}  
\end{equation}
for all fixed $k \in \N$ (where $C(n)$ is the constant from Theorem
\ref{thm:15}).  And third,
\begin{equation}\label{eq:84}
\lim_{l \rightarrow \infty} \Vol(E_l, g_\infty) = \Vol(M, g_\infty).
\end{equation}
Since the triangle inequality of Lemma \ref{lem:42}(2) implies that
\begin{equation*}
  d(g^0_k, g^1_k) \leq d_{\U_k}(g^0_k, g^0_k + \chi(E_l) (g^1_k -
  g^0_k)) + d_{\U_k}(g^0_k + \chi(E_l) (g^1_k -
  g^0_k), g^1_k)
\end{equation*}
for all $l \in \N$, taking the limits $k \rightarrow \infty$ followed
by $l \rightarrow \infty$ of both sides then gives \eqref{eq:70}.

We now prove each of \eqref{eq:82}, \eqref{eq:83} and \eqref{eq:84} in
its own lemma.

\begin{lem}\label{lem:38}
  \begin{equation*}
 \lim_{k \rightarrow \infty} d_{\U_k}(g^0_k, g^0_k + \chi(E_l) (g^1_k -
  g^0_k)) = 0
\end{equation*}
\end{lem}
\begin{proof}
  We know that
  \begin{equation*}
    g^0_k, g^0_k + \chi(E_l) (g^1_k - g^0_k) \in \U_k^0,
  \end{equation*}
  where $\U_k$ is an amenable subset.  Therefore, for each fixed $k
  \in \N$, Proposition \ref{prop:19} applies to give
  \begin{equation}\label{eq:86}
    d_{\U_k}(g^0_k, g^0_k + \chi(E_l) (g^1_k - g^0_k)) \leq K_l \norm{\chi(E_l)
    (g^1_k - g^0_k)}_g,
  \end{equation}
  where $K_l$ is some constant depending only on $l$.  (That the
  constant only depends on $l$ is the result of the fact that $g^0_k$
  and $g^1_k$ satisfy the bounds given in \eqref{eq:85} on $E_l$,
  which only depend on $l$.)  

  Now, recalling the definition \eqref{eq:85} of $E_l$, we note that
  for all $1 \leq i,j \leq n$ and all $k \in \N$, we have $|
  (g^1_k)_{ij}(x) - (g^0_k)_{ij}(x) |^2 \leq 4 l^2$ for $x \in E_l$,
  and hence the family of (local) functions
  \begin{equation*}
    \{ \chi(E_l) ((g^1_k)_{ij} - (g^0_k)_{ij}) \mid 1 \leq i,j
    \leq n,\ k \in \N \}
  \end{equation*}
  is equicontinuous at $\emptyset$.  Furthermore, since property
  (\ref{item:6}) of Definition \ref{dfn:13} implies that $\chi(E_l)
  g_a^k \rightarrow \chi(E_l) g_\infty$ a.e.~for $a=0,1$, we have that
  $\chi(E_l) (g^1_k - g^0_k) \rightarrow 0$ a.e.  Therefore, as in the
  proof of Proposition \ref{prop:1}, we have that
  \begin{equation*}
    \norm{\chi(E_l) ( g^1_k - g^0_k )}_g
    \rightarrow 0
  \end{equation*}
  for $k \rightarrow \infty$.  Together with \eqref{eq:86}, this
  implies the result immediately.
\end{proof}

\begin{lem}\label{lem:40}
  \begin{equation*}
    \lim_{k \rightarrow \infty} d_{\U_k}(g^0_k + \chi(E_l) (g^1_k - g^0_k),
    g^1_k) \leq 2 C(n) \sqrt{\Vol(M \setminus E_l, g_\infty)}
  \end{equation*}
\end{lem}
\begin{proof}
  First note that $g^1_k = g^0_k + \chi(E_l) (g^1_k - g^0_k)$ on
  $E_l$.  Therefore, by Theorem \ref{thm:15},
  \begin{equation*}
    d_{\U_k}(g^0_k + \chi(E_l) (g^1_k - g^0_k),
    g^1_k) \leq C(n) \left( \sqrt{\Vol(M \setminus E_l, g^0_k)} +
      \sqrt{\Vol(M \setminus E_l, g^1_k)} \right).
  \end{equation*}
  But now the result follows immediately from Theorem \ref{thm:19},
  since $\Vol(M \setminus E_l, g_i^k) \rightarrow \Vol(M \setminus
  E_l, g_\infty)$ for $i = 0,1$.
\end{proof}

\begin{lem}\label{lem:39}
  \begin{equation*}
    \lim_{l \rightarrow \infty} \Vol(E_l, g_\infty) = \Vol(M, g_\infty).
  \end{equation*}
\end{lem}
\begin{proof}
  Recall that $X_{g_\infty} \subseteq M$ denotes the deflated set of
  $g_\infty$, i.e., the set where $g_\infty$ is not positive definite.
  This set has volume zero w.r.t.~$g_\infty$, since $\mu_{g_\infty} =
  0$ a.e.~on $X_{g_\infty}$.  Therefore $\Vol(M, g_\infty) = \Vol(M
  \setminus X_{g_\infty}, g_\infty)$.

  We note that $\chi(E_l)$ converges a.e.~to $\chi(M \setminus
  X_{g_\infty})$ and that $\chi(E_l)(x) \leq 1$ for all $x \in M$.
  Since $g_\infty$ has finite volume, the constant function 1 is
  integrable w.r.t.~$\mu_{g_\infty}$, and therefore the Lebesgue
  dominated convergence theorem implies that
  \begin{equation*}
    \lim_{l \rightarrow \infty} \Vol(E_l, g_\infty) = \lim_{l
      \rightarrow \infty} \integral{M}{}{\chi(E_l)}{\mu_{g_\infty}}
    = \integral{M}{}{\chi(M \setminus X_{g_\infty})}{\mu_{g_\infty}}
    = \Vol(M \setminus X_{g_\infty}, g_\infty).
  \end{equation*}
\end{proof}

As already noted, Lemmas \ref{lem:38}, \ref{lem:40} and \ref{lem:39}
combine to give the desired result.  We summarize what we have just
proved in a theorem.

\begin{thm}\label{thm:17}
  Let $[g_\infty] \in \Mfhat$.  Suppose we have two sequences
  $\{g^0_k\}$ and $\{g^1_k\}$ with $g^0_k, g^1_k \overarrow{\omega}
  [g_\infty]$ for $k \rightarrow \infty$.  Then
  \begin{equation*}
    \lim_{k \rightarrow \infty} d(g^0_k, g^1_k) = 0,
  \end{equation*}
  that is, $\{g^0_k\}$ and $\{g^1_k\}$ are equivalent in the precompletion
  $\overline{\M}^{\mathrm{pre}}$ of $\M$.
\end{thm}

As we have already discussed, combining this theorem with the
existence result (Theorem \ref{thm:40}) and the first uniqueness
result (Theorem \ref{thm:20}) gives us an identification of
$\overline{\M}$ with a subset of $\Mfhat$:

\begin{dfn}\label{dfn:3}
  Denote by $\Omega : \overline{\M} \rightarrow \Mfhat$ the map
  sending an equivalence class of Cauchy sequences to the unique
  element of $\Mfhat$ that all of its representatives
  $\omega$-subconverge to.
\end{dfn}

In the next section, we will prove that $\overline{\M}$ is actually
identified with all of $\Mfhat$, i.e., $\Omega$ is surjective.  We
note here that for the purpose of studying $\overline{\M}$ we can now
use $\Omega$ to drop the distinction between an $\omega$-convergent
sequence and the element of $\Mfhat$ that it converges to.  We will
employ this trick in what follows to simplify formulas and proofs.

\section{The completion of $\M$}
\label{chap:sing-metrics}

In this section, our previous efforts come to fruition and we are able
to complete our description of $\overline{\M}$ by proving, in Section
\ref{sec:finite-volume-metr}, that the map $\Omega : \overline{\M}
\rightarrow \Mfhat$ defined in the previous chapter is a bijection.

Section \ref{sec:meas-induc-degen} provides some necessary preparation
for the surjectivity proof by going into more depth on the behavior of
volume forms under $\omega$-convergence.  After this, Section
\ref{sec:meas-techn-result} presents a partial result on the image of
$\Omega$.  Namely, we show that all equivalence classes of measurable,
bounded semimetrics (cf.~Definition \ref{dfn:23}) are contained in
$\Omega(\overline{\M})$.  This marks the final preparation we need to
prove the main result.

\subsection{Measures induced by measurable
  semimetrics}\label{sec:meas-induc-degen}

For use in Section \ref{sec:finite-volume-metr}, we need to record a
couple of properties of the measure $\mu_{\tilde{g}}$ induced by an
element $\tilde{g} \in \M_f$.

The first property is continuity of the norms of continuous functions
under $\omega$-convergence.  It follows immediately from Theorem
\ref{thm:19} and the Portmanteau theorem
\cite[Thm.~8.1]{e70:_topol_and_measur}:

\begin{lem}\label{lem:54}
  Let $\tilde{g} \in \M_f$, and let $\rho \in C^0(M)$ be any
  continuous function.  If the sequence $\{ g_k \}$ $\omega$-converges
  to $\tilde{g}$, then $\mu_{g_k}$ converges weakly to
  $\mu_{\tilde{g}}$, so in particular
  \begin{equation*}
    \lim_{k \rightarrow \infty} \norm{\rho}_{g_k} = \norm{\rho}_{\tilde{g}}.
  \end{equation*}
\end{lem}

The next fact we need is that if $\tilde{g} \in \M_f$, i.e.,
$\tilde{g}$ is a measurable, finite-volume semimetric, then the set of
$C^\infty$ functions is dense in $L^p(M, \tilde{g})$ for $1 \leq p <
\infty$, just as in the case of a smooth volume form.

To prove this claim, we first state a fact about measures on
$\R^n$.  One can prove it almost identically to
\cite[Cor.~4.2.2]{bogachev07:_measur_theor}, where the statement is
made for Borel measures.  To prove it for Lebesgue measures,
one must simply approximate Lebesgue-measurable sets by
Borel-measurable sets using the discussion of Section
\ref{sec:lebesg-meas-manif}.

\begin{thm}\label{thm:13}
  Let a nonnegative measure $\nu$ on the algebra of Lebesgue sets in
  $\R^n$ be bounded on bounded sets. Then the class $C_0^\infty(\R^n)$
  of smooth functions with bounded support is dense in $L^p(\R^n,
  \nu)$, $1 \leq p < \infty$.
\end{thm}

Now, since any $\tilde{g} \in \M_f$ has finite volume, its induced
measure $\mu_{\tilde{g}}$ clearly satisfies the hypotheses of the
theorem in any coordinate chart.  Therefore, we have:

\begin{cor}\label{cor:12}
  If $\tilde{g} \in \M_f$, then $C^\infty(M)$ is dense in
  $L^p(M,\tilde{g})$.
\end{cor}

\subsection{Bounded semimetrics}\label{sec:meas-techn-result}

In this section, we go one step further in our understanding of the
injection $\Omega : \overline{\M} \rightarrow \Mfhat$ that was
introduced in Definition \ref{dfn:3}.  Specifically, we want to see
that the image $\Omega(\overline{\M})$ contains all equivalence
classes of bounded, measurable semimetrics (cf.~Definition
\ref{dfn:23}).

Our strategy for proving this is to first prove the fact for smooth
semimetrics by showing that for any smooth semimetric $g_0$, there is
a finite path $g_t$, $t \in (0,1]$, in $\M$ with $\lim_{t \to 0} g_t =
g_0$ (where we take the limit in the $C^\infty$ topology of $\s$).
Then, if we simply let $t_k$ be any monotonically decreasing sequence
converging to zero, it is trivial to show $g_{t_k} \overarrow{\omega}
g_0$ for $k \rightarrow \infty$.
We then use this to handle the general, nonsmooth case.

\subsubsection{Paths to the boundary}\label{sec:paths-boundary}

Before we get into the proofs, we put ourselves in the proper setting,
for which we first need to introduce the notion of a quasi-amenable subset.
These are defined by weakening the requirements for an amenable subset
(cf.~Definition \ref{dfn:2}), giving up the condition of being
``uniformly inflated'':

\begin{dfn}\label{dfn:26}
  We call a subset $\U \subset \M$ \emph{quasi-amenable} if $\U$ is
  convex and we can find a constant $C$ such that for all $\tilde{g}
  \in \U$, $x \in M$ and $1 \leq i,j \leq n$,
  \begin{equation}\label{eq:134}
    |\tilde{g}_{ij}(x)| \leq C.
  \end{equation}
\end{dfn}

We also define $\partial \M$ to be the boundary of $\M$ as a
topological subset of $\s$.  Thus, it consists of all smooth
semimetrics that somewhere fail to be positive definite.

Let $\U$ be any quasi-amenable subset, and denote by
$\cl(\U)$ \label{p:cl-U} the closure of $\U$ in the $C^\infty$
topology of $\s$.  Thus, $\cl(\U)$ may contain some smooth
semimetrics.  One can easily see that any $g_0 \in \partial \M$ is
contained in $\cl(\U)$ for an appropriate quasi-amenable subset $\U$.

Now, suppose some $g_0 \in \cl(\U) \cap \btop$ is given, and let $g_1
\in \U$ have the property that $h := g_1 - g_0 \in \M$.
Exploiting the linear structure of $\M$, we
define the simplest path imaginable from $g_0$ to $g_1$:
\begin{equation}\label{eq:17}
  g_t := g_0 + t h.
\end{equation}
Then by the convexity of $\U$, $g_t$ is a path $(0,1] \rightarrow \U$
with limit (in the topology of $\s$) as $t \rightarrow 0$ equal to
$g_0$.

Recall that the length of $g_t$ is given by
\begin{equation}\label{eq:29}
  L(g_t)   =
    \integral{0}{1}{\left(
        \integral{M}{}{\tr_{g_t}((g'_t)^2)}{\mu_{g_t}}
      \right)^{1/2}}{d t}         = \integral{0}{1}{\left( \integral{M}{}{\tr_{g_t}(h^2) \sqrt{\det
            (g^{-1} g_t)}}{\mu_g} \right)^{1/2}}{d t}
\end{equation}
To prove that $g_t$ is a finite path, we must therefore estimate the
inner integrand.
This will follow from pointwise estimates combined with a
compactness/continuity argument.

\subsubsection{Pointwise estimates}\label{sec:pointwise-estimates}

Let $A = (a_{ij})$ and $B = (b_{ij})$ be real, symmetric $n \times n$
matrices, with $A_t := A+tB$ for $t \in (0,1]$.  We will assume that
$B > 0$ and that $A \geq 0$.  (In this scheme, $A$ and $B$ play the
role of $g_0(x)$ and $h(x)$, respectively, at some point $x \in M$.)
Furthermore, we fix an arbitrary matrix $C$ that is invertible and
symmetric (this plays the role of $g(x)$).

To get a pointwise estimate on $\tr_{g_t} (h^2) \sqrt{\det g^{-1}
  g_t}$, we need to estimate $\tr_{A_t}(B^2) \sqrt{\det (C^{-1}
  A_t)}$.  We prove the desired estimate in two lemmas.

For any symmetric matrix $D$, let $\lambda^D_{\mathrm{min}} =
\lambda^D_1 \leq \cdots \leq \lambda^D_n = \lambda^D_{\mathrm{max}}$
be its eigenvalues numbered in increasing order.

\begin{lem}\label{lem:1}
  \begin{align*}
    \lambda^{A_t}_{\mathrm{min}} &\geq \lambda^A_{\mathrm{min}} + t
    \lambda^B_{\mathrm{min}} \\
    \lambda^{A_t}_{\mathrm{max}} &\leq \lambda^A_{\mathrm{max}} + t
    \lambda^B_{\mathrm{max}} \leq \lambda^A_{\mathrm{max}} +
    \lambda^B_{\mathrm{max}}
  \end{align*}
\end{lem}
\begin{proof}
  Immediate from the concavity/convexity of the minimal/maximal
  eigenvalue (cf.~the proof of Lemma \ref{lem:45}).
\end{proof}

\begin{lem}\label{lem:4}
  \begin{equation*}
    \tr_{A_t} \left( B^2 \right) \sqrt{\det C^{-1} A_t} \leq
    \frac{n \left(
        \lambda^B_{\mathrm{max}} \right)^2 \left( \lambda^{A_t}_{\mathrm{max}}
      \right)^{\frac{n-1}{2}}}{\sqrt{\det C} \left(
        \lambda^B_{\mathrm{min}}
      \right)^{3/2}}  \frac{1}{t^{3/2}}
  \end{equation*}
\end{lem}
\begin{proof}
  We focus on the trace term first.
  
  Since $B$ is a symmetric matrix, there exists a basis for which $B$
  is diagonal, so that $B = \diag(\lambda_1^B,\ldots,\lambda_n^B)$.
  In this basis, if we denote $A_t^{-1}$ by $(a_t^{ij})$, then we have
  \begin{equation}\label{eq:9}
    \tr \left( \left( A_t^{-1} B \right)^2 \right)     = \sum_{ij} a_t^{ij} \lambda_j^B
      a_t^{ji} \lambda_i^B       \leq \left( \lambda^B_{\mathrm{max}} \right)^2 \sum_{ij}
      \left(a^{ij}_t\right)^2             = \left( \lambda^B_{\mathrm{max}} \right)^2 \tr \left(
        A_t^{-2}\right),
  \end{equation}
  where we have used the symmetry of $A_t^{-1}$.
  
  We note that the trace of the square of a matrix is given by the sum
  of the squares of its eigenvalues.  Therefore,
  \begin{equation}\label{eq:10}
    \tr \left( A_t^{-2}\right) = \sum_i \left( \lambda^{A_t}_i \right)^{-2} \leq
    n \left( \lambda^{A_t}_{\mathrm{min}} \right)^{-2}.
  \end{equation}
  This takes care of the trace term.

  For the determinant term, we clearly have $\det A_t \leq
  \lambda^{A_t}_{\mathrm{min}} ( \lambda^{A_t}_{\mathrm{max}}
  )^{n-1}$.
  Combining this, equations \eqref{eq:9} and \eqref{eq:10}, the
  estimate of Lemma \ref{lem:1}, and the fact that $\lambda^A_{\min}
  \geq 0$ (as $A \geq 0$) now implies the result.
\end{proof}

\subsubsection{Finiteness of $L(g_t)$}\label{sec:finiteness-lg_t}

We want to use the pointwise estimate of Lemma \ref{lem:4} to prove
the main result of the section.

It is clear that to pass from the pointwise result of Lemma
\ref{lem:4} to a global result, we will have to estimate the maximum
and minimum eigenvalues of $h$, as well as the maximum eigenvalue of
$g_t$.  We begin by noting that since we work over an amenable
coordinate atlas (cf.~Definition \ref{dfn:1}), all coefficients of
$h$, $g$ and $g_0$ are bounded in absolute value.  Therefore, so are
their determinants.  In particular, since $g > 0$ and $h > 0$, we can
assume that $\det g \geq C_0$ and $C_1 \geq \det h \geq C_2$ over each
chart of the amenable atlas for some constants $C_0,C_1,C_2 > 0$.

\begin{lem}\label{lem:21}
  The quantities $\lambda^h_{\mathrm{max}}$ and
  $\lambda^{g_t}_{\mathrm{max}}$, as local functions on each
  coordinate chart, are uniformly bounded, say
  $\lambda^h_{\mathrm{max}}(x) \leq C_3$ and $\lambda^{g_t}_{
    \mathrm{max}}(x) \leq C_4$ for all $x$ and $t$.
\end{lem}
\begin{proof}
  Note that $g_t$ lies in the quasi-amenable subset $\U$ for all $t$,
  so we have upper bounds (in absolute value) on the coefficients of
  $g_t$ and $h$ that are uniform in $x$ and $t$.  Thus, the bounds on
  their maximal eigenvalues follow straighforwardly from the min-max
  theorem
  \cite[Thm.~XIII.1]{reed78:_method_of_moder_mathem_physic_iv}.
\end{proof}

\begin{lem}\label{lem:7}
  The quantity $\lambda^h_{\min}$, as a function over each coordinate
  chart, is uniformly bounded away from 0, say $\lambda^h_{\min} \geq
  C_5 > 0$.
\end{lem}
\begin{proof}
  We clearly have $\det h(x) \leq \lambda^h_{\mathrm{min}}(x)
  \lambda^h_{\mathrm{max}}(x)^{n-1}$.  Therefore, by Lemma
  \ref{lem:21} and the discussion before it,
  $\lambda^h_{\mathrm{min}}(x) \geq \lambda^h_{\mathrm{max}}(x)^{1-n}
  \det h(x) \geq C_3^{1-n} C_2 =: C_5$.
 \end{proof}

\begin{thm}\label{thm:2}
  Define a path $g_t$ as in \eqref{eq:17}.  Then
  \begin{equation*}
    L(g_t) < \infty.
  \end{equation*}
\end{thm}
\begin{proof}
  At each point $x \in M$, we use Lemma \ref{lem:4} to see
  \begin{equation}\label{eq:15}
      \tr_{g_t(x)}(h(x)^2) \sqrt{\det(g(x)^{-1} g_t(x))}       \leq \frac{1}{\sqrt{C_0}} \frac{C_3^2}{C_5^{3/2}}
      C_4^{\frac{n-1}{2}} \frac{1}{t^{3/2}} =: \frac{C_6}{t^{3/2}}.
  \end{equation}
  
  The result then follows from (\ref{eq:29}) and the integrability of
  $t^{-3/4}$.
\end{proof}

\subsubsection{Bounded, nonsmooth
  semimetrics}\label{sec:nonsm-linfty-semim}

We now move on to showing that the equivalence class of any
bounded semimetric, not just smooth ones, is contained in
$\Omega(\overline{\M})$.

So far, we know from Proposition \ref{prop:26} that the equivalence
class of any measurable metric that can be obtained as the $L^2$ limit
of a sequence of metrics from an amenable subset belongs to
$\Omega(\overline{\M})$.  We also know from the preceding arguments
that any smooth semimetric $\tilde{g}$ is in the image of $\Omega$.
Given the remarks at the end of Section \ref{cha:almost-everywh-conv},
we can therefore unambiguously write things like $d(g_0, g_1)$---where
$g_0$ and $g_1$ are known to belong to the image of $\Omega$---in
place of expressions involving sequences $\omega$-converging to $g_0$
and $g_1$.

To begin proving the result on bounded, nonsmooth semimetrics, we want
to prove a result about quasi-amenable subsets that is a
generalization of Theorem \ref{thm:5}---weakened so that it still
applies for these more general subsets.
First, though, we need to prove a couple of lemmas.

\begin{lem}\label{lem:8}
  Let $\U \subset \M$ be quasi-amenable.  Recall that we denote the
  closure of $\U$ in the $C^\infty$ topology of $\s$ by $\cl(\U)$, and
  we denote the boundary of $\M$ in the $C^\infty$ topology of $\s$ by
  $\btop$. Then for each $\epsilon > 0$, there exists $\delta > 0$
  such that $d(g_0, g_0 + \delta g) < \epsilon$ for all $g_0 \in
  \cl(\U) \cap \btop$.
\end{lem}
\begin{proof}
  For any $g_0 \in \cl(\U) \cap \btop$, we consider the path $g_t :=
  g_0 + t h$, where $h := \delta g$ and $t \in (0,1]$.  The proof
  consists of reexamining the estimates of Theorem \ref{thm:2} and
  showing that they only depend on upper bounds on the entries of
  $g_0$ (and $g$, but we get these automatically when we work over an
  amenable atlas), and that the bound on the length of $g_t$ goes to
  zero as $\delta \rightarrow 0$.

  So, recall the main estimate \eqref{eq:15} of Theorem \ref{thm:2}:
  \begin{equation*}
    \tr_{g_t(x)}(h(x)^2) \sqrt{\det(g(x)^{-1} g_t(x))} \leq
    \frac{n \left( \lambda^{h}_{\mathrm{max}}(x)
      \right)^2  \left( \lambda^{g_t}_{\mathrm{max}}(x)
      \right)^{\frac{n-1}{2}}}{\sqrt{\det g(x)} \left(
        \lambda^{h}_{\mathrm{min}}(x) \right)^{3/2}}
    \frac{1}{t^{3/2}}.
  \end{equation*}
  
  Since $\det g(x)$ is constant w.r.t.~$\delta$, we ignore this term.
  By Lemma \ref{lem:1},
  \begin{equation*}
    \lambda^{g_t}_{\mathrm{max}}(x) \leq
    \lambda^{g_0}_{\mathrm{max}}(x) + \lambda^{h}_{\mathrm{max}}(x) =
    \lambda^{g_0}_{\mathrm{max}}(x) + 
    \delta \lambda^{g}_{\mathrm{max}}(x).
  \end{equation*}
  Therefore, using the same arguments as in Lemma \ref{lem:21},
  $\lambda^{g_t}_{\mathrm{max}}(x)$ is bounded from above, uniformly
  in $x$ and $t$, by a constant that decreases as $\delta$ decreases.
  Furthermore, this constant does not depend on our choice of $g_0 \in
  \cl(\U) \cap \btop$, since the proof of Lemma \ref{lem:21} depended
  only on uniform upper bounds on the entries of $g_0$, and we are
  guaranteed the same upper bounds on all elements of $\cl(\U) \cap
  \btop$ since $\U$ is quasi-amenable.

  We now focus our attention on the term
  \begin{equation*}
    \frac{\left( \lambda^{h}_{\mathrm{max}}(x)
      \right)^2}{\left( \lambda^{h}_{\mathrm{min}}(x) \right)^{3/2}} =
    \frac{\left( \delta \lambda^{g}_{\mathrm{max}}(x)
      \right)^2}{\left( \delta \lambda^{g}_{\mathrm{min}}(x)
      \right)^{3/2}} = \frac{\left( \lambda^{g}_{\mathrm{max}}(x)
      \right)^2}{\left( \lambda^{g}_{\mathrm{min}}(x)
      \right)^{3/2}} \sqrt{\delta}.
  \end{equation*}
  This expression clearly goes to zero as $\delta \rightarrow 0$.
  Therefore, we have shown that the constant $C_6$ in the estimate
  \eqref{eq:15} depends only on the choice of $\U$ and $\delta$, and
  that $C_6 \rightarrow 0$ as $\delta \rightarrow 0$.  The result now
  follows.
\end{proof}

The next lemma implies, in particular, that $\btop$ is \emph{not
  closed} in the $L^2$ topology of $\s$, nor is it in the topology of
$d$ on $\Omega(\overline{\M})$.  It also implies that around any point
in $\M$, there exists no $L^2$- or $d$-open neighborhood.

\begin{lem}\label{lem:9}
  Let $\U \in \M$ be any quasi-amenable subset.  Then for all
  $\epsilon > 0$, there exists a function $\rho_\epsilon \in
  C^\infty(M)$ with the properties that for all $g_1 \in \U$,
  \begin{enumerate}
  \item $\rho_\epsilon g_1 \in \btop$,
  \item $\rho_\epsilon(x) \leq 1$ for all $x \in M$,
  \item $\norm{g_1 - \rho_\epsilon g_1}_g < \delta$ and
  \item $d(g_1, \rho_\epsilon g_1) < \epsilon$.
  \end{enumerate}
\end{lem}
\begin{proof}
  Let $x_0 \in M$ be any point, and for each $n \in \N$, choose a
  function $\rho_n \in C^\infty(M)$ satisfying
  \begin{enumerate}
  \item $\rho_n(x_0) = 0$,
  \item $0 \leq \rho_n(x) \leq 1$ for all $x \in M$ and
  \item $\rho_n \equiv 1$ outside an open set $Z_n$ with $\Vol(Z_n, g)
    \leq 1/n$.
  \end{enumerate}
  Then clearly $\norm{g_1 - \rho_n g_1}_g \rightarrow 0$ as $n
  \rightarrow \infty$, and this convergence is uniform in $g_1$
  because of the upper bounds guaranteed by the fact that $g_1 \in
  \U$.  Using arguments similar to those in the last lemma, we can
  also see that the length of the path $g_t^n := \rho_n g_1 + t (g_1 -
  \rho_n g_1)$ converges to zero as $n \rightarrow \infty$.
  Therefore, choosing $n$ large enough gives the desired function.
\end{proof}

The next theorem is the desired analog of Theorem \ref{thm:5}.  Note
that only one half of Theorem \ref{thm:5} holds in this case, and even
this is proved only in a weaker form.

\begin{thm}\label{thm:8}
  Let $\U \subset \M$ be quasi-amenable.  Then for all $\epsilon > 0$,
  there exists $\delta > 0$ such that if $g_0, g_1 \in \cl(\U)$ with
  $\norm{g_0 - g_1}_g < \delta$, then $d(g_0, g_1) < \epsilon$.
\end{thm}
\begin{proof}
  First, we enlarge $\U$ if necessary to include \emph{all} metrics
  satisfying the bound given in Definition \ref{dfn:26}.  This
  enlarged $\U$ is then clearly convex by the triangle inequality for
  the absolute value, and hence it is still a quasi-amenable subset.

  Now, let $\epsilon > 0$ be given.  We prove the statement first for
  $g_0, g_1 \in \cl(\U) \cap \btop$, then use this to
  prove the general case.

  By Lemma \ref{lem:8}, we can choose $\delta_1 > 0$ such that
  $d(\hat{g}, \hat{g} + \delta_1 g) < \epsilon/3$ for all $\hat{g} \in
  \cl(\U) \cap \btop$.  We define an amenable subset of $\M$ by
  \begin{equation*}
    \U' :=
    \left\{
      \hat{g} + \delta_1 g \mid \hat{g} \in \U
    \right\}.
  \end{equation*}
  Lemma \ref{lem:1} implies that this set is indeed amenable.
  Now, by Theorem \ref{thm:5}, there exists $\delta > 0$ such that if
  $\tilde{g}_0, \tilde{g}_1 \in \U'$ with $\norm{\tilde{g}_0 -
  \tilde{g}_1}_g < \delta$, then $d(\tilde{g}_0, \tilde{g}_1) <
  \epsilon/3$.  Let $g_0, g_1 \in \cl(\U) \cap \btop$ be such that $\norm{g_0 - g_1}_g < \delta$.  If we define $\tilde{g}_i := g_i +
  \delta_1 g$ for $i = 1,2$, then it is clear that $\norm{\tilde{g}_0 -
  \tilde{g}_1}_g = \norm{g_0 - g_1}_g < \delta$.  Given this and the
  definition of $\delta_1$, we have
  \begin{equation*}
    d(g_0, g_1) \leq d(g_0, \tilde{g}_0) +
    d(\tilde{g}_0, \tilde{g}_1) + d(\tilde{g}_1, g_1) < \epsilon.
  \end{equation*}

  Now we prove the general case.  Let $\epsilon > 0$ be given.  By the
  special case we just proved, we can choose $\delta > 0$ such that if
  $\tilde{g}_0, \tilde{g}_1 \in \cl(\U) \cap \btop$ with $\norm{\tilde{g}_0 - \tilde{g}_1}_g < \delta$, then $d(\tilde{g}_0,
  \tilde{g}_1) < \epsilon/3$.  Let $g_0, g_1 \in \U$ be any elements
  with $\norm{g_0 - g_1}_g < \delta$.  By Lemma \ref{lem:9} and our
  enlargement of $\U$, we can choose a function $\rho \in C^\infty(M)$
  such that for $i = 0,1$,
  \begin{enumerate}
  \item $\rho g_i \in \cl(\U) \cap \btop$,
  \item $\rho(x) \leq 1$ for all $x \in M$, and
  \item $d(g_i, \rho g_i) < \epsilon/3$.
  \end{enumerate}
  (If $g_i \in \cl(\U) \cap \btop$ for both $i=1$ and $2$, we might as
  well just choose $\rho \equiv 1$.)  In particular, the second
  property of $\rho$ implies that
  \begin{equation*}
    \norm{\rho g_1 - \rho g_0}_g \leq \norm{g_1 - g_0}_g < \delta.
  \end{equation*}
  
  Then we immediately get
  \begin{equation*}
    d(g_0, g_1) \leq d(g_0, \rho g_0) + d(\rho g_0, \rho g_1) +
    d(\rho g_1, g_1) < \epsilon.
  \end{equation*}
  This proves the general case and thus the theorem.
\end{proof}

Using the relationship between $d$ and $\normdot_g$ determined in
Theorem \ref{thm:8}, we can prove the following.

\begin{prop}\label{prop:27}
  Let $[\tilde{g}] \in \Mfhat$ be an equivalence class of bounded,
  measurable semimetrics.  Then for any representative $\tilde{g} \in
  [\tilde{g}]$, there exists a sequence $\{ g_k \}$ in $\M$ that both
  $L^2$- and $\omega$-converges to $\tilde{g}$.  Thus $[\tilde{g}] \in
  \Omega(\overline{\M})$.

  Moreover, suppose $\tilde{g} \in \U^0$ for some quasi-amenable
  subset $\U \subset \M$.  Then for any sequence $\{g_l\}$ in $\U$
  that $L^2$-converges to $\tilde{g}$, $\{g_l\}$ is $d$-Cauchy and
  there exists a subsequence $\{g_k\}$ that also $\omega$-converges to
  $\tilde{g}$.
\end{prop}
\begin{proof}
  It is clear that for every bounded
  representative $\tilde{g} \in [\tilde{g}]$, we can find a
  quasi-amenable subset $\U \subset \M$ such that $\tilde{g} \in
  \U^0$.
  Thus, there
  exists a sequence $\{ g_l \}$ that $L^2$-converges to $\tilde{g}$.
  It is $d$-Cauchy by Theorem \ref{thm:8}.  We wish to show that it
  contains a subsequence $\{ g_k \}$ that also $\omega$-converges to
  $\tilde{g}$, so we still need to verify properties
  (\ref{item:5})--(\ref{item:7}) of Definition \ref{dfn:13}.

  By passing to a subsequence, we can assume that property
  (\ref{item:7}) is satisfied for $\{g_l\}$.  Property (\ref{item:6})
  is verified in the same way as in the proof of Lemma \ref{lem:43}.
  That is, $L^2$-convergence of $\{ g_l \}$ implies that there exists
  a subsequence $\{ g_k \}$ of $\{ g_l \}$ that converges to
  $\tilde{g}$ a.e.  Finally, a.e.-convergence of $\{g_k\}$ to
  $\tilde{g}$ and continuity of the determinant function imply that
  property \eqref{item:5} holds.
\end{proof}

Thus, like we did for more restricted types of metrics before, this
proposition allows us to cease to distinguish between bounded
semimetrics and sequences $\omega$-converging to them.

\subsection{Unbounded metrics and the proof of the main
  result}\label{sec:finite-volume-metr}

Up to this point, we have an injection $\Omega : \overline{\M}
\rightarrow \Mfhat$, and we have determined that the image
$\Omega(\overline{\M})$ contains all equivalence classes containing
bounded semimetrics.  In this section, we prove that $\Omega$ is
surjective.

\begin{thm}\label{thm:12}
  Let any $[\tilde{g}] \in \Mfhat$ be given.  Then there exists a
  sequence $\{ g_k \}$ in $\M$ such that
  \begin{equation*}
    g_k \overarrow{\omega} [\tilde{g}].
  \end{equation*}
\end{thm}
\begin{proof}
  In view of Proposition \ref{prop:27}, it remains only to prove this
  for the equivalence class of a measurable, unbounded semimetric
  $\tilde{g} \in \Mf$.

  Given any element $\hat{g} \in \M_f$, we can define $\exp_{\hat{g}}$
  on tensors of the form $\sigma \hat{g}$, where $\sigma$ is any
  function, purely algebraically.  We simply set
  \begin{equation}\label{eq:60}
    \exp_{\hat{g}}(\sigma \hat{g}) := \left( 1 + \frac{n}{4} \sigma
    \right)^{4/n} \hat{g},
  \end{equation}
  so that the expression coincides with the usual one if $\hat{g} \in
  \M$ and $\sigma \in C^\infty(M)$ with $\sigma > -\frac{4}{n}$
  (cf. Theorem \ref{prop:5}).  If $\sigma$ is additionally measurable,
  then $\exp_{\hat{g}}(\sigma \hat{g})$ will also be measurable.

  Now, let $\tilde{g} \in \Mf$.  Then we can find a measurable,
  positive function $\xi$ on $M$ such that $g_0 := \xi \tilde{g}$ is a
  bounded semimetric.  A simple calculation using the finite volume of
  $\tilde{g}$ shows that $\rho := \xi^{-1} \in L^{n/2}(M, g_0)$.

  Define the map $\psi$ by $\psi(\sigma) := \exp_{g_0}(\sigma
  g_0)$, and let
  \begin{equation}\label{eq:140}
    \lambda := \frac{4}{n} \left( \rho^{n/4} - 1 \right).
  \end{equation}
  Then clearly $\psi(\lambda) = \rho g_0 = \tilde{g}$.  Moreover, we
  claim that $\lambda \in L^2(M, g_0)$ and hence, by Corollary
  \ref{cor:12}, we can find a sequence $\{\lambda_k\}$ of smooth
  functions that converge in $L^2(M, g_0)$ to $\lambda$ .  That
  $\lambda \in L^2(M, g_0)$ follows from two facts.  First, $\rho \in
  L^{n/2}(M, g_0)$, implying that $\rho^{n/4} \in L^2(M, g_0)$.
  Second, finite volume of $g_0$ implies that the constant function $1
  \in L^2(M, g_0)$ as well.

  Since $\lambda_k \rightarrow \lambda$ in $L^2(M, g_0)$, as in the
  proof of Lemma \ref{lem:43} we can pass to a subsequence and assume
  that $\lambda_k \rightarrow \lambda$ pointwise a.e., where we note
  that here, ``almost everywhere'' means with respect to $\mu_{g_0}$.
  With respect to the fixed, smooth, strictly positive volume form
  $\mu_g$, this actually means that $\lambda_k(x) \rightarrow
  \lambda(x)$ for a.e.~$x \in M \setminus X_{g_0}$, since $X_{g_0}$ is
  a nullset with respect to $\mu_{g_0}$.  Note also that $X_{g_0} =
  X_{\tilde{g}}$, since we assumed that the function $\xi$ is
  positive.  Therefore $\lambda_k(x) \rightarrow \lambda(x)$ for
  a.e.~$x \in M \setminus X_{\tilde{g}}$.

  Furthermore, since from \eqref{eq:140} and positivity of $\xi$ it is
  clear that $\lambda > - \frac{4}{n}$, we can choose the sequence $\{
  \lambda_k \}$ such that $\lambda_k > - \frac{4}{n}$ for all $k \in
  \N$.  This implies, in particular, that $X_{\psi(\lambda_k)} =
  X_{g_0} = X_{\tilde{g}}$, which is easily seen from \eqref{eq:60}.
  
  We make one last assumption on the sequence $\{ \lambda_k \}$.
  Namely, by passing to a subsequence, we can assume that
  \begin{equation}\label{eq:141}
    \sum_{k = 1}^\infty \norm{\lambda_{k+1} - \lambda_k}_{g_0} < \infty.
  \end{equation}
  
        Now, we claim that
  \begin{equation}\label{eq:51}
    d( \psi(\sigma), \psi(\tau) ) \leq \sqrt{n} \norm{\tau - \sigma}_{g_0}
  \end{equation}
  for all $\sigma, \tau \in C^\infty (M)$
  with $\sigma, \tau > -\frac{4}{n}$.  We delay the proof of this
  statement to Lemma \ref{lem:33} below and first finish the proof of
  the theorem.
  
  We wish to construct a sequence that $\omega$-converges to
  $\tilde{g}$ using the sequence $\{ \psi(\lambda_k) \}$.  We can't
  use $\{ \psi(\lambda_k) \}$ directly, since it is a sequence in
  $\Omega(\overline{\M})$, not $\M$ itself.  So we first verify the
  properties of $\omega$-convergence for $\{ \psi(\lambda_k) \}$ and
  then construct a sequence in $\M$ that approximates $\{
  \psi(\lambda_k) \}$ well enough that it still satisfies all the
  conditions for $\omega$-convergence.
  
  Since the sequence $\{\lambda_k\}$ is convergent in $L^2(M, g_0)$,
  it is also Cauchy in $L^2(M, g_0)$.  Using the inequality
  \eqref{eq:51}, it is then immediate that $\{\psi(\lambda_k)\}$ is a
  Cauchy sequence in $(\Omega(\overline{\M}), d)$.  This verifies
  property (\ref{item:4}) of $\omega$-convergence (cf.~Definition
  \ref{dfn:13}).
  
  We next verify property (\ref{item:6}).  Note that $X_{\tilde{g}}
  \subseteq D_{\{ \psi(\lambda_k) \}}$, since we have already shown
  that $X_{\psi(\lambda_k)} = X_{\tilde{g}}$.  (Keep in mind here the
  subtle point that $X_{\psi(\lambda_k)}$ is the deflated set of the
  \emph{individual} semimetric $\psi(\lambda_k)$, while $D_{\{
    \psi(\lambda_k) \}}$ is the deflated set of the \emph{sequence}
  $\{\psi(\lambda_k)\}$.  Refer to Definitions \ref{dfn:23} and
  \ref{dfn:25} for details.)  The inclusion implies that
  \begin{equation*}
    M \setminus D_{\{ \psi(\lambda_k) \}} \subseteq M \setminus X_{\tilde{g}},
  \end{equation*}
  so it suffices to show that $\psi(\lambda_k)(x) \rightarrow
  \tilde{g}(x)$ for a.e.~$x \in M \setminus X_{\tilde{g}}$.  But this
  is clear from
  the fact, proved above,
  that $\lambda_k(x) \rightarrow \lambda(x)$ for a.e.~$x \in M
  \setminus X_{\tilde{g}}$.
  
  To verify property (\ref{item:5}), we claim that $D_{\{
    \psi(\lambda_k) \}} = X_{\tilde{g}}$, up to a nullset.  In the
  previous paragraph, we already showed that $X_{\tilde{g}} \subseteq
  D_{\{\psi(\lambda_k)\}}$. Furthermore, for a.e.~$x \in M \setminus
  X_{\tilde{g}}$, $\{ \psi(\lambda_k)(x) \}$ converges to
  $\tilde{g}(x)$, which is positive definite, so for a.e.~$x \in M
  \setminus X_{\tilde{g}}$, $\lim \det \psi(\lambda_k) > 0$.  This
  immediately implies that $D_{\{\psi(\lambda_k)\}} \subseteq
  X_{\tilde{g}} $, up to a nullset.
  
  The last property to verify is (\ref{item:7}).  But this is
  immediate from \eqref{eq:141} and \eqref{eq:51}.

  So we have shown that $\{\psi(\lambda_k)\}$ satisfies the properties
  of $\omega$-convergence, save that it is a sequence of measurable
  semimetrics, rather than a sequence of smooth metrics as required.
  To get a sequence in $\M$ that $\omega$-converges to $\tilde{g}$,
  recall that each of the functions $\lambda_k$ is smooth and
  therefore bounded, and also that $g_0$ is a bounded, measurable
  semimetric.  Therefore, for each fixed $k \in \N$, $\psi(\lambda_k)$
  is also a bounded, measurable semimetric, and so by Proposition
  \ref{prop:27} we can find a sequence $\{ g^k_l \}$ in $\M$ that
  $\omega$-converges to $\psi(\lambda_k)$ for $l \rightarrow \infty$.
  By a standard diagonal argument, it is then possible to select $l_k
  \in \N$ for each $k \in \N$ such that the sequence $\{ g^k_{l_k} \}$
  $\omega$-converges to $\tilde{g}$ for $k \rightarrow \infty$.  Thus
  we have found the desired sequence.

  It still remains to prove \eqref{eq:51}.  The following two lemmas
  do this and thus complete the proof of the theorem.
\end{proof}

\begin{lem}\label{lem:2}
  Let $\tilde{g} \in \M$.  If $\sigma, \tau \in C^\infty(M)$ satisfy
  $\sigma, \tau > - \frac{4}{n}$, then
  \begin{equation*}
    d(\exp_{\tilde{g}}(\sigma \tilde{g}), \exp_{\tilde{g}}(\tau
    \tilde{g})) \leq \sqrt{n} \norm{\tau - \sigma}_{\tilde{g}}.
  \end{equation*}
\end{lem}
\begin{proof}
  Let $\hat{g} := \exp_{\tilde{g}}(\sigma \tilde{g})$.  We first note
  that $\pos \cdot \tilde{g} = \pos \cdot \hat{g}$.  Therefore, by Proposition
  \ref{prop:5}, there is a neighborhood $V \in C^\infty(M)$ such that
  $\exp_{\hat{g}} : V \cdot \hat{g} \rightarrow \pos \cdot
  \hat{g}$ is a diffeomorphism.

  Now
  \begin{equation}\label{eq:2}
    d(\exp_{\tilde{g}}(\sigma \tilde{g}), \exp_{\tilde{g}}(\tau
    \tilde{g})) \leq \norm{\exp_{\hat{g}}^{-1} \exp_{\tilde{g}}(\tau
      \tilde{g})}_{\tilde{g}},
  \end{equation}
  since the right-hand side is the length of a radial geodesic
  emanating from $\exp_{\tilde{g}}(\sigma \tilde{g})$ and ending at
  $\exp_{\tilde{g}}(\tau \tilde{g})$.
  Showing that the right-hand side of \eqref{eq:2} is equal to
  $\sqrt{n} \norm{\tau - \sigma}$ is a straightforward computation
  using \eqref{eq:3}.
\end{proof}

\begin{lem}\label{lem:33}
  Let $g_0$ and $\psi$ be as in the proof of Theorem \ref{thm:12}.  If
  $\sigma, \tau \in C^\infty(M)$ satisfy $\sigma, \tau > -
  \frac{4}{n}$, then
  \begin{equation*}
    d(\psi(\sigma), \psi(\tau)) \leq \sqrt{n} \norm{\tau - \sigma}_{g_0}.
  \end{equation*}
\end{lem}
\begin{proof}
  Since $g_0$ is bounded, we can find a quasi-amenable subset $\U$
  such that $g_0 \in \U^0$, i.e., such that $g_0$ belongs to the
  completion of $\U$ with respect to $\normdot_g$.  Using
  Proposition \ref{prop:27}, choose a sequence $\{ g_k \}$ in $\U$
  that both $L^2$- and $\omega$-converges to $g_0$.  For each $k \in
  \N$, define a map $\psi_k$ by $\psi_k(\kappa) := \exp_{g_k}(\kappa
  g_k)$.

  By the triangle inequality, we have
  \begin{equation}\label{eq:59}
    d(\psi(\sigma), \psi(\tau)) \leq d(\psi(\sigma), \psi_k(\sigma)) +
    d(\psi_k(\sigma), \psi_k(\tau)) + d(\psi_k(\tau), \psi(\tau))
  \end{equation}
  for each $k$.  But since $g_k \in \M$, Lemma \ref{lem:2} applies to
  give
  \begin{equation}\label{eq:58}
    d(\psi_k(\sigma), \psi_k(\tau)) \leq \sqrt{n} \norm{\tau - \sigma}_{g_k} \overarrow{k \rightarrow \infty} \sqrt{n} \norm{\tau - \sigma}_{g_0},
  \end{equation}
  where the convergence follows from Lemma \ref{lem:54}.
  By \eqref{eq:59} and \eqref{eq:58}, if we can
  show that
  \begin{equation}\label{eq:4}
    d(\psi(\sigma), \psi_k(\sigma)) \rightarrow 0 \quad
    \textnormal{and} \quad d(\psi_k(\tau), \psi(\tau)) \rightarrow 0,
  \end{equation}
  then we are finished.
  But it is not hard to show that $\psi_k(\sigma)$ $L^2$-converges to
  $\psi(\sigma)$, which then implies \eqref{eq:4} by Proposition
  \ref{prop:27}.
\end{proof}

From the results of Section \ref{cha:almost-everywh-conv}, we already
know that the map $\Omega : \overline{\M} \rightarrow \Mfhat$ is an
injection.  Theorem \ref{thm:12} now states that this map is a
surjection as well.  Thus, we have already proved the main result of
this paper, which we state again here in full detail.

\begin{thm}\label{thm:41}
  There is a natural identification of $\overline{\M}$, the completion
  of $\M$ with respect to the $L^2$ metric, with $\Mfhat$, the set of
  measurable semimetrics with finite volume on $M$ modulo the
  equivalence given in Definition \ref{dfn:7}.

  This identification is given by a bijection $\Omega : \overline{\M}
  \rightarrow \Mfhat$, where we map an equivalence class $[\{g_k\}]$
  of $d$-Cauchy sequences to the unique element of $\Mfhat$ that all
  of its members $\omega$-subconverge to.  This map is an isometry if
  we give $\Mfhat$ the metric $\bar{d}$ defined by
  \begin{equation*}
    \bar{d}([g_0], [g_1]) := \lim_{k \rightarrow \infty} d(g^0_k, g^1_k)
  \end{equation*}
  where $\{g^0_k\}$ and $\{g^1_k\}$ are any sequences in $\M$
  $\omega$-subconverging to $[g_0]$ and $[g_1]$, respectively.
\end{thm}

Here, we briefly note what geometric notions are well-defined for
elements of $\Mfhat$.  Given an equivalence class $[\tilde{g}] \in
\Mfhat$, the metric space structure of different representatives may
differ---e.g., if $M = T^2$, the torus, then the equivalence class of
the zero metric also contains a geometric circle, where only one
dimension has collapsed.  On the other hand, since representatives of
a given equivalence class in $\Mfhat$ all have equal induced measures,
things like $L^p$ spaces of functions are well-defined for an
equivalence class, as they are the same across all representatives.
But even more is true---$C^k$ and $H^s$ spaces of sections of fiber
bundles can be defined Therefore, an equivalence class doesn't induce
a well-defined scalar product on a vector bundle at any individual
point, but the integral of the scalar product does not depend on the
chosen representative.

To end this section, we remark that one might hope that Theorem
\ref{thm:41} would give some information on the completion of the
space $\M / \D$ of \emph{Riemannian structures} with respect to the
distance that the $L^2$ metric induces on it.  (Here, $\D$ is the
group of orientation-preserving diffeomorphisms of $M$ acting by
pull-back.)  $\M / \D$ is the moduli space of Riemannian metrics, and
hence is of great intrinsic interest to geometers.  The problem here
is that the proof of Theorem \ref{thm:41} does not indicate which
degenerations of Cauchy sequences of metrics arise from ``vertical''
degenerations---that is, sequences $\{\varphi_n^* \tilde{g}\}$, where
$\{\varphi_n\} \subset \D$ is a degenerating sequence of
diffeomorphisms---and ``horizontal'' degenerations---that is,
sequences of metrics that can be connected by horizontal paths.  (See
\cite[\S 3]{freed89:_basic_geomet_of_manif_of} for a discussion of
horizontal and vertical paths on $\M$.)  Of course, only horizontal
degenerations are relevant for the quotient.  So there is some work
remaining to do in order to understand the completion of $\M / \D$.
We hope to investigate these questions in a future paper.

\section{Application to Teichmüller Theory}\label{cha:appl-teichm-space}

In this section, we describe an application of our main theorem to the
theory of Teichmüller space.  Teichmüller space was historically
defined in the context of complex manifolds, but the work of Fischer
and Tromba translates this original picture into the context of
Riemannian geometry, using the manifold of metrics \cite{tromba-teichmueller}.
We outline this construction of Teichmüller space in the first
subsection, then define the much-studied Weil-Petersson metric.  In
the second subsection, we prove a result on the completions of a class
of metrics that generalize the Weil-Petersson metric.

\subsection{The Weil-Petersson Metric on Teichmüller
  Space}\label{sec:teichmuller-space}

\begin{cvt}\label{cvt:2}
  For the remainder of the paper, let our base manifold $M$ be a
  smooth, closed, oriented, two-dimensional manifold of genus $p \geq
  2$.
\end{cvt}

\begin{cvt}
  In this chapter, we abandon Convention \ref{cvt:3}.  That is, when
  we write $g$ for a metric in $\M$, we no longer assume that this is
  fixed, but allow $g$ to vary arbitrarily.
\end{cvt}

We have already noted that the group $\pos$ acts on $\M$ by pointwise
multiplication, and it turns out that the quotient $\M / \pos$ is a
smooth Fréchet manifold.  Let $\D$ be the Fréchet Lie group of
orientation-preserving diffeomorphisms of $M$, and let $\DO \subset
\D$ be the subgroup of diffeomorphisms homotopic to the identity.
Then both $\D$ and $\DO$ act on $\M$ and $\M / \pos$ by pull-back.
Let $\T$ denote the Teichmüller space of $M$, $\mathcal{R}$ the
Riemann moduli space of $M$, and $MCG := \D / \DO$ the \emph{mapping
  class group} of $M$.  Then there are identifications
\begin{equation*}
  \T \cong \grpquot{(\M / \pos)}{\DO}, \qquad \mathcal{R} \cong
  \grpquot{\T}{MCG} \cong \grpquot{(\M / \pos) }{\D},
\end{equation*}
where the first identification is a diffeomorphism.  Note that
Teichmüller space finite-dimensional.

By the Poincaré uniformization theorem, there exists a unique
hyperbolic metric (one with scalar curvature $-1$) in each conformal
class $[g] \in \M / \pos$.  Furthermore, one can show that the subset
$\Mhyp \subset \M$ of hyperbolic metrics is a smooth Fréchet
submanifold.  Therefore, $\Mhyp$ is the image of a smooth section of
the principal $\pos$-bundle $\M \rightarrow \M / \pos$.  It is easy to
see that $\Mhyp$ is $\D$-invariant, and therefore
\begin{equation*}
  \T \cong \grpquot{\Mhyp}{\DO}, \qquad \mathcal{R} \cong \grpquot{\Mhyp}{\D},
\end{equation*}
where the first identification is again a diffeomorphism.  We denote
by $\pi : \Mhyp \rightarrow \Mhyp / \DO$ the projection.

It is not hard to see that the $L^2$ metric $(\cdot, \cdot)$ on $\M$
is $\D$-invariant, so it descends to a $MCG$-invariant Riemannian
metric, also denoted $(\cdot, \cdot)$, on the quotient $\Mhyp / \DO$.
This metric is called the \emph{Weil-Petersson metric}.  (It differs
from the usual definition of the Weil-Petersson metric by a constant
scalar factor; cf. \cite[\S 2.6]{tromba-teichmueller}.)  With these
definitions, we see that $(\Mhyp, (\cdot, \cdot)) \rightarrow (\Mhyp /
\DO, (\cdot, \cdot))$ is a weak Riemannian principal
$\DO$-bundle---that is, at each point $g \in \Mhyp$, the differential
$D \pi(g)$ of the projection is an isometry when restricted to the
horizontal tangent space at $g$.

\subsection{Generalized Weil-Petersson Metrics}\label{sec:metrics-arising-from}

We now wish to generalize the construction of the Weil-Petersson
metric by selecting a different section of $\M \rightarrow \M / \pos$.
In fact, we will simultaneously consider all smooth sections
$\mathcal{N}$ with the property that they are $\D$-invariant, which we
require so that we still have diffeomorphisms $\T \cong \Ncal / \DO$
and $\riem \cong \Ncal / \D$.  This idea is directly inspired by
\cite{hj-riemannian} and
\cite{habermann98:_metric_rieman_surfac_and_geomet}, though our
metrics on Teichmüller space differ from theirs.

\begin{dfn}\label{dfn:5}
  We call a smooth, $\D$-invariant section of $\M \rightarrow \M /
  \pos$ a \emph{modular section}.  Given a modular section $\Ncal
  \subset \M$, we call the quotients $\Ncal / \DO$ and $\Ncal / \D$
  the \emph{$\Ncal$-model of Teichmüller space} and the
  \emph{$\Ncal$-model of moduli space}, respectively.
\end{dfn}

For the remainder of the talk, let $\Ncal$ be an arbitrary modular
section.  It is not hard to see that $\Ncal \cong \Mhyp$ via a
$\D$-equivariant diffeomorphism, so in fact we do have the desired
diffeomorphisms $\T \cong \Ncal / \DO$ and $\riem \cong \Ncal / \D$.

Modular sections other than $\Mhyp$ of course exist---for example,
there are the Bergman and Arakelov metrics (see \cite[\S
1]{habermann98:_metric_rieman_surfac_and_geomet} for details).  We
briefly describe the Bergman metric on a Riemann surface here.  Recall
that conformal structures (elements of $\M / \pos$) are in one-to-one
correspondence with complex structures on the surface, so we can work
with these instead.  Let $c$ be a complex structure on $M$.  Then the
space of holomorphic one-forms on $(M, c)$ has complex dimension $p$,
the genus of $M$ \cite[Prop.~III.2.7]{farkas92:_rieman_surfac}.  Let
$\theta_1, \dots, \theta_p$ be an $L^2$-orthonormal basis of this
space.
That is,
\begin{equation*}
  \frac{i}{2} \int_M \theta_j \wedge \overline{\theta_k} = \delta_{jk}.
\end{equation*}
The Bergman metric is defined by
\begin{equation*}
  g_B := \frac{1}{p} \sum_{i=1}^p \theta_i \bar{\theta}_j.
\end{equation*}
It is clear from this construction that the set of all Bergman metrics
is indeed a modular section.

As in the case of the section $\Mhyp$, the $L^2$ metric on $\M$
projects
to an $MCG$-invariant metric on $\Ncal / \DO$.  We call this metric
the \emph{generalized Weil-Petersson metric} on the $\Ncal$-model of
Teichmüller space.  As in the case of the bundle $\Mhyp \rightarrow
\Mhypd$, these metrics turn the bundle $\Ncal \rightarrow \Ncal / \DO$
into a weak Riemannian principal $\DO$-bundle.

\begin{thm}\label{thm:42}
  For any $C^1$ path $\gamma : [0,1] \rightarrow \Ncal / \DO$ and any
  $g \in \pi_\Ncal^{-1}(\gamma(0))$, there exists a unique horizontal
  lift $\tilde{\gamma} : [0,1] \rightarrow \Ncal$ with
  $\tilde{\gamma}(0) = g$.

  Furthermore, $L(\tilde{\gamma}) = L(\gamma)$ and $\tilde{\gamma}$
  has minimal length among the class of curves whose image projects to
  $\gamma$ under $\pi_\Ncal$.
\end{thm}
\begin{proof}
  The existence of horizontal lifts is not usually guaranteed on
  Fréchet manifolds, but since $\Ncal / \DO$ is finite-dimensional,
  the horizontal tangent space of $\Ncal$ is finite-dimensional at
  each point.  Therefore, integral curves of horizontal vector fields
  exist (cf.~\cite[Thm.~7.2 and
  Dfn.~5.6ff]{omori97:_infin_dimen_lie_group}).

  An alternative proof, one which does not rely on the existence
  theory of solutions to ODEs in Fréchet spaces, is given in
  \cite[Thm.~6.16]{clarked):_compl_of_manif_of_rieman}.

  Minimality of $\tilde{g}$ can be easily shown using the fact that
  $\Ncal \rightarrow \Ncal / \DO$ is a weak Riemannian principal
  bundle.
\end{proof}

The next theorem applies the paper's main to the completion of $\Ncal /
\DO$ with respect to a generalized Weil-Petersson metric.  In the
following, we denote the distance function of $(\Ncal, (\cdot,
\cdot))$ by $d_\Ncal$.

\begin{thm}\label{thm:43}
  Let $\{[g_k]\}$ be a Cauchy sequence in the $\Ncal$-model of
  Teichmüller space, $\Ncal / \DO$, with respect to the generalized
  Weil-Petersson metric.  Then there exist representatives
  $\tilde{g}_k \in [g_k]$ and an element $[g_\infty] \in \Mfhat$ such
  that $\{\tilde{g}_k\}$ is a $d_\Ncal$-Cauchy sequence that
  $\omega$-subconverges to $[g_\infty]$.

  Furthermore, if $\{[g^0_k]\}$ and $\{[g^1_k]\}$ are equivalent
  Cauchy sequences in $\Ncal / \DO$, then there exist representatives
  $\tilde{g}^0_k \in [g^0_k]$ and $\tilde{g}^1_k \in [g^1_k]$, as well
  as an element $[g_\infty] \in \Mfhat$, such that $\{\tilde{g}^0_k\}$
  and $\{\tilde{g}^1_k\}$ are $d_\Ncal$-Cauchy sequences that both
  $\omega$-subconverge to $[g_\infty]$.

  Finally, if $\{[g^0_k]\}$ and $\{[g^1_k]\}$ are inequivalent Cauchy
  sequences in $\Ncal / \DO$, then there exists no choice of
  representatives $\tilde{g}^0_k \in [g^0_k]$ and $\tilde{g}^1_k \in
  [g^1_k]$ such that $\{\tilde{g}^0_k\}$ and $\{\tilde{g}^1_k\}$
  $\omega$-subconverge to the same element of $\Mfhat$.
\end{thm}
\begin{proof}
  The first claim would follow from Theorem \ref{thm:41} if we could
  show that there are representatives $\tilde{g}_k \in [g_k]$ such that
  $\{\tilde{g}_k\}$ is a $d_\Ncal$-Cauchy sequence, since this implies that it
  is also a $d$-Cauchy sequence.  So this is what we will show.

  Let's denote the distance function induced by the generalized
  Weil-Petersson metric on $\Ncal / \DO$ by $\delta$.  For each $k \in
  \N$, let $\gamma_k : [0,1] \rightarrow \Ncal / \DO$ be any path from
  $[g_k]$ to $[g_{k+1}]$ such that
  \begin{equation*}
    L(\gamma_k) \leq 2 \delta([g_k], [g_{k+1}]).
  \end{equation*}

  For any $\tilde{g}_1 \in \pi_\Ncal^{-1}([g_1])$, let $\tilde{\gamma}_1$ be
  the horizontal lift of $\gamma_1$ to $\Ncal$ with
  $\tilde{\gamma}_1(0) = \tilde{g}_1$ which is guaranteed by Theorem
  \ref{thm:42}.  Then clearly $\tilde{g}_2 := \gamma_2(1) \in
  \pi_\Ncal^{-1}([g_2])$.  Furthermore,
  \begin{equation*}
    d_\Ncal(\tilde{g}_1, \tilde{g}_2)
    \leq L(\tilde{\gamma}_1) = L(\gamma_1) \leq 2 \delta([g_1], [g_2]).
  \end{equation*}

  We repeat this process, i.e., let $\tilde{\gamma}_2$ be the unique
  horizontal lift of $\gamma_2$ with $\tilde{\gamma}_2(0) = \tilde{g}_2$, and
  set $\tilde{g}_3 := \tilde{\gamma}_2(1)$, etc.
  In this way, we get a sequence of representatives $\tilde{g}_k \in [g_k]$
  such that for each $k \in \N$,
  \begin{equation*}
    d_\Ncal(\tilde{g}_k, \tilde{g}_{k+1}) \leq 2 \delta([g_k], [g_{k+1}]).
  \end{equation*}
  Thus, since $\{[g_k]\}$ is a Cauchy sequence, $\{\tilde{g}_k\}$ is a
  $d_\Ncal$-Cauchy sequence, as was to be shown.

  The proof of the second statement is similar.
  
  To prove the last statement, note that since $\{[g^0_k]\}$ and
  $\{[g^1_k]\}$ are inequivalent, we have
  \begin{equation*}
    \lim_{k \rightarrow \infty} \delta([g^0_k], [g^1_k]) > 0.
  \end{equation*}
  Thus by Theorem \ref{thm:42}, no matter what representatives
  $\tilde{g}^0_k \in [g^0_k]$ and $\tilde{g}^1_k \in [g^1_k]$ we
  choose,
  \begin{equation*}
    \lim_{k \rightarrow \infty} d_\Ncal(\tilde{g}^0_k, \tilde{g}^1_k)
    \geq \epsilon > 0.
  \end{equation*}
  So Theorem \ref{thm:41} implies the statement immediately.
\end{proof}

\begin{rmk}\label{rmk:3}
  We note that the choice of $[g_\infty]$ in Theorem \ref{thm:43} is
  in some sense unique up to $\DO$-equivalence.  Namely, say in the
  proof of the first statement we were to have chosen a different
  sequence of paths $\gamma^0_k$ connecting $[g_k]$ to $[g_{k+1}]$.
  Since Teichmüller space is known to be contractible, the horizontal
  lifts $\tilde{\gamma}_k$ and $\tilde{\gamma}^0_k$ have the same
  endpoints $\tilde{g}_{k+1} \in [g_{k+1}]$.  Thus, they determine the
  same $d_{\Ncal}$-Cauchy sequence, implying the only choice we made
  in the proof that matters was that of $\tilde{g}_1 \in [g_1]$.
\end{rmk}

Theorem \ref{thm:43} generalizes what is known about the completion of the
Weil-Petersson metric \cite{masur-extension}, which is completed by
adding in certain cusped hyperbolic surfaces---which in particular can
be viewed as elements of $\Mfhat$.  However, they are only degenerate
along a set of disjoint simple closed geodesics---connecting nicely
with the complex-analytic viewpoint of degeneration---whereas elements
of $\Mfhat$ can be degenerate over an arbitrary subset of $M$.  With
more investigation and perhaps appropriate conditions on the section
$\Ncal$, we expect that the statement can be considerably
improved.

Despite the shortcomings of the above result, we see it as quite
useful, as it gives relatively strong information about a new class of
metrics on Teichmüller space---namely, that their completions can
consist only of finite-volume metrics.  Furthermore, it illustrates
the potential for applications of our main theorem and provides a
starting point for further investigations.

\section*{Appendix.  Relations between Various Notions of Convergence
  and Cauchy Sequences}

In the following chart, we illustrate the relationships between the
different notions of Cauchy and convergent sequences on $\M$.  We let
$\{g_k\}$ be a sequence in $\M$ and $\tilde{g} \in \Mf$.  A double
arrow (``$\Longrightarrow$'') between two statements means that the
one implies the other.  A single arrow (``$\longrightarrow$'') means
that one statement implies the other, assuming the condition that is
listed below the chart.

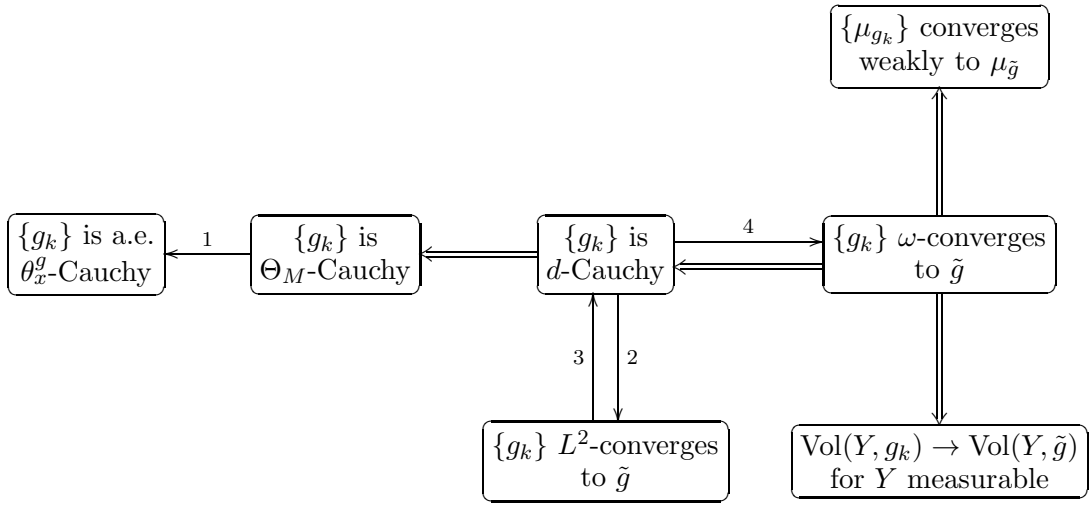
\begin{equation*}
  \xymatrix@R=48pt{
    & & & *+[F-:<3pt>]{\txt{$\{\mu_{g_k}\}$ converges\\weakly to $\mu_{\tilde{g}}$}} \\
    *+[F-:<3pt>]{\txt{$\{g_k\}$ is a.e.\\$\theta^g_x$-Cauchy}} &
    *+[F-:<3pt>]{\txt{$\{g_k\}$ is\\$\Theta_M$-Cauchy}} \ar[l]_-1 &
    *+[F-:<3pt>]{\txt{$\{g_k\}$ is\\$d$-Cauchy}} \ar@{=>}[l]
    \ar@<1ex>[d]^-2 \ar@<1ex>[r]^-4
    & *+[F-:<3pt>]{\txt{$\{g_k\}$ $\omega$-converges\\to $\tilde{g}$}}
    \ar@<1ex>@{=>}[l] \ar@{=>}[d] \ar@{=>}[u] \\
    *\txt{\phantom{some space here}} & & *+[F-:<3pt>]{\txt{$\{g_k\}$
        $L^2$-converges\\to $\tilde{g}$}} \ar@<1ex>[u]^-3
    & *+[F-:<3pt>]{\txt{$\Vol(Y,g_k) \rightarrow
        \Vol(Y,\tilde{g})$\\for $Y$ measurable}}
  }
\end{equation*}

\begin{enumerate}
\item After passing to a subsequence
\item If there exists an amenable subset $\U$ such that $\{g_k\}
  \subset \U$, then there exists some $\tilde{g} \in \U^0$ such that
  the implication holds
\item If there exists a quasi-amenable subset $\U$ such that $\{g_k\}
  \subset \U$
\item After passing to a subsequence, there exists some $\tilde{g}
  \in \Mf$ such that the implication holds
\end{enumerate}

\bibliographystyle{hamsplain}
\bibliography{main}

\end{document}